\documentclass[12pt]{article}
\usepackage{geometry}
\usepackage{amssymb}
 \usepackage{xypic}
\usepackage[all,2cell]{xy}
\usepackage{graphicx}
\usepackage{amsmath}
\usepackage{fullpage}
\usepackage{amsthm}
\usepackage{mathtools}
\usepackage{leftidx}
\def\cP{\mathcal P}
\usepackage{xcolor, soul}
\sethlcolor{yellow}
\allowdisplaybreaks
\usepackage{mathrsfs}
\usepackage{calligra}
\usepackage[all,2cell]{xy}
\UseTwocells

\newtheorem{theorem}{Theorem}[section]
\newtheorem{corollary}[theorem]{Corollary}
\newtheorem{lemma}[theorem]{Lemma}
\newtheorem{proposition}[theorem]{Proposition}

\newtheorem{definition}[theorem]{Definition}

\newtheorem{remark}[theorem]{Remark}

\newtheorem{notation}[theorem]{Notation}

\usepackage{hyperref}
\begin{document}

\title{An answer to the
	Whitehead’s asphericity question}
\author{Elton Pasku \\
Universiteti i Tiran\"es \\
Fakulteti i Shkencave Natyrore \\
Departamenti i Matematik\"es\\ 
Tiran\"e, Albania \\
elton.pasku@fshn.edu.al}

\date{}

\maketitle

\begin{abstract}

The Whitehead asphericity problem, regarded as a problem of combinatorial group theory, asks whether any subpresentation of an aspherical group presentation is also aspherical. We give a positive answer to this question by proving that if $\cP=(\mathbf{x}, \mathbf{r})$ is an aspherical presentation of the trivial group, and $r_{0} \in \mathbf{r}$ a fixed relation, then $\cP_{1}=(\mathbf{x}, \mathbf{r}_{1})$ is aspherical where $\mathbf{r}_{1}=\mathbf{r} \setminus \{r_{0}\}$. 
\end{abstract}

\maketitle

\section{Introduction} \label{intr}

A 2-dimensional CW-complex $K$ is called {\it aspherical} if $\pi_2(K)=0$. The Whitehead asphericity problem, raised as a question in \cite{JHCW}, asks whether any subcomplex of an aspherical 2-complex is also aspherical. The question can be formulated in group theoretic terms since every group presentation $\mathcal{P}$ has a geometric realisation as a 2-dimensional CW-complex $K(\mathcal{P})$ and so $\mathcal{P}$ is called aspherical if $K(\mathcal{P})$ is aspherical. A useful review of this question is in \cite{Rosebrock}.

The purpose of the present paper is to prove that if $\mathcal{P}=( \mathbf{x},\mathbf{r} )$ is an aspherical presentation of the trivial group and $r_{0} \in \mathbf{r}$ is a fixed relation, then the subpresentation $\mathcal{P}_{1}=( \mathbf{x},\mathbf{r}_{1} )$ where $\mathbf{r}_{1}=\mathbf{r} \setminus \{r_{0}\}$ is again aspherical. This in fact implies that Whitehead asphericity problem has always a positive answer since in Theorem 1 of \cite{I IJM} Ivanov proves that if the Whitehead asphericity problem is false, then there is an aspherical presentation $\cP=(\mathcal{A}, \mathcal{R}\cup \{z\})$ of the trivial group where the alphabet $\mathcal{A}$ is countable and $z \in \mathcal{A}$ such that $\cP_{1}=(\mathcal{A}, \mathcal{R})$ is not aspherical.

An immediate implication of our result and that of Bestvina-Brady \cite{BB} is that the conjecture of Eilenberg and Ganea \cite{EG} is false. This conjecture states that if a discrete group $G$ has cohomological dimension 2, then it has a 2-dimensional Eilenberg-MacLane space $K(G, 1)$.

There is a large corpus of results which are related to ours and is mostly contained in \cite{BP}, \cite{BH}, \cite{AGP}, \cite{CH82}, \cite{GR}, \cite{How79}, \cite{How81a}, \cite{How81b}, \cite{How82}, \cite{How83}, \cite{How84}, \cite{How85}, \cite{How98}, \cite{Hue82}, \cite{I IJM}, \cite{IL} and \cite{Stefan}.

In the first part of our paper we will make use of the review paper \cite{BH} of Brown and Huebschmann which contains several key results about aspherical group presentations one of which is proposition 14 that gives sufficient and necessary conditions under which a group presentation $\mathcal{P}=( \mathbf{x},\mathbf{r} )$ is aspherical. It turns out that the asphericity of $\mathcal{P}$ is encoded in the structure of the free crossed module $(\Sigma,\hat{F},\delta)$ that is associated to $\mathcal{P}$. We will explain in some details the definition of $(\Sigma,\hat{F},\delta)$. Let $Y=\hat{F} \times \mathbf{r}$ and $H=FG(Y)$ the free group on $Y$. An element of $H$ is thus a finite (possibly empty) sequence $((u_{1},r_{1})^{\varepsilon_{1}},\dots ,(u_{n},r_{n})^{\varepsilon_{n}})$ where $u_{i} \in \hat{F}$, $r_{i} \in \mathbf{r}$ and $\varepsilon_{i}=\pm 1$. There is a left action of $\hat{F}$ on $H$ defined (in exponential notations) by $^{v} ((u_{1},r_{1})^{\varepsilon_{1}},\dots ,(u_{n},r_{n})^{\varepsilon_{n}})=((vu_{1},r_{1})^{\varepsilon_{1}},\dots ,(vu_{n},r_{n})^{\varepsilon_{n}})$. Also we let $\theta:H \rightarrow \hat{F}$ be the group homomorphism defined on free generators by $\theta((u, r))=uru^{-1}$. Note that $\theta$ is an $\hat{F}$-morphism, which means that for every $v \in \hat{F}$ and every $((u_{1},r_{1})^{\varepsilon_{1}},\dots ,(u_{n},r_{n})^{\varepsilon_{n}}) \in H$, $\theta (^{v} ((u_{1},r_{1})^{\varepsilon_{1}},\dots ,(u_{n},r_{n})^{\varepsilon_{n}}))=v \theta (((u_{1},r_{1})^{\varepsilon_{1}},\dots ,(u_{n},r_{n})^{\varepsilon_{n}}))v^{-1}$.

Further, let $P$ be the normal subgroup of $H$ generated by Peiffer relations $((u,r)^{\varepsilon},(v,s)^{\delta})=(^{\theta((u,r)^{\varepsilon})}(v,s)^{\delta},(u,r)^{\varepsilon})$. We note here that to generate $P$ as a normal subgroup of $H$ it is enough to consider only all those relations above with $\varepsilon=\delta=1$. The quotient group $H/F$ will be denoted in this paper by $\Sigma$. 

There is an induced left action on the quotient group $\Sigma$ defined by 
$$^{v} (((u_{1},r_{1})^{\varepsilon_{1}},\dots (u_{n},r_{n})^{\varepsilon_{n}}) P)=((vu_{1},r_{1})^{\varepsilon_{1}},\dots ,(vu_{n},r_{n})^{\varepsilon_{n}})P.$$ 
Also it turns out that $\theta$ induces a homomorphism $\delta: \Sigma \rightarrow \hat{F}$ which satisfies the following conditions.
\begin{itemize}
\item[(i)] $\delta (^{v} a)=v \delta(a) v^{-1}$ for every $a \in \Sigma$ and $v \in \hat{F}$.
\item[(ii)] $ab= {^{\delta(a)}}b$ for every $a,b \in \Sigma$.
\end{itemize}
With these two conditions satisfied $\Sigma=(\Sigma,\hat{F},\delta)$ becomes a crossed module. Finally, we recall from \cite{BH} that $\text{Ker}\delta$ is denoted by $\pi$ and is called the module of identities for $\mathcal{P}$ since it is a left $\mathbb{Z}G$-module where $G$ is the group given by $\cP$. The left action is defined by the rule
$$g\cdot (((u_{1},r_{1})^{\varepsilon_{1}},\dots (u_{n},r_{n})^{\varepsilon_{n}}) P)=((wu_{1},r_{1})^{\varepsilon_{1}},\dots ,(wu_{n},r_{n})^{\varepsilon_{n}})P$$ 
whenever $w \in \hat{F}$ is a word representing $g$ and $((u_{1},r_{1})^{\varepsilon_{1}},\dots (u_{n},r_{n})^{\varepsilon_{n}}) P \in \pi$.

We state below Proposition 14 of \cite{BH}.
\begin{proposition} (Proposition 14 of \cite{BH})
Let $K(\mathcal{P})$ be the geometric realisation of a group presentation $\mathcal{P}=( \mathbf{x},\mathbf{r})$ and let $G$ be the group given by $\mathcal{P}$. The following are equivalent.
\begin{description}
\item [(i)] The 2-complex $K(\mathcal{P})$ is aspherical.
\item [(ii)] The module $\pi$ of identities for $\mathcal{P}$ is zero.
\item [(iii)] The relation module $\mathcal{N(P)}$ of $\mathcal{P}$ is a free left $\mathbb{Z}G$ module on the images of the relators $r \in \mathbf{r}$.
\item [(iv)] Any identity $Y$-sequence for $\mathcal{P}$ is Peiffer equivalent to the empty sequence.
\end{description}
\end{proposition}
The last condition is of a particular interest to us. First we recall from \cite{BH} that an $Y$-sequence is by definition a finite (possibly empty) string $d=((u_{1},r_{1})^{\varepsilon_{1}},...,(u_{n},r_{n})^{\varepsilon_{n}})$. Also we recall that an $Y$-sequence $d=((u_{1},r_{1})^{\varepsilon_{1}},...,(u_{n},r_{n})^{\varepsilon_{n}})$ is called an identity $Y$-sequence if it is either empty or if $\theta(d)=\prod_{i=1,n}u_{i}r_{i}^{\varepsilon_{i}}u_{i}^{-1}=1$ in $\hat{F}$. The definition of Peiffer equivalence is based on Peiffer operations on $Y$-sequences and reads as follows.
\begin{itemize}
\item [(i)] An \textit{elementary Peiffer exchange} replaces an adjacent pair $((u, r)^{\varepsilon},(v,s)^{\mu})$ in a $Y$-sequence by $((ur^{\varepsilon}u^{-1}v,s)^{\mu},(u,r)^{\varepsilon})$. 
\item [(ii)] A \textit{Peiffer deletion} deletes an adjacent pair $((u,r)^{\varepsilon},(u,r)^{-\varepsilon})$ in a $Y$-sequence. 
\item [(iii)] A \textit{Peiffer insertion} is the inverse of the Peiffer deletion.
\end{itemize}
The equivalence relation on the set of $Y$-sequences generated by the above operations is called \textit{Peiffer equivalence}. We recall from \cite{BH} what does it mean for an identity $Y$-sequence $((u_{1}, r_{1})^{\varepsilon_{1}},...,(u_{n},r_{n})^{\varepsilon_{n}})$ to have the primary identity property. This means that the indices $1,2,...,n$ are grouped into pairs $(i,j)$ such that $r_{i}=r_{j}$, $\varepsilon_{i}=-\varepsilon_{j}$ and $u_{i}=u_{j}$ modulo $N$ where $N$ is the normal subgroup of $\hat{F}$ generated by $\mathbf{r}$. Proposition 16 of \cite{BH} shows that every such sequence is Peiffer equivalent to the empty sequence. Given an identity $Y$-sequence $d$ which is equivalent to the empty sequence 1, we would be interested to know what kind of insertions $((u,r)^{\varepsilon},(u,r)^{-\varepsilon})$ are used along the way of transforming $d$ to 1. It is obvious that keeping track of that information is vital to tackle the Whitehead problem. 

Of a particular importance to our paper is the relation module $\mathcal{N}(\mathcal{P})$. Given $\mathcal{P}=( \mathbf{x},\mathbf{r} )$ a presentation for a group $G$, we let $\alpha: \hat{F} \rightarrow G$ and $\beta: N \rightarrow N/[N,N]$ be the canonical homomorphisms where $N$ is the normal closure of $\mathbf{r}$ in $\hat{F}$ and $[N,N]$ its commutator subgroup. There is a well defined $G$-action on $\mathcal{N}(\mathcal{P})=N/[N,N]$ given by
\begin{equation*}
	w^{\alpha}\cdot s^{\beta}=(wsw^{-1})^{\beta}
\end{equation*}
for every $w\in \hat{F}$ and $s \in N$. This action extends to an action of $\mathbb{Z}G$ over $\mathcal{N}(\mathcal{P})$ by setting
\begin{equation*}
	(w_{1}^{\alpha} \pm w_{2}^{\alpha})\cdot s^{\beta}=(w_{1} sw^{-1}_{1}w_{2}s^{\pm1}w^{-1}_{2})^{\beta}.
\end{equation*}
When $\cP$ is aspherical, the basis of $\mathcal{N}(\mathcal{P})$ as a free $\mathbb{Z}G$ module is the set of elements $r^{\beta}$ with $r \in \mathbf{r}$. This means in particular that whenever $\mathcal{P}=( \mathbf{x},\mathbf{r} )$ is aspherical, $\mathbf{r}$ can not contain relators $r,s$ such that $s=r^{-1}$ since this would violate the freeness of $\mathcal{N}(\mathcal{P})$.

The aim of Section \ref{poma} of the present paper is to offer an alternative way in dealing with the asphericity of a group presentation $\mathcal{P}=( \mathbf{x},\mathbf{r})$ by considering the free crossed module $(\Sigma^{e},\hat{F},\delta^{e})$ arising from the extended presentation $\mathcal{P}^{e}=( \mathbf{x},\mathbf{r} \cup \mathbf{r}^{-1})$ of the same group where $\mathbf{r}^{-1}=\{r^{-1}: r \in \mathbf{r}\}$. Here $\Sigma^{e}$ is the quotient of the free group generated by the symbols $(u,r^{\varepsilon})$ with $u \in \hat{F}$ and $r^{\varepsilon} \in \mathbf{r} \cup \mathbf{r}^{-1}$ subject to relations $((u,r^{\varepsilon}) , (v,s^{\delta}))=((u{r^{\varepsilon}u^{-1}}v, s^{\delta}),(u,r^{\varepsilon}))$, the action of $\hat{F}$ on $\Sigma^{e}$ and the map $\delta^{e}$ are the usual ones defining the free crossed module for the extended presentation $\mathcal{P}^{e}$. The advantage of working with $(\Sigma^{e},\hat{F},\delta^{e})$ is that unlike to $\Sigma$, in $\Sigma^{e}$ adjacent pairs like $((u,r^{\varepsilon}),(u,r^{-\varepsilon}))$ do not cancel out and this enables us to express the asphericity in terms of such insertions. This is realised by considering the kernel $\pi^{e}$ of $\delta^{e}$ which is the analogue of the module $\pi$ of identities for $\mathcal{P}$ in the standard theory and is not trivial when $\mathcal{P}$ is aspherical. We call $\pi^{e}$ the generalised module of identities for $\mathcal{P}$. 

To prove our results we apply techniques from the theory of semigroup actions and to this end we use concepts like the universal enveloping group $\mathcal{G}(S)$ of a given semigroup $S$, the dominion of a subsemigroup $U$ of a semigroup $S$ and the tensor product of semigroup actions. These concepts are explained, with references, in Section \ref{ma}.

\section{Useful notions from semigroup theory}

Semigroup actions and the universal enveloping group of a monoid will play a crucial role in the first part of the paper therefore we will pay a special attention to them for the benefit of the reader not familiar with the subject.  We will list below some basic notions and results that are used in the paper. For further results on the monoid actions the reader may consult the monograph \cite{Howie}. Results related to the universal enveloping group of a monoid may be found in \cite{Bergman} and \cite{Cohn}.

\subsection{Monoid actions} \label{ma}

Given $S$ a monoid with identity element 1 and $X$ a nonempty set, we say that $X$ is a \textit{left S-system} if there is an action $(s,x) \mapsto sx$ from $S \times X$ into $X$ with the properties
\begin{align*}
(st)x&=s(tx) \text{ for all } s,t \in S \text{ and } x \in X,\\
1x&=x \text{ for all } x \in X.
\end{align*}
Right $S$-systems are defined analogously in the obvious way. Given $S$ and $T$  (not necessarily different) monoids, we say that $X$ is an \textit{(S,T)-bisystem} if it is a left $S$-system, a right $T$-system, and if
\begin{equation*}
(sx)t=s(xt) \text{ for all } s \in S, t \in T \text{ and } x \in X.
\end{equation*}
If $X$ and $Y$ are both left $S$-systems, then an \textit{S-morphism} or \textit{S-map} is a map $\phi: X \rightarrow Y$ such that
\begin{equation*}
\phi(sx)=s\phi(x) \text{ for all } s \in S \text{ and } x \in X.
\end{equation*}
Morphisms of right $S$-systems and of $(S,T)$-bisystems are defined in an analogue way. 
If we are given a left $T$-system $X$ and a right $S$-system $Y$, then we can give the cartesian product $X \times Y$ the structure of an $(T,S)$-bisystem by setting
\begin{equation*}
t(x,y)=(tx,y) \text{ and } (x,y)s=(x,ys).
\end{equation*}
Let now $A$ be an $(T,U)$-bisystem, $B$ an $(U,S)$-bisystem and $C$ an $(T,S)$-bisystem. As explained above, we can give to $A \times B$ the structure of an $(T,S)$-bisystem. With this in mind we say that a $(T,S)$-map $\beta: A \times B \rightarrow C$ is a \textit{bimap} if
\begin{equation*}
\beta(au,b)=\beta(a,ub) \text{ for all } a\in A, b \in B \text{ and } u \in U.
\end{equation*}
A pair $(A\otimes_{U}B,\psi)$ consisting of a $(T,S)$-bisystem $A\otimes_{U}B$ and a bimap $\psi: A \times B \rightarrow A \otimes_{U}B$ will be called a \textit{tensor product of A and B over U} if for every $(T,S)$-bisystem $C$ and every bimap $\beta: A \times B \rightarrow C$, there exists a unique $(T,S)$-map $\bar{\beta}: A\otimes_{U}B \rightarrow C$ such that the diagram
\begin{equation*}
\xymatrix{A\times B \ar[d]_{\beta} \ar[r]^{\psi} & A \otimes_{U} B \ar[ld]^{\bar{\beta}}\\C
}
\end{equation*}
commutes.  It is proved in \cite{Howie} that $A\otimes_{U}B$ exists and is unique up to isomorphism. The existence theorem reveals that $A\otimes_{U}B=(A \times B)/\tau$ where $\tau$ is the equivalence on $A \times B$ generated by the relation
\begin{equation*}
T=\{ ((au,b),(a,ub)): a \in A, b \in B, u \in U\}.
\end{equation*}
The equivalence class of a pair $(a,b)$ is usually denoted by $a \otimes_{U}b$. To us is of interest the situation when $A=S=B$ where $S$ is a monoid and $U$ is a submonoid of $S$. Here $A$ is clearly regarded as an $(S,U)$-bisystem with $U$ acting on the right on $A$ by multiplication, and $B$ as an $(U,S)$-bisystem where $U$ acts on the left on $B$ by multiplication. 

Another concept that is important to our approach is that of the dominion which is defined in \cite{Isbell} from Isbell. By definition, if $U$ is a submonoid of a monoid $S$, then the dominion $\text{Dom}_{S}(U)$ consists of all the elements $d \in S$ having the property that for every monoid $T$ and every pair of monoid homomorphisms $f,g: S \rightarrow T$ that coincide in $U$, it follows that $f(d)=g(d)$. Related to dominions there is the well known zigzag Theorem of Isbell. We will present here the Stenstrom version of it (Theorem 8.3.3 of \cite{Howie}) which reads. \textit{Let $U$ be a submonoid of a monoid $S$ and let $d \in S$. Then, $d \in \text{Dom}_{S}(U)$ if and only if $d \otimes_{U}1=1\otimes_{U}d$ in the tensor product $A=S \otimes_{U}S$}. We mention here that this result holds true if $S$ turns out to be a group and $U$ a subgroup, both regarded as monoids. A key result (theorem 8.3.6 of \cite{Howie}) that is used in the next section is the fact that any inverse semigroup $U$ is absolutely closed in the sense that for every  semigroup $S$ containing $U$ as a subsemigroup, $\text{Dom}_{S}(U)=U$. It is obvious that groups are absolutely closed as special cases of inverse monoids (see \cite{epidom2}).

\subsection{The universal enveloping group of a monoid} \label{uegr}

If a monoid $S$ is given by the monoid presentation $\langle  Z, Q\rangle$, then its \textit{universal enveloping group} $\mathcal{G}(S)$ (see \cite{Bergman} and \cite{Cohn}) is defined to be the group given by the group presentation $( Z, \hat{Q} )$ where $\hat{Q}$ consists of all words $u v^{-1}$ whenever $(u,v) \in Q$ where $v^{-1}$ is the inverse of $v$ in the free group over $Z$. We let for future use $\sigma: FM(Z) \rightarrow S$ be the respective canonical homomorphism where $FM(Z)$ is the free monoid on $Z$. It is easy to see that there is a monoid homomorphism $\mu: S \rightarrow \mathcal{G}(S)$ which satisfies the following universal property. For every group $G$ and monoid homomorphism $f: S \rightarrow G$, there is a unique group homomorphism $\hat{f}: \mathcal{G}(S) \rightarrow G$ such that $\hat{f} \mu=f$. This universal property is an indication of an adjoint situation. Specifically, the functor $\mathcal{G}:\mathbf{Mon} \rightarrow \mathbf{Grp}$ which maps every monoid to its universal group, is a left adjoint to the forgetful functor $U: \mathbf{Grp} \rightarrow \mathbf{Mon}$. This ensures that $\mathcal{G}(S)$ is an invariant of the presentation of $S$. Also we let $j: FM(Z) \rightarrow FG(Z)$, where $FG(Z)$ is the free group on $Z$, be the the monoid homomorphism arising from the identity map on $Z$. It is obvious that $j$ is injective, hence $FM(Z)$ can be regarded as a sub monoid of $FG(Z)$ via $j$.

\section{Peiffer operations and monoid actions} \label{poma}

Assume that $\mathcal{P}=( \mathbf{x},\mathbf{r})$ is a group presentation such that $\mathbf{r}$ does not contain relations $r,s$ such that $s=r^{-1}$. This restriction will be assumed throughout the rest of the paper.  Before we explain how monoid actions are used to deal with the Peiffer operations on $Y$-sequences of a presentation $\mathcal{P}=( \mathbf{x},\mathbf{r})$, we will introduce several monoids. The first monoid $\Upsilon$ will be defined in the following fashion. Given a group presentation $\mathcal{P}=( \mathbf{x},\mathbf{r})$, we consider the extended presentation of it $\mathcal{P}^{e}=( \mathbf{x},\mathbf{r} \cup \mathbf{r}^{-1})$ where $\mathbf{r}^{-1}=\{r^{-1}| r \in \mathbf{r}\}$ is the set of inverses in $\hat{F}$ of all the relators $r \in \mathbf{r}$. We let $FM(\hat{F} \times (\mathbf{r} \cup \mathbf{r}^{-1}))$ be the free monoid on $\hat{F} \times (\mathbf{r} \cup \mathbf{r}^{-1})$ whose elements are finite (possibly empty) strings $((u_{1},r^{\varepsilon_{1}}_{1}), \dots, (u_{n},r^{\varepsilon_{n}}_{n}))$ with each $(u_{i},r^{\varepsilon_{i}}_{i}) \in \hat{F} \times (\mathbf{r} \cup \mathbf{r}^{-1})$. It is beneficial at this moment to reveal a relationship between the set of $Y$-sequences for the presentation $\mathcal{P}=( \mathbf{x},\mathbf{r})$ and the strings of $FM(\hat{F} \times (\mathbf{r} \cup \mathbf{r}^{-1}))$. From the definition we see that a $Y$-sequence is exactly an element of the free monoid $FM(Y \cup Y^{-1})$ where $Y=\hat{F} \times \mathbf{r}$ and $Y^{-1}=\{(u,r)^{-1}: (u,r) \in Y\}$ is the set of formal inverses of the elements of $Y$. But on the other hand, there is a homomorphism of free monoids $\xi: FM(Y \cup Y^{-1}) \rightarrow FM(\hat{F} \times (\mathbf{r} \cup \mathbf{r}^{-1}))$ which maps every free generator $(u,r)^{\varepsilon}$ of the first to the free generator $(u,r^{\varepsilon})$ of the second. The homomorphism $\xi$ is surjective since every generator of $FM(\hat{F} \times (\mathbf{r} \cup \mathbf{r}^{-1}))$ is in the image of $\xi$. Also we see that $\xi$ is injective since it is induced by the map $\tilde{\xi}: Y \cup Y^{-1} \rightarrow \hat{F} \times (\mathbf{r} \cup \mathbf{r}^{-1})$ such that $(u,r)^{\varepsilon} \mapsto (u,r^{\varepsilon})$ which is injective for if $\tilde{\xi}((u,r)^{\varepsilon})=\tilde{\xi}((v,s)^{\delta})$, then $u=v$ and $r^{\varepsilon}=s^{\delta}$. Since $s\neq r^{-1}$, we obtain that $s=r$ and $\varepsilon=\delta$.
\begin{remark} \label{ident}
We identify any $Y$-sequence over $\mathcal{P}=( \mathbf{x},\mathbf{r})$ with an element of the free monoid $FM(\hat{F} \times (\mathbf{r} \cup \mathbf{r}^{-1}))$ where the empty $Y$-sequence is identified with the unit of $FM(\hat{F} \times (\mathbf{r} \cup \mathbf{r}^{-1}))$ and any non empty $Y$-sequence $d=((u_{1}, r_{1})^{\varepsilon_{1}},\dots, (u_{n},r_{n})^{\varepsilon_{n}})$ is identified with $\xi(d)=((u_{1},r^{\varepsilon_{1}}_{1}),\dots, (u_{n},r^{\varepsilon_{n}}_{n}))$. 
\end{remark}

Further, if $FG(\hat{F} \times (\mathbf{r} \cup  \mathbf{r}^{-1}))$ is the free group on $\hat{F} \times (\mathbf{r} \cup \mathbf{r}^{-1})$, the map $j: FM(\hat{F} \times (\mathbf{r} \cup \mathbf{r}^{-1})) \rightarrow FG(\hat{F} \times (\mathbf{r} \cup \mathbf{r}^{-1}))$ defined in section \ref{uegr} which sends any string $d$ to itself is an injective homomorphism of monoids. This means that via $j$, $FM(\hat{F} \times (\mathbf{r} \cup \mathbf{r}^{-1}))$ can be regarded as a submonoid of $FG(\hat{F} \times (\mathbf{r} \cup \mathbf{r}^{-1}))$.

If as before $\theta^{e}: FG(\hat{F} \times (\mathbf{r} \cup \mathbf{r}^{-1})) \rightarrow \hat{F}$ is the group homomorphism defined by $\theta^{e}((u, r^{\varepsilon}))=u r^{\varepsilon} u^{-1}$, then its restriction $\theta^{e} \circ j$ on $FM(\hat{F} \times (\mathbf{r} \cup \mathbf{r}^{-1}))$ is a homomorphism of monoids. For any string $d \in FM(\hat{F} \times (\mathbf{r} \cup \mathbf{r}^{-1}))$ we will write $\theta^{e}(d)$ to mean $(\theta^{e} \circ j)(d)$. Also the left action of $\hat{F}$ on $FG(\hat{F} \times (\mathbf{r} \cup \mathbf{r}^{-1}))$ defined on free generators $(u,r^{\varepsilon})$ by $^{v}(u,r^{\varepsilon})=(vu, r^{\varepsilon})$ restricts to $FM(\hat{F} \times (\mathbf{r} \cup \mathbf{r}^{-1}))$ in the obvious fashion.

Now we are ready to define our first monoid $\Upsilon$ by the monoid presentation $\langle  \hat{F} \times (\mathbf{r} \cup \mathbf{r}^{-1}), P \rangle $ where 
and $P$ consists of all pairs $(ab,{^{\theta^{e} (a)}}b a)$ where $a,b \in \hat{F} \times (\mathbf{r} \cup \mathbf{r}^{-1})$.

The second monoid is the universal enveloping group $\mathcal{G}(\Upsilon)$ of $\Upsilon$ given by the group presentation $(\hat{F} \times (\mathbf{r} \cup \mathbf{r}^{-1}), \hat{P} )$ where $\hat{P}$ is the set of all words $aba^{-1}(^{\theta^{e}(a)}b)^{-1}$ whenever $(ab,{^{\theta^{e} (a)}}b a) \in P$. Here $c^{-1}$ denotes the inverse of the string $c$ in the free group $FG(\hat{F} \times (\mathbf{r} \cup \mathbf{r}^{-1}))$. For future reference, we denote the normal closure of $\hat{P}$ in $FG(\hat{F} \times (\mathbf{r} \cup \mathbf{r}^{-1}))$ by $\mathbf{P}$. It is important to observe here that $\mathcal{G}(\Upsilon)$ is nothing but the group $\Sigma^{e}$ which appears in the free crossed module $(\Sigma^{e}, \hat{F}, \delta^{e})$ associated with $\mathcal{P}^{e}=( \mathbf{x},\mathbf{r} \cup \mathbf{r}^{-1})$. 

Adhering to the notations of section \ref{uegr}, we let $\sigma: FM(\hat{F} \times (\mathbf{r} \cup \mathbf{r}^{-1})) \rightarrow \Upsilon$ be the canonical epimorphism, and $\mu: \Upsilon \rightarrow \mathcal{G}(\Upsilon)$ the homomorphism of the universal property of $\mathcal{G}(\Upsilon)$.

The third monoid we consider is the submonoid $\mathfrak{U}$ of $\Upsilon$, having the same unit as $\Upsilon$, and is generated from all the elements of the form $\sigma((u,r^{\varepsilon}), (u, r^{-\varepsilon}))$ with $(u,r^{\varepsilon}) \in \hat{F} \times (\mathbf{r} \cup \mathbf{r}^{-1})$. This monoid, acts on the left and on the right on $\Upsilon$ by the multiplication in $\Upsilon$. 

The last monoid considered is the subgroup $\hat{\mathfrak{U}}$ of $\mathcal{G}(\Upsilon)$ generated by $\mu(\mathfrak{U})$. Similarly to above, $\hat{\mathfrak{U}}$ acts on $\mathcal{G}(\Upsilon)$ by multiplication.

Given 
$$\alpha=((u_{1}, r_{1})^{\varepsilon_{1}},\dots, (u_{i}, r_{i})^{\varepsilon_{i}},(u_{i+1}, r_{i+1})^{\varepsilon_{i+1}},...,(u_{n}, r_{n})^{\varepsilon_{n}})$$ 
an $Y$-sequence over the group presentation $\mathcal{P}=( \mathbf{x},\mathbf{r})$, then performing an elementary Peiffer operation on $\alpha$ can be interpreted in a simple way in terms of the monoids $\Upsilon$ and $\mathfrak{U}$. Let 
$$\beta=((u_{1}, r_{1})^{\varepsilon_{1}},\dots, (u_{i}r^{\varepsilon_{i}}_{i}u^{-1}_{i}u_{i+1}, r_{i+1})^{\varepsilon_{i+1}},(u_{i}, r_{i})^{\varepsilon_{i}},...,(u_{n}, r_{n})^{\varepsilon_{n}})$$ be the sequence obtained from $\alpha$ by interchanging $(u_{i}, r_{i})^{\varepsilon_{i}}$ with $(u_{i+1}, r_{i+1})^{\varepsilon_{i+1}}$. Letting now
$$\xi(\alpha)=((u_{1}, r^{\varepsilon_{1}}_{1}),\dots, (u_{i}, r^{\varepsilon_{i}}_{i}),(u_{i+1}, r^{\varepsilon_{i+1}}_{i+1}),...,(u_{n}, r^{\varepsilon_{n}}_{n}))$$
and 
$$\xi(\beta)=((u_{1}, r_{1})^{\varepsilon_{1}},\dots, (u_{i}r^{\varepsilon_{i}}_{i}u^{-1}_{i}u_{i+1}, r_{i+1})^{\varepsilon_{i+1}},(u_{i}, r_{i})^{\varepsilon_{i}},...,(u_{n}, r_{n})^{\varepsilon_{n}})$$
be their corresponding strings in $FM(\hat{F} \times (\mathbf{r} \cup \mathbf{r}^{-1}))$ of Remark \ref{ident}, we see from the definition of $\Upsilon$, that $\sigma(\xi(\alpha))=\sigma(\xi(\beta))$. By an inductive argument we see that if $\beta$ is obtained from $\alpha$ by performing finitely many Peiffer exchanges, then $\sigma(\xi(\alpha))=\sigma(\xi(\beta))$.

Before we see the effect that a Peiffer insertion in $\alpha$ has on $\sigma(\xi(\alpha))$ we need the first claim of the following.
\begin{lemma} \label{central}
The elements of $\mathfrak{U}$ are central in $\Upsilon$ and those of $\hat{\mathfrak{U}}$ are central in $\mathcal{G}(\Upsilon)$.
\end{lemma}
\begin{proof}
We see that for every $(u, r^{\varepsilon}) \text{ and } (v,s^{\delta}) \in \hat{F} \times (\mathbf{r} \cup \mathbf{r}^{-1})$, 
\begin{align*}
\sigma((u, r^{\varepsilon}), (u, r^{-\varepsilon}), (v,s^{\delta}))&=\sigma \left(~^{\theta^{e}((u, r^{\varepsilon}))\theta^{e}((u, r^{-\varepsilon}))}(v,s^{\delta}),  (u, r^{\varepsilon}), (u, r^{-\varepsilon})\right)\\
&=\sigma((v,s^{\delta}), (u, r^{\varepsilon}), (u, r^{-\varepsilon})).
\end{align*}
Since elements $\sigma((v,s^{\delta}))$ and $\sigma((u, r^{\varepsilon}))\sigma((u, r^{-\varepsilon}))$ are generators of $\Upsilon$ and $\mathfrak{U}$ respectively, then the first claim holds true. The second claim follows easily.
\end{proof}

If we insert $((u,r)^{\varepsilon}, (u,r)^{-\varepsilon})$ at some point in a $Y$-sequence $\alpha=((u_{1},r_{1})^{\varepsilon_{1}},...,(u_{n},r_{n})^{\varepsilon_{n}})$ to obtain $\alpha'=((u_{1},r_{1})^{\varepsilon_{1}},\dots, (u,r)^{\varepsilon}, (u,r)^{-\varepsilon},...,(u_{n},r_{n})^{\varepsilon_{n}})$, then from Lemma \ref{central}, 
\begin{equation*}
\sigma\xi (\alpha')=\sigma \xi(\alpha) \cdot \sigma \xi ((u,r)^{\varepsilon})\sigma \xi ((u,r)^{-\varepsilon}),
\end{equation*}
which means that inserting $((u,r)^{\varepsilon}, (u,r)^{-\varepsilon})$ inside a $Y$-sequence $\alpha$ has the same effect as multiplying the corresponding $\sigma \xi (\alpha)$ in $\Upsilon$ by the element $\sigma \xi ((u,r)^{\varepsilon})\sigma \xi ((u,r)^{-\varepsilon})$ of $\mathfrak{U}$. For the converse, it is obvious that the word $\beta \in FM(Y \cup Y^{-1})$ representing $\sigma \xi (\alpha)\cdot \sigma \xi ((u,r)^{\varepsilon})\sigma \xi ((u,r)^{-\varepsilon})$ is Peiffer equivalent to $\alpha$. Of course the deletion has the obvious interpretation in our semigroup theoretic terms as the inverse of the above process. We retain the same names for our semigroup operations, that is insertion for multiplication by $\sigma \xi ((u,r)^{\varepsilon})\sigma \xi ((u,r)^{-\varepsilon})$ and deletion for its inverse. Related to these operations on the elements of $\Upsilon$ we make the following definition.
\begin{definition}
We denote by $\sim_{\mathfrak{U}}$ the equivalence relation in $\Upsilon$ generated by all pairs $(\sigma \xi (\alpha),\sigma \xi (\alpha)\cdot \sigma\xi (a)\sigma\xi(a^{-1}))$ where $\alpha \in \text{FM}(Y\cup Y^{-1})$ and $a \in Y \cup Y^{-1}$. We say that two elements $\sigma\xi (a_{1})\cdot \cdot \cdot \sigma\xi (a_{n})$ and $\sigma \xi (b_{1})\cdot \cdot \cdot \sigma\xi(b_{m})$ where $m,n \geq 0$ are \textit{Peiffer equivalent in $\Upsilon$} if they fall in the same $\sim_{\mathfrak{U}}$-class.
\end{definition}
From what we said before it is obvious that two $Y$-sequences $\alpha$ and $\beta$ are Peiffer equivalent in the usual sense if and only if $\sigma\xi (\alpha)\sim_{\mathfrak{U}} \sigma\xi(\beta)$. For this reason we decided to make the following convention. If $\alpha=(a_{1},...,a_{n})$ is a $Y$-sequence (resp. an identity $Y$-sequence), then its image in $\Upsilon$, $\sigma\xi (\alpha)$ will again be called a $Y$-sequence (resp. an identity $Y$-sequence). In the future instead of working directly with an $Y$-sequence $\alpha$, we will work with its image $\sigma\xi(\alpha)$.

We note that it should be mentioned that the study of $\sim_{\mathfrak{U}}$ might be as hard as the study of Peiffer operations on $Y$-sequences, and at this point it seems we have not made any progress at all. In fact this definition will become useful later in this section and yet we have to prove a few more things before we utilize it.

The process of inserting and deleting generators of $\mathfrak{U}$ in an element of $\Upsilon$ is related to the following new concept. Given $U$ a submonoid of a monoid $S$ and $d \in S$, then we say that $d$ belongs to the \textit{weak dominion of} $U$, shortly written as $d \in \text{WDom}_{S}(U)$, if for every group $G$ and every monoid homomorphisms $f,g:S \rightarrow G$ such that $f(u)=g(u)$ for every $u \in U$, then $f(d)=g(d)$. An analogue of the Stenstr\"{o}m version of Isbell's theorem for weak dominion holds true. The proof of the if part of its analogue is similar to that of Isbell theorem apart from some minor differences that reflect the fact that we are working with $WDom$ rather than $Dom$ and that will become clear along the proof, while the converse relies on the universal property of $\mu: S \rightarrow \mathcal{G}(S)$.  
\begin{proposition} \label{wd prop}
Let $S$ be a monoid, $U$ a submonoid and let $\hat{U}$ be the subgroup of $\mathcal{G}(S)$ generated by elements $\mu(u)$ with $u \in U$. Then $d \in \text{WDom}_{S}(U)$ if and only if $\mu(d) \in \hat{U}$.
\end{proposition}
\begin{proof}
The set $\hat{A}=\mathcal{G}(S)\otimes_{\hat{U}} \mathcal{G}(S)$ has an obvious $(\mathcal{G}(S),\mathcal{G}(S))$-bisystem structure. The free abelian group $\mathbb{Z}\hat{A}$ on $\hat{A}$ inherits a $(\mathcal{G}(S),\mathcal{G}(S))$-bisystem structure if we define
\begin{equation*}
g \cdot \sum z_{i}(g_{i} \otimes_{\hat{U}}h_{i})=\sum z_{i}(gg_{i}\otimes_{\hat{U}}h_{i}) \text{ and } \left(\sum z_{i} (g_{i} \otimes_{\hat{U}}h_{i})\right)\cdot g=\sum z_{i}(g_{i} \otimes_{\hat{U}}h_{i}g).
\end{equation*}
The set $\mathcal{G}(S) \times \mathbb{Z}\hat{A}$ becomes a group by defining 
\begin{equation*}
(g,\sum z_{i} g_{i} \otimes_{\hat{U}} h_{i})\cdot (g',\sum z'_{i} g'_{i} \otimes_{\hat{U}}h'_{i})=(gg', \sum z_{i} g_{i} \otimes_{\hat{U}} h_{i}g'+\sum z'_{i} gg'_{i} \otimes_{\hat{U}}h'_{i}).
\end{equation*}
The associativity is proved easily. The unit element is $(1,0)$ and for every $(g,\sum z_{i} g_{i} \otimes_{\hat{U}} h_{i})$ its inverse is the element $(g^{-1},-\sum z_{i} g^{-1}g_{i} \otimes_{\hat{U}} h_{i} g^{-1})$. Let us now define
\begin{equation*}
\beta: S \rightarrow \mathcal{G}(S) \times \mathbb{Z}\hat{A} \text{ by } s\mapsto (\mu(s),0),
\end{equation*}
which is clearly a monoid homomorphism, and
\begin{equation*}
\gamma: S \rightarrow \mathcal{G}(S) \times \mathbb{Z}\hat{A} \text{ by } s \mapsto (\mu(s), \mu(s) \otimes_{\hat{U}}1-1\otimes_{\hat{U}} \mu(s)),
\end{equation*}
which is again seen to be a monoid homomorphism. These two coincide on $U$ since for every $u \in U$
\begin{equation*}
\gamma(u)=(\mu(u),\mu(u) \otimes_{\hat{U}}1-1\otimes_{\hat{U}}\mu(u))=(\mu(u),0)=\beta(u).
\end{equation*}
The last equality and the assumption that $d \in \text{WDom}_{S}(U)$ imply that $\beta(d)=\gamma(d)$, therefore
\begin{equation*}
(\mu(d),0)=(\mu(d),\mu(d)\otimes_{\hat{U}} 1-1 \otimes_{\hat{U}} \mu(d)),
\end{equation*}
which shows that $\mu(d)\otimes_{\hat{U}} 1=1 \otimes_{\hat{U}} \mu(d)$ in the tensor product $\mathcal{G}(S)\otimes_{\hat{U}} \mathcal{G}(S)$ and therefore Theorem 8.3.3, \cite{Howie}, applied for monoids $\mathcal{G}(S)$ and $\hat{U}$, implies that $\mu(d) \in \text{Dom}_{\mathcal{G}(S)}(\hat{U})$. But $\text{Dom}_{\mathcal{G}(S)}(\hat{U})=\hat{U}$ as from Theorem 8.3.6, \cite{Howie} every inverse semigroup is absolutely closed, whence $\mu(d) \in \hat{U}$.

Conversely, suppose that $\mu(d) \in \hat{U}$ and we want to show that $d \in \text{WDom}_{S}(U)$. Let $G$ be a group and $f,g: S \rightarrow G$ two monoid homomorphisms that coincide in $U$, therefore the group homomorphisms $\hat{f}, \hat{g}: \mathcal{G}(S) \rightarrow G$ of the universal property of $\mu$ coincide in $\hat{U}$ which, from our assumption, implies that $\hat{f}(\mu(d))=\hat{g}(\mu(d))$, and then $f(d)=g(d)$ proving that $d \in \text{WDom}_{S}(U)$.
\end{proof}

Given a presentation $\mathcal{P}=( \mathbf{x},\mathbf{r})$ for a group $G$, we consider the free crossed module $(\mathcal{G}(\Upsilon), \hat{F}, \delta^{e})$ associated with $\mathcal{P}^{e}=( \mathbf{x},\mathbf{r} \cup \mathbf{r}^{-1})$. We recall that $\text{Ker}\delta^{e}$ is a $G$-module which is denoted here by $\pi^{e}$. The left action of $G$ on $\pi^{e}$ is given by the following rule. If $g \in G$ which is represented by a word $w \in \hat{F}$ and $((u_{1},r^{\varepsilon_{1}}_{1})^{\delta_{1}},\dots, (u_{n},r^{\varepsilon_{n}}_{n})^{\delta_{n}})\mathbf{P} \in \pi^{e}$, then
$$g \cdot (((u_{1},r^{\varepsilon_{1}}_{1})^{\delta_{1}},\dots, (u_{n},r^{\varepsilon_{n}}_{n})^{\delta_{n}})\mathbf{P})=((wu_{1},r^{\varepsilon_{1}}_{1})^{\delta_{1}},\dots, (wu_{n},r^{\varepsilon_{n}}_{n})^{\delta_{n}})\mathbf{P}.$$
It follows from this that $\hat{\mathfrak{U}}$ is a sub module of $\pi^{e}$. Further, since $(\mathcal{G}(\Upsilon), \hat{F}, \delta^{e})$ is free on the map $v: \mathbf{r} \cup \mathbf{r}^{-1} \rightarrow  \mathcal{G}(\Upsilon)$ such that $v(r^{\varepsilon})=\mu \sigma ((1,r^{\varepsilon}))$, then related to the map $v': \mathbf{r} \cup \mathbf{r}^{-1} \rightarrow H/P$ with $v'(r^{\varepsilon})=(1,r)^{\varepsilon}P$, there is a unique homomorphism $\nu: \mathcal{G}(\Upsilon) \rightarrow H/P$ such that $\nu v=v'$.

In terms of $\pi^{e}$ and $\hat{\mathfrak{U}}$ we prove the following analogue of Theorem 3.1 of \cite{papa-attach}.

\begin{theorem} \label{wdom}
The following assertions are equivalent.
\begin{itemize}
\item [(i)] The presentation $\mathcal{P}=( \mathbf{x},\mathbf{r})$ is aspherical.
\item [(ii)] For every identity $Y$-sequence $d$, $d \in \text{WDom}_{\Upsilon}(\mathfrak{U})$.
\item [(iii)] $\pi^{e}=\hat{\mathfrak{U}}$.
\end{itemize}
\end{theorem}
\begin{proof}
$(i) \Rightarrow (ii)$ Let $d=\sigma(a_{1})\cdot \cdot \cdot \sigma(a_{n}) \in \Upsilon$ be any identity $Y$-sequence ($a_{i}=(u_{i},r^{\varepsilon_{i}}_{i})$ for $i=1,....,n$) and as such it has to be Peiffer equivalent to 1. We proceed by showing that $d \in \text{WDom}_{\Upsilon}(\mathfrak{U})$. Let $G$ be any group and $f,g:\Upsilon \rightarrow G$ two monoid homomorphisms that coincide in $\mathfrak{U}$ and we want to show that $f(d)=g(d)$. The proof will be done by induction on the minimal number $h(d)$ of insertions and deletions needed to transform $d=\sigma(a_{1})\cdot \cdot \cdot \sigma(a_{n})$ to $1$. If $h(d)=1$, then $d \in \mathfrak{U}$ and $f(d)=g(d)$. Suppose that $h(d)=n>1$ and let $\tau$ be the first operation performed on $d$ in a series of operations of minimal length. After $\tau$ is performed on $d$, it is obtained an identity $Y$-sequence $d'$ with $h(d')=n-1$.  By induction hypothesis, $f(d')=g(d')$ and we want to prove that $f(d)=g(d)$. There are two possible cases for $\tau$. First, $\tau$ is an insertion and let $u=\sigma((u, r^{\varepsilon}), (u, r^{-\varepsilon}) )\in \mathfrak{U}$ be the element inserted. It follows that $f(d')=f(d)f(u)$ and $g(d')=g(d)g(u)$, but $f(u)=g(u)$, therefore from cancellation law in the group $G$ we get $f(d)=g(d)$. Second, $\tau$ is a deletion and let $u=\sigma((u, r^{\varepsilon}), (u, r^{-\varepsilon}) ) \in \mathfrak{U}$ be the element deleted, that is $d=d'u$. It follows immediately from the assumptions that $f(d)=g(d)$ proving that $d \in \text{WDom}_{\Upsilon}(\mathfrak{U})$.

$(ii) \Rightarrow (iii)$ Let $\tilde{d} \in \pi^{e}$. We may assume without loss of generality that no $(\mu\sigma(u, {r}^{\varepsilon}))^{-1}$ is represented in $\tilde{d}$ for if there is any such occurrence, we can multiply $\tilde{d}$ by $\mu\sigma((u,{r}^{\varepsilon}),(u,{r}^{-\varepsilon}))$ to obtain in return $\tilde{d}'$ where $(\mu\sigma(u,{r}^{\varepsilon}))^{-1}$ is now replaced by $\mu\sigma((u, {r}^{-\varepsilon}))$. It is obvious that if $\tilde{d}' \in \hat{\mathfrak{U}}$, then $\tilde{d} \in \hat{\mathfrak{U}}$ and conversely. Let now $d$ be any preimage of $\tilde{d}$ under $\mu$. It is clear that $d$ is an identity $Y$-sequence and as such $d \in \text{WDom}_{\Upsilon}(\mathfrak{U})$. Then Proposition \ref{wd prop} implies that $\tilde{d}=\mu(d) \in \hat{\mathfrak{U}}$.

$(iii) \Rightarrow (i)$ Assume that $\pi^{e}=\hat{\mathfrak{U}}$ and we want to show that any identity $Y$-sequence $d$ is Peiffer equivalent to 1. From the assumption for $d$ we have that $\mu(d) \in \hat{\mathfrak{U}}$ and then Proposition \ref{wd prop} implies that $d \in \text{WDom}_{\Upsilon}(\mathfrak{U})$. Consider now the homomorphism $\nu:\mathcal{G}(\Upsilon) \rightarrow H/P$ defined before the theorem. Writing $\tau$ for the constant morphism at 1 from $\Upsilon$ to $H/P$, we see that $\tau$ and the composition $\nu\mu$ coincide in $\mathfrak{U}$, therefore since $d \in \text{WDom}_{\Upsilon}(\mathfrak{U})$, it follows that $\nu\mu(d)=1$ in $H/P$. The asphericity of $\mathcal{P}$ now follows from Theorem 2.7, p.71 of \cite{cha}.
\end{proof}

\begin{proposition} \label{free}
	If $\mathcal{P}$ is aspherical, then $\hat{\mathfrak{U}}$ is a free $G$-module with bases the set of elements $\mu \sigma ((1,r), (1,r^{-1}))$ with $r \in \mathbf{r}$. 
\end{proposition}
\begin{proof}
	The result follows if we show that $\hat{\mathfrak{U}} \cong \mathcal{N}(\mathcal{P})$ as $G$-modules. For this we define
	$$\Omega: \mathcal{N}(\mathcal{P}) \rightarrow \hat{\mathfrak{U}}$$
	on free generators by $r^{\beta} \mapsto \mu \sigma ((1,r), (1,r^{-1}))$ which is clearly well defined and a surjective morphism of $G$-modules. Now we prove that $\Omega$ is injective. Let 
	$$\xi= \sum_{i=1}^{n} z_{i} \cdot u_{i}^{\alpha}\cdot r_{i}^{\beta} \in \text{Ker}\Omega,$$
	which means that
	\begin{equation} \label{ker}
		\prod_{i=1}^{n} (\mu \sigma ((u_{i},r_{i}),(u_{i},r^{-1}_{i})))^{z_{i}}=1.
	\end{equation}
	To prove that $\xi=0$ we will proceed as follows. Define 
	\begin{equation*}
		\gamma: FM(\hat{F} \times (\mathbf{r} \cup \mathbf{r}^{-1})) \rightarrow \mathcal{N}(\mathcal{P}) 
	\end{equation*}
	on free generators as follows
	\begin{equation*}
		(u,r^{\varepsilon}) \mapsto u^{\alpha}\cdot r^{\beta}.
	\end{equation*}
	It is easy to see that $\gamma$ is compatible with the defining relations of $\Upsilon$, hence it induces a homomorphism $g:\Upsilon \rightarrow \mathcal{N}(\mathcal{P})$ and then the universal property of $\mu$ implies the existence of $\hat{g}: \mathcal{G}(\Upsilon) \rightarrow \mathcal{N}(\mathcal{P})$ such that $\hat{g}\mu=g$. If we apply now $\hat{g}$ on both sides of (\ref{ker}) obtain
	$$2 \cdot \xi=\sum_{i=1}^{n}  z_{i} \cdot (2 \cdot u_{i}^{\alpha}\cdot r_{i}^{\beta})=0,$$
	proving that $\xi=0$.
\end{proof}

\section{Proof of the main theorem}

The proof of our main theorem is heavily based on two papers. The first one is \cite{SES} where McGlashan et al extended the Squier complex of a monoid presentation to a 3-complex and obtained a short exact sequence which will play a crucial role in the proof of our main theorem. The second one is \cite{LDHM2} where Pride realises the second homotopy group associated with a group presentation as the first homotopy group of a certain extension of the Squire complex arising from that presentation. For the sake of completeness we have added below a number of sections which tend to explain the material that is used in our proofs. Section \ref{rws} gives some basic material about rewriting systems since they are used in the construction of our complexes and in our proofs. In Section \ref{sqc} we explain in some details how the Squier complex of a monoid presentation is defined and the cellular chain complex associated with it. Further in section \ref{gc} we give the definition of the extended Squier complex as it appears in \cite{SES} and in Section \ref{ses-ses} we introduce the short exact sequence. Sections from \ref{exsqc} to \ref{3-st} are devoted to giving a more detailed proof of the exactness of the sequence. Section \ref{pc} is devoted to introducing the Pride complex associated with a group presentation and to explain ideas and results from \cite{LDHM2} since we make extensive use of them in our proofs. Finally, it is important to mention that theorem 6.6 of \cite{OK2002} is vital in the proof of key Lemma \ref{c-rel-ext}.

\subsection{Some basic concepts from rewriting systems} \label{rws}

A rewriting system is a pair $\cP=\langle \mathbf{x}, \mathbf{r}\rangle$ where $\mathbf{x}$ is a non empty set and $\mathbf{r}$ is a set of rules $r=(r_{+},r_{-}) \in F \times F$ where $F$ is the free monoid on $\mathbf{x}$ whose elements are ''words'' with letters from $\mathbf{x}$ and its unit element is the empty word $\lambda$. We will assume throughout that for every $r \in \mathbf{r}$, $r_{+}\neq r_{-}$. Related with $\mathbf{r}$ there is the so called the one single step reduction of words
$$\rightarrow_{\mathbf{r}}=\{(ur_{+}v,ur_{-}v)|r \in \mathbf{r} \text{ and } u, v \in F\}.$$
The reflexive and transitive closure of $\rightarrow_{\mathbf{r}}$ is denoted by $\rightarrow_{\mathbf{r}}^{\ast}$, and the reflexive, transitive and symmetric closure is denoted by $\leftrightarrow_{\mathbf{r}}^{\ast}$ and is also known as the Thue congruence generated by $\mathbf{r}$. The quotient $F/\leftrightarrow_{\mathbf{r}}^{\ast}$ forms a monoid $S$ whose elements are the congruence classes $\bar{u}$ of words $u \in F$, and the multiplication is given by $\bar{u}\cdot \bar{v}= \overline{uv}$. We say that the monoid $S$ is given by $\cP$, or that $\cP$ is a presentation for $S$. 

A rewriting system $\cP=\langle \mathbf{x}, \mathbf{r}\rangle$ is noetherian (or terminating) if there is no infinite chain
$$w \rightarrow_{\mathbf{r}}  w' \rightarrow_{\mathbf{r}}  \dots$$
and is confluent if whenever we have $w \rightarrow_{\mathbf{r}}^{\ast} w_{1}$ and $w \rightarrow_{\mathbf{r}}^{\ast} w_{2}$, then there is $z \in F$ such that $w_{1} \rightarrow_{\mathbf{r}}^{\ast} z$ and $w_{2} \rightarrow_{\mathbf{r}}^{\ast} z$. A rewriting system $\cP=\langle \mathbf{x}, \mathbf{r}\rangle$ is complete if it is both noetherian and confluent. 

Let $\cP=\langle \mathbf{x}, \mathbf{r}\rangle$ be a presentation for a monoid $S$. The natural epimorphism 
$$F \rightarrow S \text{ such that } w \mapsto \bar{w},$$ 
where $F$ is the free monoid on $\mathbf{x}$, extends linearly to a ring epimorphism 
$$\mathbb{Z}F \rightarrow \mathbb{Z}S$$
of the corresponding integral monoid rings. The kernel of this epimorphism is denoted by $J$ which as an abelian group is generated by all
$$u(r_{+}-r_{-})v \text{ where } u, v \in F \text{ and } r \in \mathbf{r}.$$
As a $(\mathbb{Z}F, \mathbb{Z}F)$-bimodule $J$ is generated by all $r_{+}-r_{-}$.

\subsection{The Squier complex of a monoid presentation} \label{sqc}

The material included in this section is taken from \cite{SES} (see also \cite{Sth}). At the end of the section we give shortly the respective terminology used in \cite{OK2002} which differs slightly from ours. The reason we explain this terminology is the use of Theorem 6.6 of \cite{OK2002} in the proof of our key Lemma \ref{c-rel-ext}. 

For every rewriting system $\cP=\langle \mathbf{x}, \mathbf{r}\rangle$ we can define its graph of derivations $\Gamma(\cP)$ whose vertices are the elements of $F$, and the edges are all quadruples 
$$e=(w,r, \varepsilon, w') \text{ where } w,w' \in F, \varepsilon=\pm 1, r \in \mathbf{r},$$
with initial, terminal and inverse functions
$$\iota e= wr_{\varepsilon}w', \tau e= wr_{-\varepsilon}w' \text{ and } e^{-1}=(w,r, -\varepsilon, w').$$
The edge $e$ is called positive if $\varepsilon=1$. We can think of $\Gamma(\cP)$ as a one dimensional cw-complex with 0-cells all the elements of $F$ and with 1-cells all positive edges. It is proved in \cite{sqcrev} that $\Gamma(\cP)$ is indeed a cw-complex of dimension 1. We note here that $e^{-1}=(w,r, -1, w')$ is not a new edge attached to the complex, but is defined to mean the topological inverse of the characteristic map of $e=(w,r, 1, w')$. A path $p$ of length $n$ in $\Gamma(\cP)$ is a sequence of edges $p=e_{1} \circ \dots \circ e_{i} \circ e_{i+1}\circ \dots \circ e_{n}$ where $\tau e_{i}=\iota e_{i+1}$ for $1 \leq i \leq n-1$. It is called positive if the edges are positive, and is called closed if $\iota e_{1}=\tau e_{n}$. The inverse of $p$ is defined to be the path $p^{-1}=e^{-1}_{n} \circ \dots \circ e^{-1}_{i+1} \circ e^{-1}_{i}\circ \dots \circ e^{-1}_{1}$.

There is a natural two-sided action of $F$ on $\Gamma(\cP)$. The action on vertices is given by the multiplication of $F$, and the action of $z,z' \in F$ on edges $e=(w,r, \varepsilon, w')$ is given by
$$z.e.z'=(zw,r, \varepsilon, w'z'),$$
and sometimes is called translation. This action extends to paths in the obvious way. 

Note that there is a 1-1 correspondence between the elements of $S$ given by $\cP$ and the connected components of $\Gamma(\cP)$ since  $u \leftrightarrow_{\mathbf{r}}^{\ast} v$ if and only if there is a path in $\Gamma(\cP)$ connecting $u$ with $v$. Also note that the generators of $J$ as an abelian group are the elements $\iota e- \tau e$ where $e$ is a positive edge. 

We say that two positive edges $e_{1}$ and $e_{2}$ are disjoint if they can be written in the form
$$e_{1}=f_{1}. \iota f_{2}, e_{2}=\iota f_{1}. f_{2}$$
where $f_{1}, f_{2}$ are positive edges. We say that an edge $e$ is left reduced (resp. right reduced) if it cannot be written in the form $u.f$ (resp. $f. u$) for some non empty word $u \in F$ and an edge $f$.

The Squier complex $\mathcal{D}(\mathcal{P})$ associated with $\cP$ is a combinatorial 2-complex with 1-skeleton $\Gamma(\cP)$, to which, for each pair of positive edges $e,f$ a 2-cell $[e,f]$ is attached along the closed path
$$\partial[e,f]=(e. \iota f) \circ (\tau e. f) \circ (e.\tau f)^{-1}\circ (\iota e. f)^{-1},$$
called the boundary of $[e,f]$. It is proved in \cite{sqcrev} that $\mathcal{D}(\cP)$ is a 2-dimensional cw-complex whose 1-skeleton is $\Gamma(\cP)$. Sometimes we refer to the 2-cell $[e,f]$ as a square 2-cell. If we let 
 $$\mathfrak{S}=\{\partial[e,f]| e, f \text{ are positive edges}\},$$ 
 then the set of 2-cells attached to the complex is in 1-1 correspondence with the loops in $\mathfrak{S}$. In other words, every square 2-cell $[e,f]$ is uniquely determined by the loop $\partial[e,f]$ which will be referred in what follows as the boundary of $[e,f]$. It may happen though that different pairs of positive edges give rise to the same square 2-cell. The following lemma characterises those pairs.
\begin{lemma} \label{unamb}
Two square cells $[e,f]$ and $[e',f']$ coincide if and only if there is some $w \in F$ such that $e=e'.w$ and $f'=w.f$, or $e'=e.w$ and $f=w.f'$.
\end{lemma} 
\begin{proof}
Let us prove first that $[e.w, f]=[e,w.f]$. Since a square 2-cell is uniquely determined by its own boundary, we should prove that the 2-cells have the same boundaries. Indeed,
\begin{align*}
\partial [e.w, f]&= ((e. w). \iota f) \circ ((\tau e \cdot w). f) \circ ((e^{-1}.w).\tau f) \circ ((\iota e \cdot w). f^{-1})\\
&= (e .(w \cdot \iota f)) \circ ((\tau e). (w.f)) \circ (e^{-1}. (w \cdot \tau f)) \circ (\iota e . ((w.f)^{-1}))\\
&= (e . (\iota(w.f))) \circ (\tau e . (w.f)) \circ (e^{-1}. (\tau(w.f))) \circ (\iota e . (w.f)^{-1})\\
&= \partial [e,w.f].
\end{align*}
Conversely, assume that $[e,f]=[e',f']$ and let $e=(u, r,1, v)$, $e'=(u', r',1, v')$, $f=(\bar{u}, \bar{r}, 1, \bar{v})$ and $f'=(\bar{u}', \bar{r}', 1, \bar{v}')$. Let $r=(a,b)$, $r'=(a',b')$, $\bar{r}=(\bar{a}, \bar{b})$ and $\bar{r}'=(\bar{a}',\bar{b}')$. Equating the boundaries we have the following
\begin{align*}
&(u, r, 1, v \cdot \bar{u}\bar{a}\bar{v})=e . \iota f=e' . \iota f'= (u', r', 1, v' \cdot \bar{u}' \bar{a}' \bar{v'})\\
&(ubv \cdot \bar{u}, \bar{r}, 1, \bar{v})=\tau e. f=\tau e' . f'= (u' b' v' \cdot \bar{u}', \bar{r}', 1, \bar{v}')\\
& (u av \cdot \bar{u}, \bar{r}, 1, \bar{v})= \iota e. f=\iota e' . f'= (u a' v' \cdot \bar{u}', \bar{r}', 1, \bar{v}' )\\
& (u, r, 1, v \cdot \bar{u} \bar{b} \bar{v})= e . \tau f=e'.\tau f'=(u', r', 1, v' \cdot \bar{u}' \bar{b}' \bar{v'}).
\end{align*}
From the first displayed equation we have that $u=u'$ and $(a,b)=r=r'=(a',b')$. It follows from this and the second displayed equation that $v\bar{u}=v'\bar{u'}$. If $v'$ is a prefix of $v$, then there is some $w$ such that $v'w=v$. In this case we have that
$$\iota e=u a v =uav'w =u'a' v' w=\iota e' w$$
and
$$\tau e=ubv=ubv'w=u'b'v'w=\tau e' w,$$
which shows that $e=e'.w$. From the third displayed equation we have that $(\bar{a}, \bar{b})=\bar{r}=\bar{r}'=(\bar{a}',\bar{b}')$ and that $\bar{v}=\bar{v}'$. This and the fourth displayed equation imply that $v \bar{u}=v'\bar{u}'$, hence $v'w \bar{u}=v'\bar{u}'$ and $w \bar{u}=\bar{u'}$. It follows now that 
$$\iota f'=\bar{u}'\bar{a}' \bar{v'}=w \bar{u}\bar{a}\bar{v}=w . \iota f$$
and
$$\tau f'= \bar{u}'\bar{b}' \bar{v}'=w \bar{u} \bar{b}\bar{v}=w . \tau f,$$
which prove that $f'=w. f$. 

In the case when $v$ is a prefix of $v'$ it follows in a similar way as above that there is some $w$ such that $e'=e.w$ and $f=w.f'$.
\end{proof}
\begin{corollary} \label{cunamb}
For every two positive edges $e,f$, the square 2-cell $[e,f]$ is uniquely determined by the pair $\text{st}([e,f])=(e.\iota f, \iota e. f)$ which is called the star of $[e,f]$.
\end{corollary}
\begin{proof}
If there is some 2-cell $[e',f']$ such that $e'.\iota f' =e. \iota f$, $\iota e' . f'=\iota e. f$, then as in Lemma \ref{unamb}, letting $e=(u, r,1, v)$, $e'=(u', r',1, v')$, $f=(\bar{u}, \bar{r}, 1, \bar{v})$  and $f'=(\bar{u}', \bar{r}', 1, \bar{v}')$ with $r=(a,b)$, $r'=(a',b')$, $\bar{r}=(\bar{a}, \bar{b})$ and $\bar{r}'=(\bar{a}',\bar{b}')$, we see that
$$e'.\iota f' =(u', r', 1, v'.\bar{u}' \bar{a}'\bar{v}')= e. \iota f= (u, r, 1, v. \bar{u}\bar{a}\bar{v})$$
and
$$\iota e' . f'=(u'a'v'.\bar{u}', \bar{r}', 1, \bar{v}')= \iota e. f=(uav.\bar{u}, \bar{r}, 1, \bar{v}),$$
which imply that $u=u'$, $r=r'$, $\bar{r}=\bar{r}'$, $\bar{v}=\bar{v}'$ and $v' \bar{u}'=v \bar{u}$.  If $v'$ is a prefix of $v$, then for some $w$ we have that $v=v' w$ and $w \bar{u}=\bar{u}'$. It follows in this case that
$$e=(u, r, 1, v)=(u', r', 1, v'w)=e'. w,$$
and
$$f'=(\bar{u}', \bar{r}', 1, \bar{v}')=(w \bar{u}, \bar{r}, 1, \bar{v})=w. f$$
Now Lemma \ref{unamb} implies that $[e,f]=[e',f']$. The case when $v$ is a prefix of $v'$ is analogous to the above.
\end{proof}
\begin{notation} \label{ns}
For every two positive edges $e,f$, we write the square 2-cell $[e,f]$ in a new form $[\iota e.\iota f, (e.\iota f, \iota e. f)]$. Corollary above implies that this notation is unambiguous.
\end{notation}

We have the chain complex
$$\xymatrix{\mathbf{C}(\mathcal{D}):  C_{2} \ar[r]^-{\partial_{2}} & C_{1} \ar[r]^{\partial_{1}} & C_{0} \ar[r] & 0},$$
where $C_{0}$, $C_{1}$ and $C_{2}$ are the free abelian groups generated by all 0-cells, positive edges and square 2-cells respectively. It is shown in \cite{sqcrev} that the boundary maps are given by
\begin{equation} \label{part1}
\partial_{1} e=\tau e- \iota e \text{ where } e \text{ is a positive edge}.
\end{equation}
\begin{equation} \label{part2s}
\partial_{2}[e,f]=e.(\iota f - \tau f)-(\iota e- \tau e).f \text{ where } e,f \text{ are positive edges}.
\end{equation}
The first map is defined in \cite{LDHM1} by $\partial_{1} e=\iota e- \tau e$ and it is this definition which is used in \cite{SES}. We note that the change of sign of the cellular boundary map at this level does not change the first homology group of the complex therefore we will stick to this new definition from now on. The complex $\mathbf{C}(\mathcal{D})$ has in fact some extra structure. The two-sided action of $F$ on $\Gamma(\cP)$ extends to the 2-cells by
$$w.[e,f].w'=[w.e,f.w'] \text{ where } w,w' \in F, e, f \text{ are positive edges}.$$
This action induces a left and a right action of $\mathbb{Z}F$ on $C_{2}$, which make $C_{2}$ a $(\mathbb{Z}F, \mathbb{Z}F)$-bimodule. In fact this bimodule is free on a certain basis which we will describe below. In order to make the description clear we need the following.
\begin{lemma} \label{bmsq}
Let $e,e'$ be left reduced positive edges and $f,f'$ be right reduced positive edges. Then for every $\alpha, \beta, \alpha', \beta' \in F$, $\alpha.[e,f]. \beta=\alpha' . [e',f']. \beta'$ if and only if $\alpha=\alpha'$, $\beta=\beta'$, $[e,f]=[e',f']$.
\end{lemma}
\begin{proof}
Let $e=(\lambda,r,1, v)$, $e'=(\lambda,r',1, v')$, $f=(\bar{u}, \bar{r}, 1, \lambda)$ and $f'=(\bar{u}', \bar{r}', 1, \lambda)$. From the assumption we have the equality of the 2-cells $[\alpha.e, f.\beta]=[\alpha'.e', f'.\beta']$, and then from Lemma \ref{unamb}, there is some $w \in F$ such that either $\alpha. e=\alpha'. e'. w$ and $w.f.\beta=f'. \beta'$, or $\alpha. e. w=\alpha'. e'$ and $f.\beta=w.f'. \beta'$. If any of the above holds true, then the lemma follows. We give here the proof when the first equality holds true since the other case is dual to the first. From the first given equality we have
$$(\alpha, r, 1, v)=\alpha. e=\alpha'. e'. w=(\alpha', r', 1, v'w),$$
which implies that $\alpha=\alpha'$, $r=r'$ and $v=v'w$.  Therefore
$$e=(\lambda,r,1, v)=(\lambda,r', 1, v'w)=e'.w.$$
From the second given equality we have
$$(w. \bar{u}, \bar{r}, 1, \beta)=w.f.\beta=f'. \beta'=(\bar{u}', \bar{r}',1, \beta'),$$
from which we get that $\beta=\beta'$, $\bar{r}=\bar{r}'$ and $w \bar{u}=\bar{u}'$. Therefore
$$w.f=(w. \bar{u}, \bar{r}, 1, \lambda)=(\bar{u}', \bar{r}',1, \lambda)=f'.$$
Now Lemma \ref{unamb} implies that $[e,f]=[e',f']$. 
\end{proof}
\begin{corollary}
Every 2-cell $[g,h]$ can be written uniquely as $[g,h]=u. [e, f].v$ where $e$ is a left reduced positive edge and $f$ a right reduced positive edge.
\end{corollary}
\begin{proof}
Indeed, there is some $u \in F$ such that $g=u.e$ with $e$ left reduced, and similarly there is some $v\in F$ such that $h=f.v$ and $f$ is right reduced, hence 
$$[g,h]=[u.e, f.v]=u.[e,f].v.$$
Lemma \ref{bmsq} implies that the above representation of $[g,h]$ is unique.
\end{proof}
This corollary implies that we can regard $C_{2}$ as a free $(\mathbb{Z}F, \mathbb{Z}F)$-bimodule with basis the set of all square 2-cells $[e,f]$ with $e$ left reduced and $f$ right reduced. Also it follows that $\partial_{2}$ is a $(\mathbb{Z}F, \mathbb{Z}F)$-bimodule homomorphism. Indeed, for every $u,v \in F$ and $[e,f]$ a square 2-cell with $e$ left reduced and $f$ right reduced we have that
\begin{align*}
\partial_{2}(u.[e,f].v)&= \partial_{2}([u.e,f.v])\\
&= u.e.(\iota(f.v)-\tau(f.v))-(\iota(u.e)-\tau(u.e)).f.v \tag{by (\ref{part2s})}\\
&=u.(e.(\iota f-\tau f) -(\iota e-\tau e).f).v\\
&=u. \partial_{2}([e,f]).v.
\end{align*}
Also $C_{1}$ is $(\mathbb{Z}F, \mathbb{Z}F)$-bimodule where the action of $\mathbb{Z}F$ on $C_{1}$ is induced by the action of $F$ on the edges of $\Gamma(\cP)$. It is obvious that in fact $C_{1}$ is a free $(\mathbb{Z}F, \mathbb{Z}F)$-bimodule with basis the set of all left and right reduced positive edges. Regarding $C_{0}$ as a $(\mathbb{Z}F, \mathbb{Z}F)$-bimodule with action inherited by the multiplication of $F$, it follows easily from (\ref{part1}) that $\partial_{1}$ is also $(\mathbb{Z}F, \mathbb{Z}F)$-bimodule homomorphism.

In the paper \cite{OK2002} of Otto and Kobayashi, a monoid presentation is denoted by $(\Sigma, R)$ and the rewriting rules of $R$ are denoted by $r \rightarrow \ell$. The edges of the graph of derivations in \cite{OK2002} are denoted by $(x,u,v,y)$ where $x,y \in \Sigma^{\ast}$ and $(u \rightarrow v) \in E=R \cup R^{-1}$. In \cite{OK2002} it is considered the set of closed paths
$$D=\{e_{1}xu_{2} \circ v_{1}xe_{2} \circ e_{1}^{-1}xv_{2} \circ u_{1}xe_{2}^{-1}| e_{1}=(u_{1},v_{1})\in R, e_{2}=(u_{2},v_{2})\in R, x \in \Sigma^{\ast}\}.$$
It is important to observe that each circuit of $D$ is in fact the boundary of a square 2-cell as the following shows
$$e_{1}xu_{2} \circ v_{1}xe_{2} \circ e_{1}^{-1}xv_{2} \circ u_{1}xe_{2}^{-1}=\partial[(\lambda,e_{1},1,\lambda),(x,e_{2},1,\lambda)],$$
where $e_{1}=(u_{1},v_{1})$ and $e_{2}=(u_{2},v_{2})$. The free $\mathbb{Z} \Sigma^{\ast}$ bi-module $\mathbb{Z} \Sigma^{\ast} \cdot R \cdot \mathbb{Z} \Sigma^{\ast}$ considered in \cite{OK2002} is the abelian group $C_{1}$ of our complex $\mathbf{C}(\mathcal{D})$, where each rewriting rule $r=(u,v) \in R$ is in our setting the positive edge $(\lambda, r, 1, \lambda)$. Also we see that the maps $\partial_{1}$ are the same in both papers. On the other hand, the free $\mathbb{Z} \Sigma^{\ast}$ bi-module $\mathbb{Z} \Sigma^{\ast} \cdot D \cdot \mathbb{Z} \Sigma^{\ast}$ of \cite{OK2002} is the abelian group $C_{2}$ of $\mathbf{C}(\mathcal{D})$, and the map $\partial_{2}$ of \cite{OK2002} is the same as the one defined in \cite{SES} and \cite{LDHM1}. Finally, the exact sequence of theorem 6.6 of \cite{OK2002} in our notations will be
$$\xymatrix{C_{2} \ar[r]^-{\partial_{2}} & J.R.\mathbb{Z} \Sigma^{\ast}+\mathbb{Z} \Sigma^{\ast}. R. J \ar[r]^-{\partial_{1}} & J^{2} \ar[r] & 0 .}$$
Even if we alter the definition of $\partial_{1}$ as suggested in \cite{sqcrev}, the above sequence is still exact.
The interpretation of the exactness in the middle of the above sequence is that $Ker \partial_{1} \cap (J.R.\mathbb{Z} \Sigma^{\ast}+\mathbb{Z} \Sigma^{\ast}. R. J)=Im \partial_{2}$.

\subsection{The 3-dimensional complex $(\mathcal{D}(\mathcal{P}),\mathbf{p})$-The general case} \label{gc}

In \cite{SES} it was necessary for several reasons to attach extra 2-cells to $\mathcal{D}(\cP)$ along an arbitrary set $\mathbf{p}$ of closed paths in $\Gamma(\cP)$ none of which is the boundary of a square 2-cell. For every $w,w' \in F$ and every $p \in \mathbf{p}$ a 2-cell $[w,p,w']$ is attached along the closed path
$$\partial [w,p,w']=w.p.w'.$$
So the set of the new 2-cells is in a 1-1 correspondence with triples $[w,p,w']$. We have in particular that when $w=w'=\lambda$, $[\lambda, p,\lambda]$ is the 2-cell attached along the closed path $p$. The extended 2-complex will be written in what follows as $\mathcal{D}(\mathcal{P}) \sqcup F.\mathbf{p}.F$. The bi action of $F$ on the cells of $\mathcal{D}(\cP)$ now extends to the 2-cells of $\mathcal{D}(\mathcal{P}) \sqcup F.\mathbf{p}.F$ by setting for each new 2-cell $[w,p,w']$ and each $v,v'\in F$
$$v.[w,p,w'].v'=[vw,p,w'v'].$$
In particular we have that 
$$[w,p,w'] =w.[\lambda, p, \lambda].w'$$
and that this expression is unique.

The construction of $(\mathcal{D}(\mathcal{P}),\mathbf{p})$ is then completed by adding 3-cells as follows. For each positive edge $f$ and each 2-cell $\sigma$ with $\partial \sigma= e_{1}^{\varepsilon_{1}}\circ \dots \circ e_{n}^{\varepsilon_{n}}$,  3-cells $[f,\sigma]$ and $[\sigma, f]$ are attached to the 2-skeleton by mapping their boundaries to respectively: 
\begin{itemize}
	\item [(1)] the 2-cells $\iota f. \sigma$, $\tau f. \sigma$ together with 2-cells $[f,e_{i}]$ for $1 \leq i \leq n$,
	\item[(2)] the 2-cells $\sigma. \iota f$, $\sigma. \tau f$ together with 2-cells $[e_{i},f]$ for $1 \leq i \leq n$.
\end{itemize}
The 2-sided action of $F$ on the 2-skeleton extends naturally to the 3-cells. For $[f,\sigma]$, $[\sigma, f]$ and $u, v \in F$,
$$u. [f, \sigma].v= [u.f, \sigma. v] \text{ and } u. [\sigma, f].v=[u.\sigma, f.v].$$

That $(\mathcal{D}(\mathcal{P}),\mathbf{p})$ is a 3-dimensional cw-complex with 2-skeleton $\mathcal{D}(\mathcal{P}) \sqcup F.\mathbf{p}.F$ follows from \cite{sqcrev}. The complex $\mathbf{C}(\mathcal{D}(\cP))$ now extends to
$$\xymatrix{\mathbf{C}(\mathcal{D}(\mathcal{P}), \mathbf{p}): 0 \ar[r] & C_{3}^{\mathbf{p}} \ar[r]^-{\tilde{\partial}_{3}} & C_{2}^{\mathbf{p}} \oplus C_{2} \ar[r]^-{\tilde{\partial}_{2}} & C_{1} \ar[r]^{\partial_{1}} & C_{0} \ar[r] & 0}$$
where $C_{3}^{\mathbf{p}}$ is the free abelian group generated by the set of all 3-cells, and $C_{2}^{\mathbf{p}}$ is the free abelian group generated by the set of all newly added 2-cells $[u,p,v]$. In fact $C_{2}^{\mathbf{p}}$ is a free $(\mathbb{Z}F, \mathbb{Z}F)$-bimodule with basis the set of 2-cells $[\lambda, p, \lambda]$ with $p \in \mathbf{p}$. This is a consequence of the fact that every 2-cell $[w, p,w']$ is uniquely written as $w.[\lambda, p, \lambda].w'$. 

The boundary map $\tilde{\partial}_{2}$ restricted to $C_{2}$ is $\partial_{2}$, and for every other 2-cell $[w,p, w']$ with $p \in \mathbf{p}$ and $\partial p= f_{1}^{\delta_{1}} \circ \dots \circ f_{n}^{\delta_{n}}$, it is defined in \cite{SES} (and also in \cite{sqcrev})
\begin{equation} \label{part2p}
\tilde{\partial}_{2}([w,p ,w'])=\sum_{i=1}^{n}\delta_{i}w.f_{i}.w'.
\end{equation}
This implies that $\tilde{\partial}_{2}$ is a $(\mathbb{Z}F, \mathbb{Z}F)$-bimodule homomorphism since for every 2-cell $[w,p ,w']$ and every $v,v' \in F$ we have that
\begin{align*}
\tilde{\partial}_{2}(v.[w,p ,w'].v')&=\tilde{\partial}_{2}([vw,p ,w'v'].)\\
&=\sum_{i=1}^{n}\delta_{i}vw.f_{i}.w'v'\\
&=v.\left(\sum_{i=1}^{n}\delta_{i}w.f_{i}.w'\right).v'\\
&=v. \tilde{\partial}_{2}([w,p,w']).v'.
\end{align*}
Finally, the definition of $\tilde{\partial}_{3}$ is done in the following way. For every positive edge $f$ and every 2-cell $\sigma$ with $\tilde{\partial}_{2}\sigma= \sum_{i=1}^{n}\varepsilon_{i} e_{i}$ we have defined in \cite{sqcrev}
\begin{equation} \label{d3a}
\tilde{\partial}_{3}[f,\sigma]= (\tau f - \iota f).\sigma- \sum_{i=1}^{n}\varepsilon_{i}[f,e_{i}],
\end{equation}
and
\begin{equation} \label{d3b}
\tilde{\partial}_{3}[\sigma, f]= \sigma.(\tau f- \iota f)+\sum_{i=1}^{n}\varepsilon_{i}[e_{i},f].
\end{equation}
We remark here that our definition of cellular boundary map at this level differs from the one given is \cite{SES} only by the sign and that all the results of \cite{SES} which are used in the present paper remain valid after this sign alteration.

\subsection{A short exact sequence for $(\mathcal{D}(\mathcal{P}),\mathbf{p})$} \label{ses-ses}

Assume now that $(\mathcal{D}(\mathcal{P}), \mathbf{p})$ is the 3-complex of section \ref{gc} associated with an arbitrary monoid presentation $\cP=\langle \mathbf{x}, \mathbf{r} \rangle$ and $\mathbf{p}$ is a set of closed paths in $\Gamma(\cP)$ none of which is the boundary of a square 2-cell. Since $C_{2}^{\mathbf{p}}$ can be regarded as a free $(\mathbb{Z}F, \mathbb{Z}F)$-bimodule with basis $\{[\lambda, p,\lambda]| p \in \mathbf{p}\}$ we can define a $(\mathbb{Z}F, \mathbb{Z}F)$-homomorphism
$$\varphi: C_{2} \oplus C_{2}^{\mathbf{p}} \rightarrow \mathbb{Z}S.\mathbf{p}. \mathbb{Z}S$$
by mapping $C_{2}$ to 0, and every 2-cell $[u,p, v]$ to $\bar{u}.p. \bar{v}$. The kernel of $\varphi$ is denoted by $K^{\mathbf{p}}$. It follows from \cite{SES} that 
$$K^{\mathbf{p}}=C_{2} + J.\mathbf{p}. \mathbb{Z}F+ \mathbb{Z}F. \mathbf{p}. J.$$
Also it is shown that $B_{2}(\mathcal{D}(\mathcal{P}),\mathbf{p}) \subseteq K^{\mathbf{p}}$ and that the restriction of $\tilde{\partial}_{2}$ on $K^{\mathbf{p}}$ sends $K^{\mathbf{p}}$ onto $B_{1}(\mathcal{D}(\mathcal{P}))$, therefore we have the complex
\begin{equation} \label{cxkp}
	\xymatrix{\mathcal{K}_{\mathbf{p}}:0 \ar[r] & B_{2}(\mathcal{D}(\mathcal{P}),\mathbf{p}) \ar[r]^-{incl.} & K^{\mathbf{p}} \ar[r]^-{\tilde{\partial}_{2}} & B_{1}(\mathcal{D}(\mathcal{P}))  \ar[r] & 0}
\end{equation}
It is proved in Proposition 14 of \cite{SES} that when $\mathbf{p}$ is any homology trivializer, then the sequence (\ref{cxkp}) is exact. The proof goes in several stages. In the first stage it is proved in Lemma 13 that for any pair $(\mathbf{p},\mathbf{q})$ of homology trivializers for $\mathcal{D}(\cP)$, the exactness of $\mathcal{K}_{\mathbf{p}}$ is equivalent to that of $\mathcal{K}_{\mathbf{q}}$. In the second stage (Proposition 14-the special case) it is proved that if the monoid presentation is complete (all the critical pairs resolve and the system is terminating), and $\mathbf{p}$ is obtained by choosing resolutions of all the critical pairs, then (\ref{cxkp}) is exact. In the last stage, the general case of Proposition 14 is derived as an implication of the results of the first two stages. We have decided to write our own proof of Proposition 14 in full for reasons explained in sections \ref{1-st}, \ref{2-st} and \ref{3-st}. In order to make the proof clear we have added two intermediate sections. Section \ref{exsqc} gives the description of $(\mathcal{D}(\mathcal{P}),\mathbf{p})$ when the presentation is complete and Section \ref{osqc} where a strict and terminating order on 0,1 and 2-cells is given which is then used to prove the special case of Proposition 14.

\subsection{The 3-dimensional complex $(\mathcal{D}(\mathcal{P}),\mathbf{p})$-A special case} \label{exsqc}

In this section we will consider the 3-complex $(\mathcal{D}(\mathcal{P}), \mathbf{p})$ as defined in section \ref{gc} but when the presentation $\cP=\langle \mathbf{x}, \mathbf{r} \rangle$ is complete, which means in particular that all the critical pairs resolve and $\mathbf{p}$ arises from choosing resolutions of all the critical pairs of $\mathbf{r}$. We will give below the definition of critical pairs.
\begin{definition} \label{critpairs}
A pair $(e, f)$ of positive edges with the same initial forms a critical pair if either
\begin{itemize}
	\item [(1)] One of the pair is both left and right reduced (a critical pair of inclusion type), or
	\item[(2)] One of the pair is left reduced but not right reduced, the other is right reduced but not left reduced and they are not disjoint (a critical pair of overlapping type).
\end{itemize}
In the first case $e=(u, r, 1, v)$, $f=(\lambda, s, 1, \lambda)$ and $s_{+}=ur_{+}v$. In the second case, $e=(\lambda, r, 1, v)$, $f=(u, s, 1, \lambda)$ and $r_{+}v=us_{+}$ where $r_{+}$ overlaps with $s_{+}$ and both $u,v \neq \lambda$, otherwise we would have a pair as in the first case. We say that a critical pair $(e,f)$ is resolvable if:
\begin{itemize}
\item[(i)] $\tau e=\tau f$, or
\item[(ii)] $\tau e \neq \tau f$ and at least one of the following occur.
\begin{itemize}
\item[(ii-1)] There is a positive path $\mathbf{l}_{(e,f)}$ from $\tau e$ to $\tau f$, or
\item[(ii-2)] There is is a positive path $\mathbf{r}_{(f,e)}$ from $\tau f$ to $\tau e$, or
\item[(ii-3)] There are positive paths $\lambda_{(e,f)}$ and $\rho_{(e,f)}$ from $\tau e$ and $\tau f$ respectively to a common vertex $w$ which is distinct from $\tau f$ and $\tau e$.
\end{itemize}
\end{itemize}
\end{definition}
Let now $\mathfrak{p}$ be the set of all the critical pairs of $\mathbf{r}$. For every pair $(e,f) \in \mathfrak{p}$, we attach a 2-cell $\sigma$ in the following way. If the pair is of type (i), then the corresponding 2-cell is attached along the closed path
$$\partial \sigma= e \circ f^{-1}.$$
If the pair is of type (ii-1), then $\sigma$ is attached along the path
$$\partial \sigma= e \circ \mathbf{l}_{(e,f)} \circ f^{-1}.$$
If the pair is of type (ii-2), then $\sigma$ is attached along the path
$$\partial \sigma= e \circ \mathbf{r}_{(f,e)}^{-1} \circ f^{-1},$$
and if the pair is of type (ii-3), then $\sigma$ is attached along the path
$$\partial \sigma=e \circ \lambda_{(e,f)} \circ \rho_{(e,f)}^{-1} \circ f^{-1}.$$
In each case $\partial \sigma$ will be called the boundary of $\sigma$. 

Also we attach 2-cells $[u, \sigma, v]$ for every $u,v \in F$ along the loop $u. \partial \sigma. v$. This 2-cell is called the translate of $\sigma$ by $u$ and $v$. Letting $\mathbf{p}$ be the set of 2-cells attached as above for every $(e,f) \in \mathfrak{p}$, and $F.\mathbf{p}.F$ the set of their translates, the resulting complex will be denoted throughout by $\mathcal{D}(\mathcal{P}) \sqcup F.\mathbf{p}.F$. We remark again that in \cite{sqcrev} we have proved that $\mathcal{D}(\mathcal{P}) \sqcup F.\mathbf{p}.F$ is indeed a cw-complex of dimension 2 which is an extension of $\mathcal{D}(\mathcal{P})$. Regarding 2-cells of $\mathcal{D}(\mathcal{P}) \sqcup F.\mathbf{p}.F$ we prove below a number of useful results.
\begin{lemma} \label{unambp}
If $\sigma, \sigma' \in \mathbf{p}$ and $\alpha, \beta, \alpha', \beta' \in F$, then $[\alpha, \sigma, \beta]=[\alpha', \sigma', \beta']$ if and only if $\alpha=\alpha', \beta=\beta'$ and $\sigma=\sigma'$.
\end{lemma}
\begin{proof}
We first exclude the case when $\sigma$ arises from the resolution of a critical pair of inclusion type, let say $e=(u, r, 1, v)$ and $f=(\lambda,s,1,\lambda)$ with $s_{+}=u r_{+}v$, while $\sigma'$ arises from a critical pair of overlapping type like $e'=(\lambda, r', 1, v')$ and $f'=(u',s',1,\lambda)$ where $r'_{+}v'=u's'_{+}$ and $r'_{+}$ overlap with $s'_{+}$ but $u', v' \neq \lambda$. Indeed, equating the boundaries of $[\alpha, \sigma, \beta]$ and $[\alpha', \sigma', \beta']$ we have
$$(\alpha u, r, 1, v \beta)=\alpha. e. \beta=\alpha'. e'. \beta'= (\alpha', r' , 1, v'\beta'),$$
which implies $\alpha u=\alpha'$, and since
$$(\alpha, s, 1, \beta)=\alpha. f. \beta= \alpha'.f'. \beta'=(\alpha' u', s', 1, \beta'),$$
we obtain that $\alpha=\alpha'u'$. It follows that $\alpha=\alpha' u'=\alpha u u'$, which is impossible since $u' \neq 1$. 
There are now two possible cases remaining for $\sigma, \sigma'$. First, they both arise from resolving critical pairs of inclusion type. Let $e=(u, r, 1, v)$ and $f=(\lambda,s,1,\lambda)$ with $s_{+}=u r_{+}v$ be the edges that resolve to create $\sigma$, and $e'=(u', r', 1, v')$ and $f'=(\lambda,s',1,\lambda)$ with $s'_{+}=u' r'_{+}v'$ be the edges that resolve to create $\sigma'$. Since $[\alpha, \sigma, \beta]=[\alpha', \sigma', \beta']$, then
$$(\alpha, s, 1, \beta)=\alpha.f. \beta=\alpha'. f'. \beta'=(\alpha', s', 1, \beta'),$$
from which we have that $\alpha=\alpha', \beta=\beta', s=s'$. But we also have that
$$(\alpha u, r, 1, v \beta)=\alpha. e. \beta=\alpha'. e'. \beta'=(\alpha' u', r', 1, v' \beta'),$$
which imply that $r=r'$, $u=u'$ and $v=v'$. So we get that $e=e'$ and $f=f'$. The second case is when both 2-cells $\sigma, \sigma'$ arise from overlapping pairs. Let $e=(\lambda, r, 1, v)$ and $f=(u,s,1,\lambda)$ where $r_{+}v=us_{+}$ and $r_{+}$ overlap with $s_{+}$ be edges that generate $\sigma$. Similarly, $e'=(\lambda, r', 1, v)$ and $f=(u',s',1,\lambda)$ where $r'_{+}v'=u's'_{+}$ and $r'_{+}$ overlap with $s'_{+}$ be edges that generate $\sigma'$. From the condition that $[\alpha, \sigma, \beta]=[\alpha', \sigma', \beta']$, we have that
$$(\alpha, r, 1, v\beta)=\alpha. e. \beta= \alpha' . e'. \beta'=(\alpha', r', 1, v'\beta'),$$
from which we infer that $\alpha=\alpha'$, $r=r'$ and $v\beta=v'\beta'$. Also we have that
$$(\alpha u, s, 1, \beta)=\alpha. f. \beta=\alpha'. f'. \beta'=(\alpha'u', s', 1, \beta'),$$
which implies that $s=s'$, $\alpha u=\alpha' u'$ and $\beta=\beta'$. After cancelations we see that $u=u'$ and $v=v'$, therefore $e=e'$ and $f=f'$, hence $\sigma=\sigma'$. 
\end{proof}
\begin{corollary} \label{cunambp}
If $[\alpha,\sigma,\beta]$ is a translate of some $\sigma \in \mathbf{p}$ which arises from the resolution of a critical pair $(e,f)$ of positive edges with the same initial $\eta$, then $[\alpha,\sigma,\beta]$ is uniquely determined by the pair $\text{st}([\alpha,\sigma,\beta])=(\alpha.e.\beta, \alpha.f.\beta)$ which is called the star of $[\alpha,\sigma,\beta]$.
\end{corollary}
\begin{proof}
If $[\alpha',\sigma',\beta']$ is the translate of another $\sigma' \in \mathbf{p}$ which arises from the resolution of the critical pair $(e',f')$ such that $\alpha.e.\beta=\alpha'.e'.\beta'$ and $\alpha.f.\beta=\alpha'.f'.\beta'$, then from the proof of Lemma \ref{unambp} we see that $\alpha=\alpha'$, $\beta=\beta'$, $e=e'$ and $f=f'$. This shows that $[\alpha,\sigma,\beta]=[\alpha',\sigma',\beta']$.
\end{proof}
\begin{notation} \label{nsi}
If $[\alpha,\sigma,\beta]$ is a translate of some $\sigma \in \mathbf{p}$ which arises from the resolution of a critical pair $(e,f)$ of positive edges with the same initial $\eta$, then we will denote the 2-cell $[\alpha,\sigma,\beta]$ by $[\alpha \eta \beta, (\alpha.e.\beta, \alpha.f.\beta)]$. The above corollary implies that this notation is unambiguous.
\end{notation}
\begin{remark}
Since a 2-cell $[u, \sigma, v]$ with $\sigma \in \mathbf{p}$ is uniquely determined by its own boundary $u.\partial \sigma .v$ (in fact by its own star), we usually write it as $u.\sigma.v$. When $u=\lambda$, this 2-cell is written as $\sigma.v$, when $v=\lambda$ it is written as $u.\sigma$ and when both $u=v=\lambda$, the 2-cells is just as $\sigma$.
\end{remark}
Finally we prove in the subsequent lemma that no 2-cell can be at the same time square and arising from another set of closed paths as those described in section \ref{gc} or in the present section. 
\begin{lemma} \label{sqneqp}
\begin{itemize}
\item[(i)] The set of square 2-cells and $F.\mathbf{p}.F$ are disjoint whenever $\mathbf{p}$ arises from the resolutions of a set $\mathfrak{p}$ of critical pairs.
\item[(ii)] The set of square 2-cells and $F.\mathbf{p}.F$ are disjoint whenever $\mathbf{p}$ is a set of closed paths as in section \ref{gc}.
\end{itemize}
\end{lemma}
\begin{proof}
(i) Assume that there is some $u.\sigma.v \in F.\mathbf{p}.F$ which is square. If $\sigma$ is obtained by resolving a critical pair $(e,f)$, then by Lemma \ref{cunambp} the pair $(u.e.v, u.f.v)$ uniquely determines $u.\sigma.v$. Since this 2-cell is also square, let say $[g,h]$, then from Corollary \ref{cunamb} it is uniquely determined by the pair $(g.\iota(h), \iota(g).h)$ where the edges transform disjoint subwords of $\iota(g)\iota(h)$ unlike the edges of $(u.e.v, u.f.v)$. This is a contradiction.\\
(ii) Assume that for some $p \in \mathbf{p}$ and $u,v \in F$, $[u,p,v]=[e,f]$ where $e,f$ are positive edges. Letting $\partial p= e^{\varepsilon_{1}}_{1} \circ \dots \circ e^{\varepsilon_{n}}_{n}$, it follows that
$$u. e^{\varepsilon_{1}}_{1}.v \circ \dots \circ u.e^{\varepsilon_{n}}_{n}.v= e.\iota f \circ \tau e. f \circ e^{-1}.\tau f \circ \iota e. f^{-1},$$
consequently $n=4$ and 
\begin{align*}
u. e^{\varepsilon_{1}}_{1}.v&=e.\iota f\\
u. e^{\varepsilon_{2}}_{2}.v&=\tau e. f\\
u. e^{\varepsilon_{3}}_{3}.v&=e^{-1}.\tau f\\
u. e^{\varepsilon_{4}}_{4}.v&=\iota e. f^{-1}.
\end{align*}
The first two equations imply that there are positive edges $e'$ and $f'$ such that $e=u.e'$ and $f=f'.v$, hence 
\begin{align*}
u. e^{\varepsilon_{1}}_{1}.v&=u.e'.\iota f'.v\\
u. e^{\varepsilon_{2}}_{2}.v&=u.\tau e'. f'.v\\
u. e^{\varepsilon_{3}}_{3}.v&=u.e'^{-1}.\tau f'.v\\
u. e^{\varepsilon_{4}}_{4}.v&=u.\iota e'. f'^{-1}.v.
\end{align*}
It follows from these that $\partial p=\partial[e',f']$ which contradicts the assumption on $\mathbf{p}$ made in section \ref{gc}.
\end{proof}

\subsection{Ordering the Squier complex} \label{osqc}

The aim of this section is to describe the cellular boundary maps of $(\mathcal{D}(\mathcal{P}), \mathbf{p})$ in our special case. To do that we need first to establish a terminating strict order on the $n$-cells (with $ n=0,1,2$) of the complex. The order on $n$-cells (with $ n=0,1,2$) which is given in \cite{SES} is based on several strict orders which are frequently used in the theory of rewriting systems. For the sake of completeness we will present them here in some details. Let $(X, <_{1})$ and $(Y,<_{2})$ be two strict orders. Then in the direct product $X \times Y$ we can define a strict order $<_{(1,2)}$, called the lexicographical product, by setting
$$(x',y') <_{(1,2)} (x,y) \Leftrightarrow (x'<_{1}x) \vee (x'=x \wedge y'<_{2} y).$$
It is proved in Lemma 2.4.1 of \cite{BN} that $<_{(1,2)}$ is indeed a strict order. Also if $<_{1}$ and $<_{2}$ are terminating, then so is $<_{(1,2)}$. This is Theorem 2.4.2 of \cite{BN}. We can define the dual version $<_{(2,1)}$ of the lexicographical product by setting 
$$(x',y') <_{(2,1)} (x,y) \Leftrightarrow (y'<_{2}y) \vee (y'=y \wedge x'<_{1} x).$$
We call $<_{(2,1)}$ the colexicographical product of $(X, <_{1})$ and $(Y,<_{2})$. The analogues of Lemma 2.4.1 and Theorem 2.4.2 of \cite{BN} for $<_{(2,1)}$ also hold true. 

Now let $(B, <_{B})$ be a strict order, $A$ a nonempty set and $f:A \rightarrow B$ a map. We define a relation $<_{A}$ on $A$ by setting
$$a' <_{A} a \Leftrightarrow f(a')<_{B} f(a).$$
It is easy to check that $<_{A}$ is a strict order and that it is terminating whenever $<_{B}$ is terminating. 

Finally, if $(X_{i}, <_{i})$ with $i \in I$ is a disjoint family of strict orders, then in $\cup_{i \in I}X_{i}$ we can define a new relation $<$ by setting 
$$x < y \Leftrightarrow \exists i \in I, \text{ such that } x,y \in X_{i} \text{ and } x<_{i} y.$$
It is easy to see that $(X, <)$ is a strict order and that it is terminating whenever each $<_{i}$ is such. The strict order $(X, <)$ is commonly known as the coproduct of all $(X_{i}, <_{i})$.

Returning to our complex $(\mathcal{D}(\mathcal{P}), \mathbf{p})$, since the system $\cP=\langle \mathbf{x}, \mathbf{r}\rangle$ is complete, we have that for every $(r_{+},r_{-}) \in \mathbf{r}$, $r_{+} \neq r_{-}$. Let $\vartriangleright$ be a well ordering on $\mathbf{x}$. The corresponding length-lexicographical ordering on $F$ is defined as follows. For $u,v \in F$, we write $u >_{llex} v$ if and only if $|u|>|v|$, or $|u|=|v|$, $u=au'$, $v=bv'$ where $a,b \in \mathbf{x}$, $u',v' \in F$, and one of the following holds:
\begin{itemize}
	\item [(i)] $a \vartriangleright b$,
	\item[(ii)] $a=b$, $u' >_{llex} v'$.
\end{itemize}
It turns out that $>_{llex}$ is a well ordering on $F$ (see \cite{Book+Otto}). We can always assume that $>_{llex}$ is compatible with $\mathbf{r}$ in the sense that $r_{+} >_{llex} r_{-}$, for if there are rules $(r_{+},r_{-})$ satisfying the opposite, we can exchange $r_{+}$ with $r_{-}$. Well orderings in $F$ that are compatible with $\mathbf{r}$ are usually called reduction well-orderings and are the basis for starting the Knuth-BendixBendix completion procedure. 

So far we have defined a reduction well order on the 0-skeleton of $\mathcal{D}(\mathcal{P})$ which will be denoted by $\prec_{0}$. This order induces a terminating strict order in the set of positive edges of $\mathcal{D}(\mathcal{P})$ in a way which will be described in details below. Consider three strict orders on the set $F$. The first strict order is $(F, \prec_{0})$ which is already defined. The second one is $(F, <_{\ell})$ where $<_{\ell}$ is defined by
$$u <_{\ell} v \Leftrightarrow |u|< |v|.$$
It is obvious that $<_{\ell}$ is a terminating strict order. The third strict order is $(F, <_{s})$ where $<_{s}$ is defined by 
$$u<_{s}v \Leftrightarrow u=\alpha v \text{ with } \alpha \in F\setminus \{\lambda\}.$$
We check that $<_{s}$ is transitive. Indeed, if $u<_{s}v$ and $v<_{s} w$, then there are words $\alpha, \beta \in F\setminus \{\lambda\}$ such that 
$$u=\alpha v \text{ and } v=\beta w,$$
hence
$$u=\alpha \beta w,$$
which implies that $u<_{s} w$. To see that $<_{s}$ is terminating we observe that if $u<_{s}v$, then $|u|>|v|$. Therefore any infinite sequence
$$u_{1} <_{s} u_{2} <_{s} \dots <_{s} u_{n} <_{s} \dots$$
induces an infinite sequence 
$$|u_{1}| > |u_{2}| > \dots > |u_{n}| > \dots$$
of numbers in $\mathbb{N} \cup \{0\}$ which is impossible.
Now consider the lexicographical product $(F \times F, <_{(\ell, 0)})$ of $(F, <_{\ell})$ with $(F, \prec_{0})$ which is given by
\begin{align*}
(u,v) <_{(\ell,0)} (u',v') &\Leftrightarrow ((u <_{\ell} u' ) \vee (u=u' \wedge v\prec_{0} v'))\\
& \Leftrightarrow  ((|u| < |u'| ) \vee (u=u' \wedge v\prec_{0} v')).
\end{align*}
Finally we consider the colexicographical product $((F\times F) \times F, <_{(s, (\ell,0))})$ of $(F \times F, <_{(\ell, 0)})$ with $(F, <_{s})$ which is given by
{\small
\begin{align*}
((u,v),w) <_{(s, (\ell,0))} ((u',v'),w') &\Leftrightarrow ((w<_{s} w') \vee (w'=w \wedge (u,v) <_{(\ell,0)} (u',v')))\\
& \Leftrightarrow ((w<_{s} w') \vee (w=w' \wedge (((|u| < |u'| ) \vee (u=u' \wedge v\prec_{0} v')))).
\end{align*}
}
Now we are ready to define a terminating strict order $\prec_{1}$ on the set of positive edges of $\mathcal{D}(\mathcal{P})$. For each $w \in F$ we let $E_{w}$ be the set of all positive edges with initial vertex $w$. Let $\prec_{1,w}$ be the relation on $E_{w}$ defined as follows. For any pair of edges $e=(u,r,+1,v)$ and $f=(u',r',+1,v')$ from $E_{w}$ regarded as 4-tuples $e=(u,(r_{+}, r_{-}),v)$ and $f=(u',(r'_{+}, r'_{-}),v')$ we define 
$$(u,(r_{+}, r_{-}),v) \prec_{1,w} (u',(r'_{+}, r'_{-}),v') \Leftrightarrow ((r_{+}, r_{-}),v) <_{(s, (\ell,0))} ((r'_{+}, r'_{-}),v').$$
In other words we have that $e \prec_{1,w} f$ if and only if one of the following occurs:
\begin{itemize}
	\item[(i)] $v'$ is a proper suffix of $v$, or
	\item[(ii)] $v=v'$ and $|r_{+}|< |r'_{+}|$, or
	\item[(iii)] $v=v'$, $r_{+} = r'_{+}$ and $r_{-} \prec_{0} r'_{-}$.
\end{itemize}
It follows that $(E_{w},\prec_{1,w})$ is a terminating strict order. Now we let $E=\cup_{w \in F} E_{w}$ and $(E, \prec_{1})$ be the corresponding coproduct of the strict orders $(E_{w},\prec_{1,w})$ which is a terminating strict order. 

We remark that any pair of edges with the same initial are always comparable with respect to $\prec_{2}$. Now we extend the orders $\prec_{0}$ and $\prec_{1}$ to the 2-skeleton of the 3-complex $(\mathcal{D}(\mathcal{P}), \mathbf{p})$ as follows. First, if a 2-cell $\sigma$, square or arising from the resolution of a critical pair, is attached along the loop $\partial \sigma=e \circ \lambda \circ \rho^{-1} \circ f^{-1}$, we let $\lambda_{\sigma}=e$ and $\rho_{\sigma}=f$. Remark here that $\lambda_{\sigma} \prec_{1} \rho_{\sigma}$. Also we note that the biggest 0-cell of the 0-skeleton of $\sigma$ is $\iota e=\iota f$, which will be denoted by $w_{\sigma}$, and every vertex represented in $\lambda$ or $\rho$ is strictly less than $w_{\sigma}$. This is an obvious consequence of the fact that $\prec_{0}$ is compatible with $\mathbf{r}$ which means that in $\Gamma(\cP)$ there are no non trivial positive loops of edges. 
It is possible to define a strict order on the set of all 2-cells of $\mathcal{D}(\mathcal{P}) \sqcup F.\mathbf{p}.F$ solely in terms of $\prec_{0}$ and $\prec{1}$.

Let $A$ be the set of all 2-cells of $\mathcal{D}(\mathcal{P}) \sqcup F.\mathbf{p}.F$, $(F, \prec_{0})$ the strict order defined on $F$ and if $E$ is the set of positive edges of $\mathcal{D}(\cP)$, we let $(E, \prec_{1})$ be the respective strict order defined above. Then we have the colexicographical product $(E \times E, <_{E\times E})$ of $(E, \prec_{1})$ with $(E, \prec_{1})$ and the lexicographical product $(F \times (E \times E), <_{F\times (E\times E)})$ of $(F, \prec_{0})$ with the colexicographical product $(E \times E, <_{E \times E})$ above. Now we define a map
$$f: A \rightarrow F \times (E \times E)$$
as follows. For every 2-cell $\sigma$ with $\text{st}(\sigma)= (\lambda_{\sigma}, \rho_{\sigma})$ and with $w_{\sigma}=\iota \lambda_{\sigma}= \iota \rho_{\sigma}$ we let
$$f(\sigma)=(w_{\sigma}, (\lambda_{\sigma}, \rho_{\sigma})).$$
Now the strict order $<_{F \times (E \times E)}$ induces via $f$ the strict order $\prec_{2}$ by
$$\sigma_{2} \prec_{2} \sigma_{1} \Leftrightarrow (w_{\sigma_{2}}, (\lambda_{\sigma_{2}}, \rho_{\sigma_{2}})) <_{F \times (E \times E)} (w_{\sigma_{1}}, (\lambda_{\sigma_{1}}, \rho_{\sigma_{1}})).$$
We can rephrase the above by saying that $\sigma_{2} \prec_{2} \sigma_{1}$ if and only if:
\begin{itemize}
	\item [(i)] $w_{\sigma_{2}} \prec_{0} w_{\sigma_{1}}$; or
	\item[(ii)] $w_{\sigma_{2}} = w_{\sigma_{1}}$ and $\rho_{\sigma_{2}}\prec_{1} \rho_{\sigma_{1}}$; or
	\item[(iii)] $w_{\sigma_{2}} = w_{\sigma_{1}}$, $\rho_{\sigma_{2}}=\rho_{\sigma_{1}}$ and $\lambda_{\sigma_{2}}\prec_{1} \lambda_{\sigma_{1}}$. 
\end{itemize}
This is a terminating strict order on the set of all 2-cells of $(\mathcal{D}(\mathcal{P}), \mathbf{p})$. 

Under the current assumptions, we show below how the ordering of 2-cells is reflected on the boundary maps of 2-cells and 3-cells. Before we recall that every 2-cell $\sigma$ of $(\mathcal{D}(\mathcal{P}), \mathbf{p})$, square or a translate of one arising from the resolution of a critical pair, is attached along a loop $\partial \sigma=e \circ \lambda \circ \rho^{-1} \circ f^{-1}$ where $(e,f)=\text{st}(\sigma)$, $e \prec_{1} f$, $\lambda$ and $\rho$ are positive or trivial paths whose vertices are strictly below $w_{\sigma}$. In algebraic terms this means that if the length $n$ of $\partial \sigma$ (the number of edges represented in $\partial \sigma$) is at least 3, then
\begin{equation} \label{par2-3}
\partial_{2}\sigma= e-f +\sum^{n}_{i=3} \varepsilon_{i}e_{i}
\end{equation}
where for each $i \geq 3$, $\iota e_{i} \prec_{0} w_{\sigma}$. And if it happens that both $\lambda$ and $\rho$ are the constant paths at $\tau e$ (and also $\tau f$), then
\begin{equation} \label{par2-2}
\partial_{2}\sigma= e-f .
\end{equation}
In the case when the 3-cell is $[e, \sigma]$ where $\text{st}(\sigma)=(f,g)$, then from (\ref{d3a}), (\ref{par2-3}) and (\ref{par2-2}) we have
\begin{equation*}
\tilde{\partial}_{3}[e, \sigma]= (\tau e - \iota e).\sigma-([e,f]-[e,g]+\xi)
\end{equation*}
where $\xi$ is either zero, or it is a 2-chain made of square 2-cells whose maximal vertex is less than $w=\iota e. w_{\sigma}$. There are only three 2-cells in the above expression whose maximal vertex is $w$ and therefore bigger than every other 2-cell represented there. They are 
$$\iota e. \sigma, [e,f] \text { and } [e,g].$$ 
Since 
$$\iota e. f \prec_{1} \iota e. g \text{ and } e. \iota g \prec_{1} \iota e. f,$$
it follows that
\begin{equation} \label{mb1}
[e, f] \prec_{2} [e, g]\prec_{2} \iota e. \sigma,
\end{equation}
and then
\begin{equation} \label{d3pl}
\tilde{\partial}_{3}[e, \sigma]=-\iota e. \sigma+[e, g]-[e, f]+\varsigma
\end{equation} 
where $\varsigma$ is a 2-chain made up of 2-cells all of which have maximal vertices less than $w$. This and (\ref{mb1}) imply that the 2-chain
\begin{equation} \label{stl1}
[e, g]-[e, f]+\varsigma \text{ is made of 2-cells strictly lesser than } \iota e. \sigma.
\end{equation}
In the second case when the 3-cell is $[\sigma, e]$ where $\text{st}(\sigma)=(f,g)$, then from (\ref{d3b}), (\ref{par2-3}) and (\ref{par2-2}) we have 
\begin{equation*}
\tilde{\partial}_{3}[\sigma, e]= \sigma.(\tau e - \iota e)+([f, e]-[g,e]+\xi)
\end{equation*}
where $\xi$ is either zero, or it is a 2-chain made of square 2-cells whose maximal vertex is less than $w=w_{\sigma}.\iota e$. As in the first case, the 2-cells represented here whose maximal 2-cell is $w$ are
$$\sigma.\iota e, [f,e] \text{ and } [g,e].$$
Since
$$f . \iota e \prec_{1} g. \iota e \prec_{1} \iota g. e,$$
then
\begin{equation} \label{mb2}
\sigma.\iota e \prec_{2} [f, e] \prec_{2} [g,e]
\end{equation}
and then
\begin{equation} \label{d3pr}
\tilde{\partial}_{3}[\sigma, e]=-[g,e]+[f,e]-\sigma.\iota e+\varsigma
\end{equation}
where $\varsigma$ is a 2-chain made up of 2-cells all of which have maximal vertices less than $w$. This and (\ref{mb2}) imply that the 2-chain
\begin{equation} \label{stl2}
[f,e]-\sigma.\iota e+\varsigma \text{ is made of 2-cells strictly lesser than } [g,e].
\end{equation}

\subsection{The first stage-Lemma 13} \label{1-st}

We will give in Lemma \ref{l13ses} a more detailed proof of Lemma 13 which is largely based on the proof given in \cite{SES} but avoids several obscure arguments there. Before we prove Lemma \ref{l13ses} we give the following lemmas which hold true in the Squire complex associated with any monoid presentation.
\begin{lemma} \label{uu}
For each word $u \in F$, the homomorphism $\mathbf{u}: C_{1} \rightarrow C_{1}$ defined on free generators by $\mathbf{u}(f)=f.u$, is injective.
\end{lemma}
\begin{proof}
Let $\sum_{i} \varepsilon_{i}f_{i} \in \text{Ker}\mathbf{u}$ where $f_{i}$ is a positive edge and $\varepsilon_{i}=\pm 1$, hence $\sum_{i} \varepsilon_{i} f_{i}.u=0$. It follows that the terms will pair up into pairs $(\varepsilon_{i}f_{i}.u, \varepsilon_{j}f_{j}.u)$ such that $\varepsilon_{i}+\varepsilon_{j}=0$ and $f_{i}.u=f_{j}.u$. Cancelling out $u$ we get that $f_{i}=f_{j}$. This implies that $\sum_{i} \varepsilon_{i}f_{i}=0$.
\end{proof}
\begin{lemma} \label{zro}
For every $n_{1},...,n_{k} \in \mathbb{Z}^{\ast}$ and $f_{1},....,f_{2}$ positive edges, and every positive edge $e$, if $n_{1}f_{1}+ \dots + n_{k}f_{k}=0$, then $n_{1}[e, f_{1}]+ \dots + n_{k}[e, f_{k}]=0$, and similarly $n_{1}[f_{1}, e]+ \dots + n_{k}[ f_{k}, e]=0$.
\end{lemma}
\begin{proof}
There is a homomorphism of abelian groups $\mathbf{e}: C_{1} \rightarrow C_{2}$ induced by the assignment $f \mapsto [e,f]$ for every free generator $f$ of $C_{1}$, therefore
$$n_{1}[e, f_{1}]+ \dots + n_{k}[e, f_{k}]=\mathbf{e}(n_{1}f_{1}+ \dots + n_{k}f_{k})=\mathbf{e}(0)=0.$$
The second equality is proved in a similar fashion.
\end{proof}

\begin{lemma} \label{l13ses}
Let $\mathbf{q}$ be another homology trivializer for $\mathcal{D}(\cP)$. If the sequence
\begin{equation} \label{cxkq}
	\xymatrix{0 \ar[r] & B_{2}(\mathcal{D}(\mathcal{P}),\mathbf{q}) \ar[r]^-{incl.} & K^{\mathbf{q}} \ar[r]^-{\tilde{\partial}_{2}} & B_{1}(\mathcal{D}(\mathcal{P}))  \ar[r] & 0}
\end{equation}
is exact, then so is (\ref{cxkp}).
\end{lemma}
\begin{proof}
In what follows we write by $\tilde{\partial}_{i}^{\mathbf{p}}$ and $\tilde{\partial}_{i}^{\mathbf{q}}$ the cellular boundary maps for the complexes related to $\mathbf{p}$ and $\mathbf{q}$. Let $\xi$ be a 2-cycle in $K^{\mathbf{p}}$. By Lemma 12 of \cite{SES} we can write $\xi$ as a sum
$$\xi=\xi'+ \sum_{i=1}^{m} \varepsilon_{i}(\iota e_{i}-\tau e_{i}).[\lambda,p_{i}, w_{i}] +\sum_{i=m+1}^{m+n} \varepsilon_{i} [w_{i}, p_{i}, \lambda].(\iota e_{i}-\tau e_{i})$$
where $\xi'$ is a 2-chain in $C_{2}(\mathcal{D})$, and for $i=1,...,m+n$, $\varepsilon_{i}=\pm 1$, $w_{i} \in F$, $e_{i}$ is a positive edge in $\mathcal{D}$ and $p_{i} \in \mathbf{p}$. 
Suppose that, for some $i \in \{1, . . . , m + n\}$, $\partial p_{i} = f_{1}^{\delta_{1}} \circ \dots \circ f_{k}^{\delta_{k}}$ where for $j=1, . . . , k$, $\delta_j= \pm 1$ and $f_j$ is a positive edge in $\mathcal{D}$. If $i \leq m$ then because of the 3-cell $[e_i, [\lambda, p_i, w_i]]$ we have from (\ref{d3a}) that
$$(\iota e_{i}-\tau e_{i}).[\lambda,p_{i}, w_{i}] = -\sum_{j=1}^{k}\delta_{j} [e_{i},f_{j}].w_{i} \text{ (modulo } B_{2}(\mathcal{D}(\mathcal{P}),\mathbf{p})),$$
and for $i>m$, because of the 3-cell $[[w_i, p_i, \lambda], e_{i}]$ we have from (\ref{d3b}) that
$$[w_i, p_i, \lambda].(\iota e_{i}-\tau e_{i})= \sum_{j=1}^{k}\delta_{j} w_{i}.[f_{j}, e_i] \text{ (modulo } B_{2}(\mathcal{D}(\mathcal{P}),\mathbf{p})).$$
Thus $\xi$ is homologous modulo $B_{2}(\mathcal{D}(\mathcal{P}),\mathbf{p}))$ to some 2-cycle $\varsigma$, say, in $C_2(\mathcal{D})$, therefore we can write
$$\xi= \varsigma+ \beta \text{ with } \beta \in B_{2}(\mathcal{D}(\mathcal{P}),\mathbf{p})).$$
It follows now that
\begin{align*}
0&= \tilde{\partial}_{2}^{\mathbf{p}}(\varsigma+ \beta )\\
&= \tilde{\partial}_{2}^{\mathbf{p}}(\varsigma)+ \tilde{\partial}_{2}^{\mathbf{p}}(\beta)\\
&= \tilde{\partial}_{2}^{\mathbf{q}}(\varsigma)+0, \tag{since $\tilde{\partial}_{2}^{\mathbf{q}}|_{C_{2}(\mathcal{D})}=\tilde{\partial}_{2}^{\mathbf{p}}|_{C_{2}(\mathcal{D})}$}\\
\end{align*}
which shows that $\varsigma \in \text{Ker} \tilde{\partial}_{2}^{\mathbf{q}}$. But $\varsigma \in C_2(\mathcal{D})$ and $C_2(\mathcal{D}) \subseteq K^{\mathbf{q}}$, therefore $\varsigma \in \text{Ker} \tilde{\partial}_{2}^{\mathbf{q}} \cap K^{\mathbf{q}}$, hence by hypothesis there is some 3-chain $\omega$ in $C_3(\mathcal{D}, \mathbf{q})$ such that $\tilde{\partial}^{\mathbf{q}}_{3}(\omega)=\varsigma$. Lemma follows if we prove that there is some 3-chain $\omega'$ in $C_3(\mathcal{D}, \mathbf{p})$ such that $\tilde{\partial}^{\mathbf{p}}_{3}(\omega')=\varsigma$. To achieve this we first define group homomorphisms
$$\rho_{2}: C_{2} \oplus C^{\mathbf{q}}_{2} \rightarrow C_{2} \oplus C^{\mathbf{p}}_{2}$$
and 
$$\rho_{3}: C_3(\mathcal{D}, \mathbf{q}) \rightarrow C_3(\mathcal{D}, \mathbf{p})$$
such that the following diagram commutes.
\begin{equation} \label{rc3c2}
\xymatrix{C_3(\mathcal{D}, \mathbf{q}) \ar[d]_{\rho_{3}} \ar[r]^{\tilde{\partial}^{\mathbf{q}}_{3}}  & C_{2} \oplus C^{\mathbf{q}}_{2} \ar[d]^{\rho_{2}}\\
 C_3(\mathcal{D}, \mathbf{p}) \ar[r]^{\tilde{\partial}^{\mathbf{p}}_{3}}  & C_{2} \oplus C^{\mathbf{p}}_{2}}
\end{equation}
The map $\rho_{2}$ is defined as follows. Each square 2-cell is mapped identically to itself, which gives in fact the inclusion $(\mathbb{Z} F,\mathbb{Z} F)$-bimodule homomorphism
$$\rho_{2}': C_{2} \rightarrow C_{2}\oplus C_{2}^{\mathbf{p}}.$$
Now we define a second $(\mathbb{Z} F,\mathbb{Z} F)$-bimodule homomorphism 
$$\rho_{2}'': C^{\mathbf{q}}_{2} \rightarrow C_{2} \oplus C_{2}^{\mathbf{p}},$$
by specifying its values on the $(\mathbb{Z} F,\mathbb{Z} F)$ bimodule basis $\mathbf{q}$ of $C^{\mathbf{q}}_{2}$. For each $q \in \mathbf{q}$ with $\tilde{\partial}^{\mathbf{q}}_{2}(q) \neq 0$, since $\mathbf{p}$ is a homology trivializer, we choose a 2-chain $\varsigma_{q}$ in $C_{2} \oplus C_{2}^{\mathbf{p}}$ such that
\begin{equation} \label{pre-a}
 \tilde{\partial}^{\mathbf{p}}_{2}(\varsigma_{q})=\tilde{\partial}^{\mathbf{q}}_{2}(q).
\end{equation}
Now we define a mapping from the basis $\mathbf{q}$ of $C_{2}^{\mathbf{q}}$ as a free $(\mathbb{Z} F,\mathbb{Z} F)$-bimodule to $C_{2} \oplus C_{2}^{\mathbf{p}}$ by setting 
$$q \mapsto \left\{ \begin{array}{ccl} 
\varsigma_{q} & \text{if} & \tilde{\partial}^{\mathbf{q}}_{2}(q) \neq 0 \\ 
0 & \text{if} & \tilde{\partial}^{\mathbf{q}}_{2}(q)=0. 
\end{array}\right.$$
This map extends uniquely to a $(\mathbb{Z} F,\mathbb{Z} F)$-bimodule homomorphism 
$$\rho_{2}'': C^{\mathbf{q}}_{2} \rightarrow C_{2} \oplus C^{\mathbf{p}}_{2}.$$
Now $\rho_{2}$ is the unique $(\mathbb{Z} F,\mathbb{Z} F)$-bimodule homomorphism (hence a group homomorphism) induced by $\rho_{2}'$ and $\rho_{2}''$. 
In the case when for $q \in \mathbf{q}$ with $\tilde{\partial}^{\mathbf{q}}_{2}(q) \neq 0$ and $\varsigma_{q}=\sum_{i} \alpha_{i} \varsigma_i$ we have that
$$\rho_{2}([u,q,v])=\sum_{i} \alpha_{i} u.\varsigma_i.v.$$
Further we see that whenever $\tilde{\partial}^{\mathbf{q}}_{2}(q) \neq 0$, say
$$\tilde{\partial}^{\mathbf{q}}_{2}(q)=\sum_{j} \delta_{j} g_{j},$$
and $\varsigma_{q}=\sum_{i} \alpha_{i} \varsigma_i$, the equality (\ref{pre-a}) translates to
\begin{equation*} 
\tilde{\partial}^{\mathbf{p}}_{2}\left(\sum_{i} \alpha_{i} \varsigma_i\right)= \sum_{j} \delta_{j} g_{j},
\end{equation*}
and since $\tilde{\partial}^{\mathbf{p}}_{2}$ is a $(\mathbb{Z} F,\mathbb{Z} F)$-bimodule homomorphism, then for every $u,v \in F$ we have that
\begin{equation} \label{bound}
\tilde{\partial}^{\mathbf{p}}_{2}\left(\sum_{i} \alpha_{i} u.\varsigma_i.v\right)= \sum_{j} \delta_{j} u.g_{j}.v.
\end{equation}
Assuming in addition that for each particular $i$
$$\tilde{\partial}^{\mathbf{p}}_{2}(\varsigma_i)= \sum_{i_k} \mu_{i_k} h_{i_k},$$
where $\mu_{i_k} = \pm 1$ and $h_{i_{k}}$ is a positive edge, since $\tilde{\partial}^{\mathbf{p}}_{2}$ is a $(\mathbb{Z} F,\mathbb{Z} F)$-bimodule homomorphism, we have that
$$\tilde{\partial}^{\mathbf{p}}_{2}(u.\varsigma_i.v)= \sum_{i_k} \mu_{i_k} u.h_{i_k}.v,$$
and then we get from (\ref{bound}) that
\begin{equation} \label{bound-1}
\sum_{i} \alpha_{i} \sum_{i_k}  \mu_{i_k}u.h_{i_k}.v= \sum_{j} \delta_{j} u.g_{j}.v.
\end{equation}
 We define a group homomorphism $\rho_{3}$ by specifying its values on each free abelian group generator $[f,\sigma]$ or $[\sigma, f]$ where $\sigma$ is a 2-cell from $C_{2} \oplus C_{2}^{\mathbf{q}}$. If $\sigma$ is a square 2-cell, we define for each positive edge $f$
 $$\rho_{3}([f,\sigma])=[f,\sigma] \text{ and } \rho_{3}([\sigma, f])=[\sigma, f].$$
If $\sigma=[u, q,v]$ with $q \in \mathbf{q}$ and $\tilde{\partial}^{\mathbf{q}}_{2}(q)=0$, we define for every positive edge $f$
$$\rho_{3}([f,[u,q,v]])=0 \text{ and } \rho_{3}([[u,q,v], f])=0.$$ 
If $\sigma=[u, q,v]$ with $q \in \mathbf{q}$, $\tilde{\partial}^{\mathbf{q}}_{2}(q) \neq 0$ and $\varsigma_{q}=\sum_i \alpha_i  \varsigma_i $, then we define
$$\rho_{3}([f,[u,q,v]])=\sum_i \alpha_i [f, u. \varsigma_i. v] \text{ and } \rho_{3}([[u,q,v],f])=\sum_i \alpha_i [u. \varsigma_i. v, f].$$
Next we check that the diagram (\ref{rc3c2}) commutes by checking that for every free abelian group generating 3-cell $[f,\sigma]$ or $[\sigma,f]$ we have that 
$$\tilde{\partial}^{\mathbf{p}}_{3} (\rho_{3}([f,\sigma])=\rho_{2} (\tilde{\partial}^{\mathbf{q}}_{3}([f,\sigma]))$$
and similarly
$$\tilde{\partial}^{\mathbf{p}}_{3} (\rho_{3}([\sigma,f])=\rho_{2} (\tilde{\partial}^{\mathbf{q}}_{3}([\sigma,f])).$$
If $\sigma$ is a square 2-cell and $f$ a positive edge, we remark first that $\tilde{\partial}^{\mathbf{q}}_{3}([f,\sigma])=\tilde{\partial}^{\mathbf{p}}_{3}([f,\sigma])$ is a 2-chain in $C_{2}$. It follows that
\begin{align*}
\tilde{\partial}^{\mathbf{p}}_{3}(\rho_{3} ([f,\sigma]))&= \tilde{\partial}^{\mathbf{p}}_{3}([f,\sigma])\\
&=\tilde{\partial}^{\mathbf{q}}_{3}([f,\sigma])\\
&= \rho_{2} (\tilde{\partial}^{\mathbf{q}}_{3}([f,\sigma])) \tag{from the remark above.}
\end{align*}
The proof for 3-cells $[\sigma, f]$ with $\sigma$ is square is similar to the above and is omitted. Now we check the case when $\sigma=[u, q,v]$ with $q \in \mathbf{q}$ and $\tilde{\partial}^{\mathbf{q}}_{2}(q)=0$. If $q$ is attached along the closed path $g^{\delta_{1}}_{1} \circ \dots \circ g^{\delta_{n}}_{n}$, then $[u,q,v]$ is attached along the closed path $u.g^{\delta_{1}}_{1}.v \circ \dots \circ u.g^{\delta_{n}}_{n}.v$, and since $\tilde{\partial}^{\mathbf{q}}_{2}(q)=0$, it follows that 
$$\delta_{1}u.g_{1}.v + \dots +\delta_{n}u.g_{n}.v=\tilde{\partial}_{2}([u,q,v])=u. \tilde{\partial}^{\mathbf{q}}_{2}(q) .v=0.$$
Now Lemma \ref{zro} implies that
\begin{equation} \label{izro}
\delta_{1}[f,u.g_{1}.v] + \dots +\delta_{n}[f,u.g_{n}.v]=0.
\end{equation}
For every positive edge $f$ we have that
\begin{align*}
\rho_{2}(\tilde{\partial}^{\mathbf{q}}_{3} ([f, [u, q, v]))&= \rho_{2}\left ((\tau f-\iota f). [u,q,v]- \sum_{i=1}^{n} \delta_{i} [f, u.g_{i}.v]\right)\\
&= \rho_{2} ((\tau f-\iota f). [u,q,v]) \tag{by (\ref{izro})}\\
&=0\\
&= \tilde{\partial}^{\mathbf{p}}_{3}(0)\\
&=\tilde{\partial}^{\mathbf{p}}_{3}(\rho_{3} ([f, [u, q, v]).
\end{align*}
The proof for $[[u,q,v],f]$ is similar to the above and is omitted. The two remaining cases correspond to the case when $\sigma=[u, q,v]$ with $q \in \mathbf{q}$, $\tilde{\partial}^{\mathbf{q}}_{2}(q) \neq 0$ and $\varsigma_{q}=\sum_i \alpha_i  \varsigma_i $. We check for convenience the case when the 3-cell is $[f,[u,q,v]]$.
\begin{align*}
\tilde{\partial}^{\mathbf{p}}_{3}(\rho_{3} ([f,[u,q,v]]))&=\tilde{\partial}^{\mathbf{p}}_{3} \left(\sum_i \alpha_i [f, u. \varsigma_i. v]\right)\\
&= \sum_{i} \alpha_i (\tau f - \iota f). (u. \varsigma_i . v) - \sum_i \alpha _{i} \sum_{i_k} \mu_{i_k} [f, u.h_{i_k}.v]\\
&= \sum_{i} \alpha_i (\tau f - \iota f). (u. \varsigma_i . v) - \sum_{j} \delta_{j} [f,u.g_{j}.v] \tag{by (\ref{bound-1}) and Lemma \ref{zro}}\\
&= \rho_{2} \left((\tau f-\iota f). [u,q,v]  - \sum_{j} \delta_{j} [f,u.g_{j}.v]\right)\\
&= \rho_{2} (\tilde{\partial}^{\mathbf{q}}_{3}([f,[u,q,v]])).
\end{align*}
For the 3-chain $\omega$ above we let $\omega'=\rho_{3}(\omega)$ and see that
\begin{align*}
\tilde{\partial}^{\mathbf{p}}_{3}(\omega')&= \tilde{\partial}^{\mathbf{p}}_{3}(\rho_{3}(\omega))\\
&= \rho_{2} (\tilde{\partial}^{\mathbf{q}}_{3}(\omega)) \tag{since (\ref{rc3c2}) commutes}\\
&= \rho_{2}(\varsigma)\\
&= \varsigma \tag{since $\varsigma \in C_{2}$}
\end{align*}
which completes the proof.
\end{proof}

\subsection{The second stage-The proof for the exactness of (\ref{cxkp}) in a special case} \label{2-st}

This section covers the proof of the first half of Proposition 14 of \cite{SES} which states that (\ref{cxkp}) is exact when the monoid presentation $\mathcal{P}=\langle \mathbf{x}, \mathbf{r}\rangle $ is complete and the set $\mathbf{p}$ is obtained by choosing resolutions of all the critical pairs of $\mathbf{r}$. As we did in the previous section \ref{1-st} with Lemma 13, in order to enhance clarity and rigour, we give here a slightly different proof than the one given in \cite{SES}. To make the proof clear it will be convenient for us to introduce a strict order $\prec_{2,mul}$ on the set $[K^{\textbf{p}}]$ whose elements are multisets $[\xi]$ one for each nonzero 2-chain $\xi \in K^{\textbf{p}}$. In what follows we will describe the set $[K^{\textbf{p}}]$ and the order $\prec_{2,mul}$ in the set $[K^{\textbf{p}}]$ . Let $\xi$ be a non zero $2$-chain from $K^{\mathbf{p}}$. We can express $\xi$ uniquely as a sum
\begin{equation} \label{ksi}
\xi = \sum_{i=1}^m n_i \sigma_i 
\end{equation}
where $m > 0$, $\sigma_1, \dots, \sigma_m$ are distinct $2$-cells, and $n_1, \dots, n_m$ are non-zero integers. Denote by $[\xi]$ the corresponding multiset whose elements are the 2-cells $\sigma_{i}$ represented in $\xi$ with their corresponding multiplicities $|n_{i}|$.  Since $\prec_{2}$ is a terminating strict order,  it follows from Lemma 2.5.4 and Theorem 2.5.5 of \cite{BN} that $\prec_{2}$ induces a terminating strict order $\prec_{2,mul}$ in the set $[K^{\textbf{p}}]$ . We will describe this order in the form that appears in the proof of Proposition \ref{sc}. Assume that $\sigma_{1}$ is a maximal 2-cell among the 2-cells that are represented in (\ref{ksi}), and assume that $\sigma_{1}$ is removed from (\ref{ksi}) and is replaced by a $2$-chain $\xi_{1}=\varepsilon_1 \mu_1 + \cdots + \varepsilon_l \mu_l$ ($l > 0$, and for $j = 1, 2, \dots, l$, $\varepsilon_j $ is a non zero integer and $\mu_j$ is a $2$-cell), where
for each $1 \leq j \leq l$, $\mu_j \prec_{2}\sigma_{1}$.  Under the assumption that the resulting 2-chain $\xi_{1}$ is non zero, and after representing it in the form (\ref{ksi}), 
it follows from definition 2.5.3 of \cite{BN} that $[\xi_{1}]\prec_{2,mul} [\xi]$ since every 2-cell that appears in $\xi_{1}$ is strictly less than $\sigma_{1}$ which has already disappeared. We can simply take as $X$ of definition 2.5.3 of \cite{BN} the multiset $[\xi]$ and than take $Y=[\xi_{1}]$.

The following three lemmas will be essential in the proof of Proposition \ref{sc}.

\begin{lemma} \label{tcp}
Let $\mathcal{P}=\langle \mathbf{x}, \mathbf{r}\rangle $ be any monoid presentation. If $(e,f)$ is a pair of positive edges with the same initial, then the edges $e$ and $f$ are either disjoint, or there are $\alpha,\beta \in F$ and positive edges $e',f'$ such that $(e,f)=(\alpha.e'.\beta,\alpha. f'.\beta)$ and $(e',f')$ is a critical pair.
\end{lemma}
\begin{proof}
Letting $e=(u,r,1,v)$ and $f=(\eta, s,1,\xi)$ we have that $ur_{+}v=\eta s_{+} \xi$. If $r_{+}$ and $s_{+}$ do not overlap, then $e$ and $f$ are disjoint. If $\eta$ is a prefix of $u$ and $\xi$ is a suffix of $v$, then there are $u', v' \in F$ such that $u=\eta u'$ and $v=v' \xi$. It follows that
$$e=\eta. (u', r,1,v').\xi \text{ and } f=\eta. (\lambda, s,1, \lambda).\xi$$
and writing $e'=(u', r,1,v')$ and $f'=(\lambda, s,1, \lambda)$ we have that $(e',f')$ is a critical pair of inclusion type. The case when $u$ is a prefix of $\eta$ and $v$ is a suffix of $\xi$ is symmetric to the above. Now we see the case when $u$ is a prefix of $\eta$ and $\xi$ is a suffix of $v$ and leave it to the reader to check the symmetric case when $\eta$ is a prefix of $u$ and $v$ is a suffix of $\xi$. It follows in our case that there are $\eta', v' \in F$ such that $\eta=u\eta'$ and $v=v'\xi$, consequently
$$e=u.(\lambda, r,1, v').\xi \text{ and } f=u.(\eta', s,1, \lambda).\xi.$$
Writing $e'=(\lambda, r,1, v')$ and $f'=(\eta', s,1, \lambda)$ we see that $(e',f')$ is a critical pair of overlapping type.
\end{proof}

\begin{lemma} \label{r3}
Let $\mathcal{P}=\langle \mathbf{x}, \mathbf{r}\rangle $ be a complete monoid presentation and $\mathbf{p}$ be the set of 2-cells obtained by resolving all the critical pairs of $\mathbf{r}$. Let $\sigma=[w,p,w'] \in F .\mathbf{p}.F$ such that $\text{st}(\sigma)=(f.u, g.u)$ for some $u \in F$. Then there is some $\sigma' \in F. \mathbf{p}.F$ such that $\sigma'.u=\sigma$ and that $\text{st}(\sigma')=(f,g)$.
\end{lemma}
\begin{proof}
Let $\text{st}(p)=(f_{1},g_{1})$, then $\text{st}(\sigma)=(w.f_{1}.w', w.g_{1}.w')$. Assume first that $(f_{1},g_{1})$ is a critical pair of inclusion type, then
$$f_{1}=(\eta, r,1, \eta'), g_{1}=(\lambda , s, 1, \lambda) \text{ and } \eta r_{+} \eta'=s_{+}.$$
Letting
$$f=(\bar{\eta}, \bar{r},1, \bar{\eta}') \text{ and } g= (\bar{\bar{\eta}}, \bar{s},1, \bar{\bar{\eta}}'),$$
we see from the condition that
$$(w\eta, r,1, \eta'w')=w.f_{1}.w'= f.u=(\bar{\eta}, \bar{r},1, \bar{\eta}'u),$$
and
$$(w,s,1,w')=w.g_{1}.w'=g.u= (\bar{\bar{\eta}}, \bar{s},1, \bar{\bar{\eta}}'u).$$
It follows that $\bar{r}=r$, $\bar{s}=s$ and $w'=\bar{\bar{\eta}}'u$ consequently we have that
$$f.u= (w\eta, r,1, \eta'w')=(w\eta, r,1, \eta'\bar{\bar{\eta}}'u)=(w\eta, r,1, \eta'\bar{\bar{\eta}}').u$$
and 
$$g.u=(w,s,1,\bar{\bar{\eta}}'u)=(w,s,1,\bar{\bar{\eta}}').u.$$
We now get that
$$f=(w\eta, r,1, \eta'\bar{\bar{\eta}}') \text{ and } g =(w,s,1,\bar{\bar{\eta}}').$$
Finally letting $\sigma'=[w, p, \bar{\bar{\eta}}']$ we see that
$$\text{st}(\sigma')=(f,g)$$
and that
$$\text{st}(\sigma'.u)=(f.u,g.u)=\text{st}(\sigma)$$
and then Corollary \ref{cunambp} implies that $\sigma'.u=\sigma$.

In the second case we assume that $(f_{1},g_{1})$ is a critical pair of overlapping type therefore
$$f_{1}=(\lambda, r, 1, v) \text{ and } g_{1}=(u,s,1,\lambda) \text{ with } r_{+}v=us_{+}.$$
With the notations of the first part of the proof, we have that
$$(w,r,1,vw')=w.f_{1}.w'=f.u=(\bar{\eta}, \bar{r}, 1, \bar{\eta}'u),$$
and
$$(wu, s,1,w')=w.g_{1}.w'= g.u=(\bar{\bar{\eta}}, \bar{s},1, \bar{\bar{\eta}}'u),$$
from which we obtain that $\bar{r}=r$, $\bar{s}=s$ and $w'=\bar{\bar{\eta}}'u$. It follows now that
$$f.u= (w,r,1,v\bar{\bar{\eta}}'u) \text{ and } g.u= (wu,s,1,\bar{\bar{\eta}}'u),$$
which imply that
$$f= (w,r,1,v\bar{\bar{\eta}}') \text{ and } g= (wu,s,1,\bar{\bar{\eta}}').$$
Letting now $\sigma'=[w, p, \bar{\bar{\eta}}']$ we see that
$$\text{st}(\sigma')=(f,g)$$
and that
$$\text{st}(\sigma'.u)=(f.u,g.u)=\text{st}(\sigma)$$
which implies that $\sigma'.u=\sigma$.
\end{proof}

\begin{lemma} \label{r3sq}
Let $\mathcal{P}=\langle \mathbf{x}, \mathbf{r}\rangle $ be any monoid presentation and $\mathcal{D}(\cP)$ the associated Squier complex. Let $\sigma$ be a square 2-cell such that $\text{st}(\sigma)=(f.u, g.u)$ for some $u \in F$. Then there is some other square 2-cell $\sigma'$ such that $\sigma'.u=\sigma$ and that $\text{st}(\sigma')=(f,g)$.
\end{lemma}
\begin{proof}
Assume that $\sigma=[f_{1},g_{1}]$, then $f.u=f_{1}.\iota g_{1}$ and $g.u=\iota f_{1}. g_{1}$. Letting
$$f=(w,r,1,w'), g=(\bar{w}, \bar{r}, 1,\bar{w}'), f_{1}=(w_{1}, r_{1},1, w'_{1}), g_{1}=(\bar{w}_{1}, \bar{r}_{1},1, \bar{w}'_{1}),$$
since
\begin{equation} \label{inteq}
(\bar{w},\bar{r},1,\bar{w}'u)=g.u=\iota f_{1}.g_{1}=(\iota f_{1}\bar{w}_{1}, \bar{r}_{1},1, \bar{w}_{1}'),
\end{equation}
we have that $\bar{w}_{1}'=\bar{w}'u$ which implies that
$$g_{1}=(\bar{w}_{1}, \bar{r}_{1},1, \bar{w}'_{1})=(\bar{w}_{1}, \bar{r}_{1},1, \bar{w}'u)=(\bar{w}_{1}, \bar{r}_{1},1, \bar{w}').u.$$
Letting $g'_{1}=(\bar{w}_{1}, \bar{r}_{1},1, \bar{w}')$ and $\sigma'=[f_{1},g'_{1}]$ we see that
$$
\sigma=[f_{1},g_{1}]= [f_{1}, g'_{1}.u]= [f_{1},g'_{1}].u=\sigma'.u.
$$
But we also have that
$$(w,r,1,w'u)=f.u=f_{1}.\iota g_{1}=(w_{1},r_{1},1,w'_{1}\bar{w}_{1}\bar{r}_{1 +}\bar{w}'_{1})=(w_{1},r_{1},1,w'_{1}\bar{w}_{1}\bar{r}_{1 +}\bar{w}'u),$$
which implies that
$$f=(w,r,1,w')=(w_{1},r_{1},1,w'_{1}\bar{w}_{1}\bar{r}_{1 +}\bar{w}')=f_{1}.\iota g'_{1}.$$
Now we see that 
\begin{align*}
\iota f_{1}.g'_{1}&= (\iota f_{1} \bar{w}_{1}, \bar{r}_{1},1,\bar{w}')\\
&= (\bar{w},\bar{r},1,\bar{w}') \tag{by (\ref{inteq})}\\
&=g.
\end{align*}
Finally, $\text{st}(\sigma')=(f_{1}.\iota g'_{1} ,\iota f_{1}.g_{1}')=(f,g)$.
\end{proof}

\begin{proposition} \label{sc}
Let $\mathcal{P}=\langle \mathbf{x}, \mathbf{r}\rangle $ be a complete monoid presentation and $\mathbf{p}$ be the set of 2-cells obtained by choosing resolutions of all the pairs in $\mathbf{r}$, then the sequence (\ref{cxkp}) is exact.
\end{proposition}
\begin{proof}
Assume by contradiction that there is a 2-cycle $\xi \in K^{\mathbf{p}} \setminus B_{2}(\mathcal{D}(\mathcal{P}), \mathbf{p})$, and let $\xi$ be written in
the form (\ref{ksi}).  We can assume that the $2$-cell $\sigma_1$ is maximal in $\xi$, and assume that $\sigma_1 = [w,(f, g)]$.

\textbf{Case I.} $\sigma_1 \in C_2$, that is, $f$ and $g$ are disjoint. Since $\widetilde{\partial_2}(\xi) = 0$, then $g$ must cancel under the boundary map, so there is a $2$-cell
$\sigma_j$ ($j > 1$) in $\xi$ such that $\text{st}(\sigma_j )= [w,(e, g)]$ with $e \neq f$ for if $e=f$, then $\sigma_{j} \notin F.\mathbf{p}.F$ and from Lemma \ref{sqneqp} we must have that $\sigma_{j}$ is square 2-cell from $\mathcal{D}(\mathcal{P})$ but this would violate the conclusion of Corollary \ref{cunamb} which requires $\sigma_{j}=\sigma_{1}$. So it remains that $e\neq f$. The maximality of $\sigma_{1}$ and the fact that edges $e$ and $f$ (with the same initial $w$) are comparable implies that $e \prec_{1} f$. Then $e$ must be disjoint from $g$. Since $\sigma_{1}$ is square, we may write it as $[f_{1}, g_{1}]$, and then $e=e_{1}. \iota g_{1}$ for some positive edge $e_{1}$. It follows that
$$\iota e_{1} .\iota g_{1}=\iota e=\iota f=\iota f_{1}. \iota g_{1},$$
which shows that $e_{1}$ and $f_{1}$ are edges with the same initial, therefore since the system is complete, it follows from Lemma \ref{tcp} that there is a 2-cell $\delta$ in $(\mathcal{D}(\mathcal{P}), \mathbf{p})$ with $\text{st}(\delta)=(e_{1}, f_{1})$. Consider now the 3-cell $[\delta, g_{1}]$. From our assumption on $\xi$, the $2$-cycle
\begin{align*}
\xi_{1} &= \xi + n_1 \partial_3[\delta, g_{1}]\\
&= n_{1}\sigma_{1} + n_{2}\sigma_{2} + \dots + n_{m}\sigma_{m} + n_{1}(-[f_{1}, g_{1}]+[e_{1}, g_{1}] - \delta. \iota g_{1} + \varsigma ) \tag{by (\ref{d3pr})}\\
&= n_{2}\sigma_{2} + \dots + n_{m}\sigma_{m} + n_{1}[e_{1}, g_{1}] - n_{1} \delta. \iota g_{1} + n_{1} \varsigma
\end{align*}
is in $K^{\mathbf{p}} \setminus B_{2}(\mathcal{D}(\mathcal{P}), \mathbf{p})$ and it is obtained by $\xi$ by replacing $n_{1}\sigma_{1}$ with the 2-chain
$$n_{1}[e_{1}, g_{1}] - n_{1} \delta. \iota g_{1} + n_{1} \varsigma$$
whose 2-cells are from (\ref{stl2}) strictly lesser than $\sigma_{1}=[f_{1},g_{1}]$ with respect to $\prec_{2}$, therefore we have that $[\xi_{1} ]\prec_{2,mul} [\xi]$.

\textbf{Case II.} $\sigma_1 \in F.\mathbf{p}.F$.  In this case it is impossible for $\sigma_{1}$ to be expressed as
\[
\sigma_1 = u.\sigma.v
\]
where $\sigma \in \mathbf{p}$, and $u$ and $v$ are both irreducible words. Indeed, if both $u,v$ were to be irreducible, since $\varphi (\xi) = 0$, there
must exist $2 \leq j \leq m$ such that
\[
\sigma_j = u'.\sigma.v',
\]
where $\overline{u} = \overline{u'}$ and $\overline{v} = \overline{v'}$. Hence $u' \xrightarrow{*}_r u$ and $v' \xrightarrow{*}_r v$ since $u, v$ are irreducible.
But since $\sigma_j \neq \sigma_1$, we must have either $u' \neq u$ or $v' \neq v$. Thus 
$$w_{\sigma_{1}}=u w_{\sigma} v \prec_{0} u' w_{\sigma} v'=w_{\sigma_{j}},$$
which contradicts our assumption that $\sigma_1$ is a maximal $2$-cell in $\xi$. 

Having excluded the above presentation for $\sigma_{1}$, we note that two possible cases remain.

\textbf{Case II-i.} There is some word $u\neq \lambda$ such that $\sigma_{1}=u.\sigma$ and $u$ is reduced, therefore there is some positive edge $e$ with $\iota e=u$. We let $\sigma=[w,(f_{1},g_{1})]$. The existence of the positive edge $e$ assures that there is the $3$-cell $[e, \sigma]$, and then from the assumption we have the existence of a 2-cycle in $K^{\mathbf{p}} \setminus B_{2}(\mathcal{D}(\mathcal{P}), \mathbf{p})$ 
\begin{align*}
\xi_{1} &= \xi + n_1 \partial_3 [e, \sigma]\\
&= n_{1}\sigma_{1} + n_{2}\sigma_{2} + \dots + n_{m}\sigma_{m} +n_{1}(-\iota e. \sigma+ [e,g_{1}] - [e,f_{1}] +\varsigma) \tag{by (\ref{d3pl})}\\
&= n_{2}\sigma_{2} + \dots + n_{m}\sigma_{m} +(n_{1} [e,g_{1}]-n_{1}[e,f_{1}] + n_{1}\varsigma),
\end{align*}
which is obtained from $\xi$ by replacing $n_{1}\iota e. \sigma$ with the 2-chain 
$$n_{1} [e,g_{1}]-n_{1}[e,f_{1}] + n_{1}\varsigma$$
whose 2-cells are from (\ref{stl1}) strictly lesser than $\iota e. \sigma$, therefore $[\xi_{1} ] \prec_{2,mul} [\xi]$.

\textbf{Case II-ii.} There is some word $u\neq \lambda$ such that $\sigma_{1}=\sigma.u$ with $u$ being reduced therefore there is some positive edge $e$ with $\iota e=u$. Renumbering if necessary, we can assume that $\sigma_1, \dots, \sigma_k$ are all the $2$-cells in
$\xi$ having maximal vertex $w$. For $i = 1, 2, \dots, k$ assume that $\sigma_i =[w,(f_i, g_i)]$ (where $f_1 = f$ and $g_1 = g$). If $\text{st}(\sigma)=(f'_{1},g'_{1})$, then
$$(f_{1},g_{1})=\text{st}(\sigma_{1})=\text{st}(\sigma.\iota e)=(f'_{1}.\iota e, g'_{1}.\iota e).$$
Since for $i = 2, \dots, k$ the edge $g_i$ is not greater than $g_{1}$, it follows that there are positive edges $g'_{i}$ such that $g_{i}=g'_{i}.\iota e$. But since each $f_{i}\prec_{1} g_{i}$ it follows that there are also positive edges $f'_{i}$ such that $f_{i}=f'_{i}.\iota e$. Lemma \ref{r3} and Lemma \ref{r3sq} now imply that there are 2-cells $[w',(f'_{i}, g'_{i})]$ such that $[w, (f_i, g_i)]=[w',(f'_{i}, g'_{i})].\iota e$. Since $\widetilde{\partial_2}(\xi) = 0$, all the edges with initial vertex $w$ from the $2$-cells $\sigma_1, \dots, \sigma_k$
must cancel under the boundary map, that is
\[
\sum_{i=1}^k n_i (f_i - g_i) = 0.
\]
Lemma \ref{uu} implies that
\[
\sum_{i=1}^k n_i (f'_i - g'_i) = 0.
\]
Consequently we must have from Lemma \ref{zro} that
\begin{equation} \label{nsq}
\sum_{i=1}^k n_i \left( [f'_i, e] - [g'_i, e] \right)= 0.
\end{equation}
From the assumption we have the existence of a 2-cycle in $K^{\mathbf{p}} \setminus B_{2}(\mathcal{D}(\mathcal{P}), \mathbf{p})$
\begin{align*}
\xi_{1} &= \xi + \sum_{i=1}^k n_i \partial_3[[w', (f'_i, g'_i)], e] \\
&= \xi + \sum_{i=1}^k n_{i}(-[g'_{i},e] +[f'_{i},e] -[w', (f'_i, g'_i)].\iota e +\varsigma) \tag{by (\ref{d3pr})}\\
&= \xi + \sum_{i=1}^k n_{i}(-[g'_{i},e] +[f'_{i},e]) - \sum_{i=1}^k n_{i} [w', (f'_i, g'_i)].\iota e +  \sum_{i=1}^k n_{i}\varsigma\\
&= \xi - \sum_{i=1}^k n_{i} [w', (f'_i, g'_i)].\iota e +  \sum_{i=1}^k n_{i}\varsigma \tag{by (\ref{nsq})}\\
&= \xi - \sum_{i=1}^k n_{i} [w, (f_i, g_i)] +  \sum_{i=1}^k n_{i}\varsigma \tag{since $[w, (f_i, g_i)]=[w',(f'_{i}, g'_{i})].\iota e$}\\
&= \sum_{i=k+1}^m n_i \sigma_i + \sum_{i=1}^k n_{i}\varsigma,
\end{align*}
where all $2$-cells in $\sum_{i=1}^k n_{i}\varsigma$ have maximal vertex less than $w$, therefore $[\xi_{1} ] \prec_{2,mul} [ \xi]$. Repeating the above argument any number $n \geq 1$ of times we can find 2-cycles $\xi_{n} \in K^{\mathbf{p}} \setminus B_{2}((\mathcal{D}(\mathcal{P}), \mathbf{p})$ which are homologous with $\xi$ and such that $[\xi_{n}] \prec_{2,mul} \dots \prec_{2,mul} [\xi_{1}]\prec_{2,mul} [\xi]$. But this contradicts the fact that $\prec_{2,mul}$ is Noetherian, therefore there are no 2-cycles in $K^{\mathbf{p}} \setminus B_{2}(\mathcal{D}(\mathcal{P}), \mathbf{p})$.
\end{proof}

\subsection{The third stage-Proof of Proposition 14} \label{3-st}

Since at this stage it is used the so called Knuth-Bendix completion procedure, we will begin by describing the procedure very briefly. The Knuth-Bendix procedure \cite{Book+Otto}, produces a complete system out of any given system and equivalent to it. Given a rewriting system $\cP=\langle\mathbf{x}, \mathbf{r} \rangle$ and a reduction well order $\succ$ on $F$ that is compatible with $\mathbf{r}$ (there is always such one as explained in section \ref{osqc}), one can produce a complete rewriting system $\cP^{\infty}$ that is equivalent to $\cP$ in the following way. Put $\mathbf{r}_{0}=\mathbf{r}$. For each non-resolvable critical pair of edges $(e,f)$ in $\mathcal{D}(\mathcal{P})$ we chose positive or constant paths $\lambda_{(e,f)}$, $\rho_{(e,f)}$ from $\tau e$ and $\tau f$ respectively to distinct irreducibles. If either $\tau e$ of $\tau f$ is irreducible, then the corresponding path is the trivial one. Let $\mathbf{r}_{1}$ be the set of rules obtained from $\mathbf{r}$ by adding for each such critical pair $(e,f)$ the rule $(\tau \lambda_{(e,f)}, \tau \rho_{(e,f)})$ if $\tau \lambda_{(e,f)} \succ \tau \rho_{(e,f)}$, otherwise adding the rule $(\tau \rho_{(e,f)} ,\tau \lambda_{(e,f})$. It is clear that $\cP_{1}=\langle \mathbf{x}, \mathbf{r}_{1} \rangle$ is equivalent to $\cP=\langle \mathbf{x}, \mathbf{r} \rangle$ and that $\mathbf{r} \subseteq \mathbf{r}_{1}$ where the inclusion is strict if $\cP$ is not complete. Assume by induction that we have defined a sequence of equivalent rewriting systems
$$\cP=\langle \mathbf{x}, \mathbf{r}_{0}\rangle,...,\cP_{n-1}=\langle \mathbf{x}, \mathbf{r}_{n-1}\rangle,\cP_{n}=\langle \mathbf{x}, \mathbf{r}_{n}\rangle,$$
and consequently, an increasing sequence of complexes
$$\mathcal{D}(\mathcal{P}) \subseteq \dots \subseteq \mathcal{D}(\cP_{n-1}) \subseteq \mathcal{D}(\cP_{n}),$$
where $\cP_{n}=\langle \mathbf{x}, \mathbf{r}_{n}\rangle$ is obtained from $\cP_{n-1}=\langle \mathbf{x}, \mathbf{r}_{n-1}\rangle$ by resolving all the non-resolvable critical pairs of $\mathcal{D}(\cP_{n-1})$. Put $\mathbf{r}_{\infty}=\underset{n \geq 0}{\cup}\mathbf{r}_{n}$ and let $\cP_{\infty}=\langle \mathbf{x}, \mathbf{r}_{\infty}\rangle$ be the resulting rewriting system. The corresponding complex $\mathcal{D}(\cP_{\infty})$ will be later denoted by $\mathcal{D}^{\infty}$. The rewriting system $\cP_{\infty}=\langle \mathbf{x}, \mathbf{r}_{\infty}\rangle $ is obviously equivalent to $\cP$ and it is complete since it is compatible with the order $\succ$ on $F$ and for every non-resolvable pair $(e,f)$ of edges found in some $\mathcal{D}(\cP_{n})$, there is an edge $g$ in $\mathcal{D}(\cP_{n+1})$ connecting the endpoints of the positive (or trivial) paths $\lambda_{(e,f)}$ and $\rho_{(e,f)}$ of $\mathcal{D}(\cP_{n})$.

Now assume that we have a monoid presentation $\cP=\langle \mathbf{x}, \mathbf{r}\rangle $ not necessarily complete and that $\mathbf{r}$ is compatible with a length-lexicographic order in the free monoid $F$ on $\mathbf{x}$. Using the Knuth-Bendix procedure, we obtain a new presentation $\cP^{\infty}=\langle \mathbf{x}, \mathbf{r}^{\infty}\rangle $ with $\mathbf{r} \subseteq \mathbf{r}^{\infty}$ and where $\mathbf{r}^{\infty}$ is compatible with the order on $F$. The Squier complex $\mathcal{D}^{\infty}=\mathcal{D}(\cP^{\infty})$ has trivializer $\mathbf{p}^{\infty}$ obtained by choosing the resolutions of all the critical pairs of $\mathbf{r}^{\infty}$. From Proposition \ref{sc} we have the exactness of 
\begin{equation} \label{cxkinfty}
\xymatrix{0 \ar[r] & B_{2}(\mathcal{D}^{\infty},\mathbf{p}^{\infty}) \ar[r]^-{incl.} & K^{\mathbf{p}^{\infty}} \ar[r]^-{\tilde{\partial}^{\infty}_{2}} & B_{1}(\mathcal{D}^{\infty})  \ar[r] & 0,}
\end{equation}
$$$$
where $K^{\mathbf{p}^{\infty}}=C_{2}(\mathcal{D}^{\infty}) + J. \mathbf{p}^{\infty}. \mathbb{Z}F+ \mathbb{Z}F.\mathbf{p}^{\infty}. J$. In \cite{SES}, in order to make use of Lemma 13, it is defined in terms of $\mathbf{p}^{\infty}$ a set $\mathbf{q}$ of closed paths in $\mathcal{D}(\cP)$ and than a map $\rho$ from $(\mathcal{D}^{\infty}, \mathbf{p}^{\infty})$ to $(\mathcal{D}(\cP), \mathbf{q})$. An analysis of the induced map $\hat{\rho}:\mathbf{C}(\mathcal{D}^{\infty}, \mathbf{p}^{\infty}) \rightarrow \mathbf{C}(\mathcal{D}(\cP), \mathbf{q})$ on the respective chain complexes shows that $\mathbf{q}$ is a homology trivializer. The mapping $\rho$ is in fact a mapping from the cell subdivision of $(\mathcal{D}^{\infty}, \mathbf{p}^{\infty})$ to that of $(\mathcal{D}(\cP), \mathbf{q})$ and no argument is given to assure that $\rho$ induces a continuous map of complexes. In the absence of such an argument, we decided to bypass the topological approach of \cite{SES} and get instead to the map $\hat{\rho}:\mathbf{C}(\mathcal{D}^{\infty}, \mathbf{p}^{\infty}) \rightarrow \mathbf{C}(\mathcal{D}(\cP), \mathbf{q})$ in a more direct way.

We begin by pointing out that $\mathcal{D}(\mathcal{P})$ is a subcomplex of $(\mathcal{D}^{\infty}, \mathbf{p}^{\infty})$, therefore $C_{1}(\mathcal{D}(\mathcal{P})) \leq C_{1}(\mathcal{D}^{\infty}, \mathbf{p}^{\infty})$. We will define first a retraction $\hat{\rho_{1}}: C_{1}(\mathcal{D}^{\infty}, \mathbf{p}^{\infty}) \rightarrow C_{1}(\mathcal{D}(\mathcal{P}))$.  To define $\hat{\rho_{1}}$ we proceed as follows. Let $\mathbb{P}(\mathcal{D}^{\infty}, \mathbf{p}^{\infty})$ and similarly $\mathbb{P}(\mathcal{D}(\mathcal{P}))$ be the sets of edge-paths in $(\mathcal{D}^{\infty}, \mathbf{p}^{\infty})$ and $\mathcal{D}(\mathcal{P})$ respectively and define
$$\rho: \mathbf{r}^{\infty} \rightarrow \mathbb{P}(\mathcal{D}(\mathcal{P}))$$
as follows. For every $r \in \mathbf{r}^{\infty} \setminus \mathbf{r}$, 
$$\rho(r)=e_{1}^{\varepsilon_{1}} \circ \dots \circ e_{n}^{\varepsilon_{n}}$$ 
where $e_{1}^{\varepsilon_{1}} \circ \dots \circ e_{n}^{\varepsilon_{n}}$ is a fixed path in $\mathbb{P}(\mathcal{D}(\mathcal{P}))$ connecting $r_{+}$ with $r_{-}$, and if $r \in \mathbf{r}$, then 
$$\rho(r)=(\lambda, r,1, \lambda).$$ 
Letting $\mathbb{E}(\mathcal{D}^{\infty}, \mathbf{p}^{\infty})$ be the set of edges of $(\mathcal{D}^{\infty}, \mathbf{p}^{\infty})$, we see that $\rho$ induces a map 
$$\rho_{\mathbb{E}}: \mathbb{E}(\mathcal{D}^{\infty}, \mathbf{p}^{\infty}) \rightarrow \mathbb{P}(\mathcal{D}(\mathcal{P})),$$
by setting for each edge $f=(u,r,\delta,v)$,
$$\rho_{\mathbb{E}}(f)=\rho_{\mathbb{E}}(u,r,\delta,v)=u.(\rho(r))^{\delta}.v.$$
We remark for later use that for every positive edge $f=(u,r,1,v)$,
\begin{equation} \label{rho-}
\rho_{\mathbb{E}}(f^{-1})=\rho_{\mathbb{E}}(u,r,-1,v)=u.(\rho(r))^{-1}.v=(\rho_{\mathbb{E}}(f))^{-1},
\end{equation}
and that for every edge $f=(u,r,\delta,v)$ and every $w,w' \in F$,
\begin{equation} \label{rho-s}
\rho_{\mathbb{E}}(w.f.w')=\rho_{\mathbb{E}}((wu,r,\delta,v.w'))=wu.(\rho(r))^{\delta}.vw'=w.\rho_{\mathbb{E}}(f).w'.
\end{equation}
Further we see that for every edge $f=(u,r,\delta,v)$, if $r \in \mathbf{r}$, then 
$$\rho_{\mathbb{E}}(f)=\rho_{\mathbb{E}}(u,r,\delta,v)=u.(\rho(r))^{\delta}.v.=u.(\lambda, r,1, \lambda)^{\delta}.v=(u,r,\delta,v)=f,$$
and if $r \in \mathbf{r}^{\infty} \setminus \mathbf{r}$ with $\rho(r)=e_{1}^{\varepsilon_{1}} \circ \dots \circ e_{n}^{\varepsilon_{n}}$, then 
$$\rho_{\mathbb{E}}(f)=\rho_{\mathbb{E}}(u,r,\delta,v)=u.(\rho(r))^{\delta}.v= (u. e^{\varepsilon_{1}}_{1}.v \circ \dots \circ u.e^{\varepsilon_{n}}_{n}.v)^{\delta},$$
is a path which connects $\iota(f)=ur_{+\delta}v$ with $\tau(f)=ur_{-\delta}v$, therefore
$$\iota \rho_{\mathbb{E}}(f)= \iota f \text{  and } \tau \rho_{\mathbb{E}}(f)= \tau f.$$
It follows from this that if $f=(u,r, \delta, v)$ and $e=(w, s, \varepsilon, w')$ are edges such that $\tau e=\iota f$, then 
$$\tau \rho_{\mathbb{E}}(e)=\tau e=\iota f= \iota \rho_{\mathbb{E}}(f).$$
This implies that the mapping $\rho_{\mathbb{E}}$ induces another mapping
$$\rho_{\mathbb{P}}: \mathbb{P}(\mathcal{D}^{\infty}, \mathbf{p}^{\infty}) \rightarrow \mathbb{P}(\mathcal{D}(\mathcal{P})),$$
by setting for each path $p=f_{1} \circ \dots \circ f_{r}$ in $\mathbb{P}(\mathcal{D}^{\infty}, \mathbf{p}^{\infty})$,
$$\rho_{\mathbb{P}}(p)= \rho_{\mathbb{E}}(f_{1}) \circ \dots \circ \rho_{\mathbb{E}}(f_{r}).$$
It follows from this that 
$$\iota \rho_{\mathbb{P}}(p)=\iota \rho_{\mathbb{E}}(f_{1})=\iota f_{1}$$
and
$$\tau \rho_{\mathbb{P}}(p)=\tau \rho_{\mathbb{E}}(f_{r})=\tau f_{r}.$$
In particular this means that if $p$ is a closed path at some $w$, then $\rho_{\mathbb{P}}(p)$ is again a closed path at $w$.

Relative to $\rho$ we define a $(\mathbb{Z}F, \mathbb{Z}F)$-bimodule homomorphism 
$$\hat{\rho_{1}}: C_{1}(\mathcal{D}^{\infty}, \mathbf{p}^{\infty}) \rightarrow C_{1}(\mathcal{D}(\mathcal{P}))$$
as follows. We first let
$$\psi: \mathbb{P}(\mathcal{D}(\mathcal{P})) \rightarrow C_{1}(\mathcal{D}(\mathcal{P}))$$
be the mapping defined by 
$$g_{1}^{\gamma_{1}} \circ \dots \circ g_{r}^{\gamma_{r}} \mapsto \gamma_{1}g_{1}+ \dots + \gamma_{r}g_{r}.$$
Regarding $\psi$ we prove the following lemma which will be used later in this section.
\begin{lemma} \label{psi-}
\begin{itemize}
\item[(i)]If $\alpha, \beta \in \mathbb{P}(\mathcal{D}(\mathcal{P}))$ such that $\beta \circ \alpha$ exists, then $\psi(\beta \circ \alpha)=\psi(\beta)+ \psi(\alpha)$
\item[(ii)] For every $\alpha \in \mathbb{P}(\mathcal{D}(\mathcal{P}))$, $\psi(\alpha^{-1})=-\psi(\alpha)$.
\item[(iii)] For every $\alpha \in \mathbb{P}(\mathcal{D}(\mathcal{P}))$ and every $w,w' \in F$, $\psi(w.\alpha.w')=w.\psi(\alpha).w'$.
\end{itemize}
\end{lemma}
\begin{proof}
(i) If $\beta=g^{\gamma_{1}}_{1} \circ \dots \circ g_{r}^{\gamma_{r}}$ and $\alpha=h_{1}^{\mu_{1}} \circ \dots \circ h_{s}^{\mu_{s}}$, then
\begin{align*}
\psi(\beta \circ \alpha)&= \psi(g^{\gamma_{1}}_{1} \circ \dots \circ g_{r}^{\gamma_{r}} \circ h_{1}^{\mu_{1}} \circ \dots \circ h_{s}^{\mu_{s}})\\
&= \gamma_{1}g_{1}+ \dots + \gamma_{r}g_{r}+ \mu_{1}h_{1}+ \dots + \mu_{s}h_{s}\\
&=\psi(g^{\gamma_{1}}_{1} \circ \dots \circ g_{r}^{\gamma_{r}})+ \psi(h_{1}^{\mu_{1}} \circ \dots \circ h_{s}^{\mu_{s}})\\
&=\psi(\beta)+ \psi(\alpha).
\end{align*}
(ii) If $\alpha=h_{1}^{\mu_{1}} \circ \dots \circ h_{s}^{\mu_{s}}$, then $\alpha^{-1}=h_{s}^{-\mu_{s}} \circ \dots \circ h_{1}^{-\mu_{1}}$ and
\begin{align*}
\psi(\alpha^{-1})&=\psi(h_{s}^{-\mu_{s}} \circ \dots \circ h_{1}^{-\mu_{1}})\\
&= -\mu_{s}h_{s}- \dots - \mu_{1}h_{1}\\
&-\psi(h_{1}^{\mu_{1}} \circ \dots \circ h_{s}^{\mu_{s}})\\
&=-\psi(\alpha).
\end{align*}
(iii) If $\alpha=h_{1}^{\mu_{1}} \circ \dots \circ h_{s}^{\mu_{s}}$, then $w.\alpha.w'=w.h_{1}^{\mu_{1}}.w' \circ \dots \circ w.h_{s}^{\mu_{s}}.w'$ and
\begin{align*}
\psi(w.\alpha.w')&= \psi(w.h_{1}^{\mu_{1}}.w' \circ \dots \circ w.h_{s}^{\mu_{s}}.w')\\
&= \psi (w.h_{1}^{\mu_{1}}.w') + \dots + \psi (w.h_{s}^{\mu_{s}}.w') \tag{by (i)}\\
&= \mu_{1} w.h_{1}.w'+ \dots + \mu_{s} w.h_{s}.w'\\
&=w.(\mu_{1}h_{1}+ \dots + \mu_{s}h_{s}).w'\\
&= w.\psi(h_{1}^{\mu_{1}} \circ \dots \circ h_{s}^{\mu_{s}}).w'\\
&= w.\psi(\alpha).w'.
\end{align*}
\end{proof}
Now we define for every abelian group generator $(u,r,1,v)$ of $C_{1}(\mathcal{D}^{\infty}, \mathbf{p}^{\infty})$,
$$\hat{\rho_{1}}(u, r, 1, v) = \psi (\rho_{\mathbb{E}}(u,r,1,v)).$$
This assignment induces a group homomorphism which is also a $(\mathbb{Z}F, \mathbb{Z}F)$-bimodule homomorphism. We need only to check that for every $w,w' \in F$ and every $(u,r,1,v)$, 
$$\hat{\rho_{1}}(w.(u, r, 1, v).w')=w.\hat{\rho_{1}}(u, r, 1, v).w'.$$
If $r \in \mathbf{r}$, then
\begin{align*}
\hat{\rho_{1}}(w.(u, r, 1, v).w')&=\hat{\rho_{1}}(wu,r,1,vw')\\
&= \psi \rho_{\mathbb{E}}(wu,r,1,vw')\\
&=\psi (wu,r,1,vw')\\
&= (wu,r,1,vw')\\
&=w.(u,r,1,v).w'\\
&= w. (\psi \rho_{\mathbb{E}}(u,r,1,v)).w'\\
&= w. \hat{\rho_{1}}(u,r,1,v).w'.
\end{align*}
If $r \in \mathbf{r}^{\infty}\setminus \mathbf{r}$, and $\rho(r)=e_{1}^{\varepsilon_{1}} \circ \dots \circ e_{n}^{\varepsilon_{n}}$, then
\begin{align*}
\hat{\rho_{1}}(w.(u, r, 1, v).w')&=\hat{\rho_{1}}(wu,r,1,vw')\\
&= \psi \rho_{\mathbb{E}}(wu,r,1,vw')\\
&= \psi (wu. e^{\varepsilon_{1}}_{1}.vw' \circ \dots \circ wu.e^{\varepsilon_{n}}_{n}.vw') \tag{by (\ref{rho-s})}\\
&= \varepsilon_{1}wu.e_{1}.vw'+ \dots + \varepsilon_{n}wu.e_{n}.vw'\\
&=w.(\varepsilon_{1}u.e_{1}.v+ \dots + \varepsilon_{n}u.e_{n}.v).w'\\
&= w. (\psi \rho_{\mathbb{E}}(u, r, 1, v)).w'\\
&= w. \hat{\rho_{1}} (u, r, 1, v). w'.
\end{align*}
Finally, since for every $r \in \mathbf{r}$, and $u,v \in F$,
\begin{equation} \label{retr}
\hat{\rho_{1}}(u, r, 1, v) = \psi (\rho_{\mathbb{E}}(u,r,1,v))=\psi (u,r,1,v)=(u,r,1,v),
\end{equation}
it follows that $\hat{\rho_{1}}$ is a retraction.

Let now
$$\mathbf{q}=\{\rho_{\mathbb{P}}(p): p \in \mathbf{p}^{\infty}\}$$
and $(\mathcal{D}(\cP), \mathbf{q})$ the 3-complex of section \ref{gc}. We note that complexes $\mathcal{D}(\cP)$ and $(\mathcal{D}(\cP), \mathbf{q})$ have the same 1-skeleton therefore $C_{1}(\mathcal{D}(\cP))=C_{1}(\mathcal{D}(\cP), \mathbf{q})$. It follows that our retraction $\hat{\rho_{1}}$ is also a retraction 
$$\hat{\rho_{1}}: C_{1}(\mathcal{D}^{\infty}, \mathbf{p}^{\infty}) \rightarrow C_{1}(\mathcal{D}(\cP), \mathbf{q}).$$
Next we will define a homomorphism of free abelian groups
$$\hat{\rho_{2}}: C_{2}(\mathcal{D}^{\infty}, \mathbf{p}^{\infty}) \rightarrow C_{2}(\mathcal{D}(\cP), \mathbf{q}),$$
as follows. The generators of $C_{2}(\mathcal{D}^{\infty}, \mathbf{p}^{\infty})=C_{2}(\mathcal{D}^{\infty}) \oplus C_{2}^{\mathbf{p}^{\infty}}$ are either square 2-cells $[e,f]$ or 2-cells $[u,p,v]$ with $u,v \in F$ and $p \in \mathbf{p}^{\infty}$. If $[e,f]$ is a square 2-cell where $\hat{\rho_{1}}(e)=\sum_{i} \alpha_{i} e_{i}$ and $\hat{\rho_{1}}(f)=\sum_{j} \beta_{j} f_{j}$, then we define a homomorphism of free abelian groups
$$\hat{\rho_{2}}': C_{2}(\mathcal{D}^{\infty}) \rightarrow C_{2}(\mathcal{D}(\cP)),$$
by specifying it values on the generators:
$$\hat{\rho_{2}}'([e,f])=\sum_{i,j} \alpha_{i} \beta_{j} [e_{i},f_{j}].$$
Note that when both $e$ and $f$ are edges of $\mathcal{D}(\cP)$, then $\hat{\rho_{1}}(e)=e$, $\hat{\rho_{1}}(f)=f$ and $\hat{\rho_{2}}'([e,f])=[e,f]$. Now we define a second homomorphism of free abelian groups
$$\hat{\rho_{2}}'': C_{2}^{\mathbf{p}^{\infty}} \rightarrow C^{\mathbf{q}}_{2},$$
on each generator $[u,p,v]$ by
$$\hat{\rho_{2}}''([u,p,v])=[u, \rho_{\mathbb{P}}(p), v].$$
The pair $(\hat{\rho_{2}}',\hat{\rho_{2}}'')$ induces the desired homomorphism
$$\hat{\rho_{2}}: C_{2}(\mathcal{D}^{\infty}, \mathbf{p}^{\infty}) \rightarrow C_{2}(\mathcal{D}(\cP), \mathbf{q}).$$
We prove the following.
\begin{lemma} \label{rho}
	For every path $\rho=f_{1}^{\beta_{1}}\circ \dots \circ f_{n}^{\beta_{n}}$ in $\mathcal{D}(\mathcal{P})$ where every $f_{j}$ is a positive edge and every $\beta_{j}=\pm 1$, we have that
	$$\partial_{1}(\beta_{1}f_{1} +\dots + \beta_{n}f_{n})= \tau (\rho)- \iota (\rho).$$
\end{lemma}
\begin{proof}
	The proof will be done by induction on $n$. For $n=1$, 
	$$\partial_{1}(\beta_{1}f_{1})=\beta_{1}(\tau f_{1}- \iota f_{1}),$$
	therefore, depending on the sign of $\beta_{1}$, we have that $\partial_{1}(\beta_{1}f_{1})=\tau(f_{1}^{\beta_{1}})-\iota(f_{1}^{\beta_{1}})$. For the inductive step, we write
	$$\rho=f_{1}^{\beta_{1}} \dots f_{n}^{\beta_{n}}\cdot f_{n+1}^{\beta_{n+1}}=\rho_{1} \cdot f_{n+1}^{\beta_{n+1}}.$$ 
	From the assumption for $\rho_{1}$ we have
	$$\partial_{1}(\beta_{1}f_{1} +\dots + \beta_{n}f_{n})= \tau (\rho_{1})- \iota (\rho_{1})=\iota(f_{n+1}^{\beta_{n+1}})-\iota(\rho),$$
	and then
	\begin{align*}
		\partial_{1}(\beta_{1}f_{1} +\dots + \beta_{n}f_{n} +\beta_{n+1}f_{n+1})&= \partial_{1}(\beta_{1}f_{1} +\dots + \beta_{n}f_{n})+ \partial_{1}(\beta_{n+1}f_{n+1})\\
		&= \iota(f_{n+1}^{\beta_{n+1}})-\iota(\rho)+ (\tau(f_{n+1}^{\beta_{n+1}})-\iota(f_{n+1}^{\beta_{n+1}}))\\
		&=\tau(f_{n+1}^{\beta_{n+1}})-\iota(\rho)\\
		&=\tau(\rho)-\iota(\rho).
	\end{align*}
\end{proof}
\begin{lemma} \label{rho-p}
For every positive edge $f$ of $(\mathcal{D}^{\infty}, \mathbf{p}^{\infty})$, $\partial_{1} \hat{\rho_{1}} (f)=\tau f- \iota f$.
\end{lemma}
\begin{proof}
If $f$ is a positive edge of $\mathcal{D}(\cP)$, then
\begin{align*}
\partial_{1} \hat{\rho_{1}} (f)&=\partial_{1}(f) \tag{definition of $\hat{\rho_{1}}$}\\
&= \tau f- \iota f.
\end{align*}
If $f$ is from $(\mathcal{D}^{\infty}, \mathbf{p}^{\infty})$ but not from $\mathcal{D}(\mathcal{P})$, say $f=(u,r,1,v)$, then $r \in \mathbf{r}^{\infty} \setminus \mathbf{r}$ and we can write $f$ as $u.f'.v$ where $f'=(\lambda, r,1, \lambda)$. Assuming now that 
$$\rho_{\mathbb{E}}(f')=f_{1}^{\beta_{1}}\circ \dots \circ f_{n}^{\beta_{n}},$$
we have
\begin{align*}
\partial_{1} \hat{\rho_{1}} (f)&=\partial_{1} \hat{\rho_{1}} (u.f'.v)\\
&=u.\partial_{1}(\hat{\rho_{1}}(f')).v \tag{$\partial_{1}$ and $\hat{\rho_{1}}$ are $(\mathbb{Z}F,\mathbb{Z}F)$-morphisms}\\
&= u. \partial_{1} (\beta_{1}f_{1}+\dots +\beta_{n}f_{n}).v \tag{definition of $\hat{\rho_{1}}$}\\
&= u.(r_{-}-r_{+}).v \tag{Lemma \ref{rho}}\\
&= \tau f - \iota f.
\end{align*}
\end{proof}

Finally, we define a third homomorphism of free abelian groups
$$\hat{\rho_{3}}: C_{3}(\mathcal{D}^{\infty}, \mathbf{p}^{\infty}) \rightarrow C_{3}(\mathcal{D}(\mathcal{P}), \mathbf{q}))$$
as follows. If $e$ is any positive edge with $\hat{\rho_{1}}(e)=\sum_{i}\varepsilon_{i}e_{i}$ and $\sigma$ a 2-cell in $(\mathcal{D}^{\infty}, \mathbf{p}^{\infty})$ such that $\hat{\rho_{2}}(\sigma)=\sum_{j} \mu_{j} \sigma_{j}$, then we define
$$\hat{\rho_{3}}([e,\sigma])=\sum_{i,j} \varepsilon_{i} \mu_{j} [e_{i}, \sigma_{j}] \text{ and } \hat{\rho_{3}}([\sigma,e])=\sum_{i,j} \varepsilon_{i} \mu_{j} [\sigma_{j},e_{i}].$$
We recall here that when $\sigma=[u,p,v]$ with $u,v \in F$ and $p \in \mathbf{p}^{\infty}$, then $\hat{\rho_{2}}(\sigma)=[u, \rho_{\mathbb{P}}(p),v]$, and when $\sigma=[g,h]$ with $\hat{\rho}_{1}(g)=\sum_{j} \beta_{j}g_{j}$ and $\hat{\rho}_{1}(h)=\sum_{s} \gamma_{s}h_{s}$, then $\hat{\rho_{2}}(\sigma)=\sum_{j,s}\beta_{j}\gamma_{s}[g_{j}, h_{s}]$.

\begin{lemma} \label{chm}
The following hold true.
\begin{itemize}
\item[(i)] $id_{C_{0}}  \partial^{\infty}_{1}=\partial_{1}  \hat{\rho_{1}}$.
\item[(ii)] $\hat{\rho_{1}}\tilde{\partial}^{\infty}_{2}=\tilde{\partial}_{2}\hat{\rho_{2}}$.
\item[(iii)] $\hat{\rho_{2}}\tilde{\partial}^{\infty}_{3}=\tilde{\partial}_{3}\hat{\rho_{3}}$.
\end{itemize}
\end{lemma}
\begin{proof}
(i) Let $f$ be a generator of the free abelian group $C_{1}(\mathcal{D}^{\infty})$. We have that
\begin{align*}
\partial_{1}(\hat{\rho_{1}}(f))&=\tau f - \iota f \tag{by Lemma \ref{rho-p}}\\
&= \partial^{\infty}_{1}(f)\\
&=id_{C_{0}}(\partial^{\infty}_{1}(f)).
\end{align*}
(ii) Let first $[u,p,v]$ be a generator of $C_{2}^{\mathbf{p}^{\infty}}$ with $p=g_{1}^{\gamma_{1}} \circ \dots \circ g_{r}^{\gamma_{r}}$, then
$$\tilde{\partial}^{\infty}_{1} ([u, p,v])= \gamma_{1}u.g_{1}.v+ \dots + \gamma_{r}u. g_{r}.v.$$
It follows that
\begin{align*}
\hat{\rho_{1}}(\tilde{\partial}^{\infty}_{2}([u,p,v]))&= \hat{\rho_{1}}(\gamma_{1}u.g_{1}.v+ \dots + \gamma_{r}u. g_{r}.v.)\\
&= \gamma_{1} \hat{\rho_{1}}(u.g_{1}.v) + \dots + \gamma_{r} \hat{\rho_{1}}(u.g_{r}.v)\\
&= \gamma_{1} \psi \rho_{\mathbb{E}}(u.g_{1}.v)+ \dots + \gamma_{r} \psi \rho_{\mathbb{E}}(u.g_{r}.v)\\
&= u.(\gamma_{1} \psi \rho_{\mathbb{E}}(g_{1})+ \dots + \gamma_{r} \psi \rho_{\mathbb{E}}(g_{r})).v \tag{by (\ref{rho-s} and Lemma \ref{psi-})}\\
&= u.(\psi ((\rho_{\mathbb{E}}(g_{1}))^{\gamma_{1}})+ \dots + \psi ((\rho_{\mathbb{E}}(g_{r}))^{\gamma_{r}})).v \tag{by Lemma \ref{psi-}}\\
&= u.(\psi((\rho_{\mathbb{E}}(g_{1}))^{\gamma_{1}}\circ \dots \circ (\rho_{\mathbb{E}}(g_{r}))^{\gamma_{r}})).v \tag{by Lemma \ref{psi-}}\\
&=u.(\psi (\rho_{\mathbb{E}}(g^{\gamma_{1}}_{1})\circ \dots \circ \rho_{\mathbb{E}}(g^{\gamma_{r}}_{r}))).v \tag{by (\ref{rho-})}\\
&= u.(\psi (\rho_{\mathbb{P}}(g_{1}^{\gamma_{1}} \circ \dots \circ g_{r}^{\gamma_{r}}))).v\\
&= u. (\psi (\rho_{\mathbb{P}}(p))).v\\
&=u. \tilde{\partial}_{2} (\rho_{\mathbb{P}}(p)).v\\
&= \tilde{\partial}_{2}([u,\rho_{\mathbb{P}}(p),v])\\
&= \tilde{\partial}_{2}(\hat{\rho_{2}}([u,p,v])).
\end{align*}
Next we discuss the case when the generator is a square 2-cell $[e,f]$ and that $\hat{\rho_{1}}(e)=\sum_{i}\alpha_{i}e_{i}$ and $\hat{\rho_{1}}(f)=\sum_{j}\beta_{j}f_{j}$. Then we have
\begin{align*}
\hat{\rho}_{1}(\tilde{\partial}^{\infty}_{2}([e,f]))&= \hat{\rho}_{1}(e.(\iota f- \tau f)-(\iota e-\tau e).f)\\
&=\sum_{i}\alpha_{i}e_{i}\cdot \sum_{j}\beta_{j}(\iota f_{j}-\tau f_{j})-\sum_{i}\alpha_{i}(\iota e_{i}-\tau e_{i})\cdot \sum_{j} \beta_{j} f_{j} && \text{(Lemma \ref{rho})}\\
&= \sum_{i,j} \alpha_{i} \beta_{j} \tilde{\partial}_{2}[e_{i},f_{j}]\\
&= \tilde{\partial}_{2}\left(\sum_{i,j} \alpha_{i} \beta_{j} [e_{i},f_{j}]\right)\\
&= \tilde{\partial}_{2}(\hat{\rho_{2}}([e,f])).
\end{align*}
(iii) Let $[e,\sigma]$ be a 3-cell where $\sigma$ is any 2-cell in $(\mathcal{D}^{\infty}, \mathbf{p}^{\infty})$ with $\hat{\rho_{2}}(\sigma)=\sum_{j} \mu_{j} \sigma_{j}$ and $e$ is a positive edge with $\rho_{\mathbb{E}}(e)=e_{1}^{\varepsilon_{1}} \circ \dots \circ e_{n}^{\varepsilon_{n}}$ hence $\hat{\rho_{1}}(e)=\sum_{i}\varepsilon_{i}e_{i}$. Assume also that $\tilde{\partial}^{\infty}_{2}(\sigma)= \sum_{k} \delta_{k} f_{k}$ where for each $k$, $\hat{\rho_{1}}(f_{k})= \sum_{s}\beta_{ks}g_{ks}$. It follows that from (ii) that
\begin{equation} \label{brz}
\tilde{\partial}_{2}\left(\sum_{j} \mu_{j} \sigma_{j}\right)=\tilde{\partial}_{2}(\hat{\rho_{2}}(\sigma))	=\hat{\rho_{1}}\tilde{\partial}^{\infty}_{2}(\sigma)=\sum_{k}\delta_{k}\sum_{s}\beta_{ks} g_{ks}.
\end{equation}
If for each $j$ we let $\mathcal{T}_{\sigma_{j}}$ be the set of terms $\alpha_{j} x_{j}$ of $\tilde{\partial}_{2}(\sigma_{j})$ where $\alpha_{j}=\pm 1$ and $x_{j}$ a positive edge, it follows from (\ref{brz}) and from Lemma \ref{zro} that for each $e_{i}$ we have that
\begin{equation} \label{rwe}
\sum_{j} \mu_{j}\sum_{\alpha_{j} x_{j} \in \mathcal{T}_{\sigma_{j}}}\alpha_{j}[e_{i},x_{j}]=\sum_{k}\delta_{k}\sum_{s}\beta_{ks} [e_{i},g_{ks}].
\end{equation}
Further we see that
\begin{align*}
	\hat{\rho_{2}}\tilde{\partial}^{\infty}_{3}([e,\sigma])&=\hat{\rho_{2}}\left((\tau e- \iota e).\sigma- \sum_{k} \delta_{k}[e,f_{k}]\right)\\
	&=\sum_{j}\mu_{j}(\tau e - \iota e).\sigma_{j}-\sum_{k} \delta_{k}\left(\sum_{i}\varepsilon_{i}\sum_{s} \beta_{ks}[e_{i},g_{ks}]\right)  \tag{definition of $\hat{\rho_{2}}$}\\
	&=\sum_{j}\mu_{j}\left(\sum_{i}\varepsilon_{i}(\tau e_{i} - \iota e_{i})\right).\sigma_{j}-\sum_{k} \delta_{k}\left(\sum_{i}\varepsilon_{i}\sum_{s} \beta_{ks}[e_{i},g_{ks}]\right) && \tag{Lemma \ref{rho}}\\
	&=\sum_{i}\varepsilon_{i}\left(\sum_{j}\mu_{j}(\tau e_{i} - \iota e_{i})\right).\sigma_{j}-\sum_{i} \varepsilon_{i}\left(\sum_{k}\delta_{k}\sum_{s} \beta_{ks}[e_{i},g_{ks}]\right)\\
	&=\sum_{i}\varepsilon_{i} \left(\sum_{j}\mu_{j}(\tau e_{i} - \iota e_{i}).\sigma_{j}- \sum_{k}\delta_{k}\sum_{s} \beta_{ks}[e_{i},g_{ks}]\right)\\
	&=\sum_{i}\varepsilon_{i} \left(\sum_{j}\mu_{j}(\tau e_{i} - \iota e_{i}).\sigma_{j}- \sum_{j} \mu_{j}\sum_{\alpha_{j} x_{j} \in \mathcal{T}_{\sigma_{j}}}\alpha_{j}[e_{i},x_{j}]\right) \tag{by (\ref{rwe})}\\
	&=\sum_{i,j}\varepsilon_{i} \mu_{j}\left((\tau e_{i} - \iota e_{i}).\sigma_{j}-\sum_{\alpha_{j} x_{j} \in \mathcal{T}_{\sigma_{j}}}\alpha_{j}[e_{i},x_{j}]\right)\\
	&=\tilde{\partial}_{3}\left(\sum_{i,j}\varepsilon_{i} \mu_{j}[e_{i}, \sigma_{j}]\right)  \tag{by (\ref{d3a})}\\
	&=\tilde{\partial}_{3}\hat{\rho_{3}}([e,\sigma]).
\end{align*}
The proof for the 3-cell $[\sigma, e]$ is similar to the above and is omitted here. 
\end{proof}

It follows from Lemma \ref{chm} that the following diagram commutes
$$
\xymatrix{C_{3}(\mathcal{D}^{\infty}, \mathbf{p}^{\infty}) \ar[d]_{\hat{\rho_{3}}} \ar[r]^(0.50){\tilde{\partial}_{2}^{\infty}} & C_{2}(\mathcal{D}^{\infty}) \oplus C^{\mathbf{p}^{\infty}}_{2} \ar[d]_{\hat{\rho_{2}}} \ar[r]^(0.60){\tilde{\partial}_{2}^{\infty}} & C_{1}(\mathcal{D}^{\infty}) \ar[d]_{\hat{\rho_{1}}} \ar[r]^(0.55){\partial^{\infty}_{1}} & C_{0} \ar[d]^{id_{C_{0}}}\\
C_{3}(\mathcal{D}, \mathbf{q}) \ar[r]^(0.50){\tilde{\partial}_{2}} & C_{2}(\mathcal{D}) \oplus C_{2}^{\mathbf{q}} \ar[r]^(0.55){\tilde{\partial}_{2}} & C_{1}(\mathcal{D}) \ar[r]^(0.55){\partial_{1}} & C_{0}
}
$$
This and the facts that $\hat{\rho_{1}}$ is epi, $id_{C_{0}}$ is mono and that $H_{1}(\mathcal{D}^{\infty}, \mathbf{p}^{\infty})=0$, imply that $H_{1}(\mathcal{D}, \mathbf{q})=0$ which means that $\mathbf{q}$ is a homology trivializer for $\mathcal{D}(\cP)$.

\begin{lemma} \label{pbf}
For every $p \in \mathbf{p}^{\infty}$, if $p=f_{1}^{\delta_{1}} \circ \dots \circ f_{n}^{\delta_{n}}$ where each $f_{i}$ is a positive edge and $\delta_{i}=\pm 1$, then $\tilde{\partial}_{2}([\lambda, \rho_{\mathbb{P}}(p), \lambda])=\sum_{i=1}^{n}\delta_{i} \hat{\rho_{1}}(f_{i})$.
\end{lemma}
\begin{proof}
Let for each $i=1,...,n$, 
$$\rho_{\mathbb{E}}(f_{i})= e^{\varepsilon_{i,1}}_{i,1} \circ \dots \circ e^{\varepsilon_{1,k_{i}}}_{i,k_{i}},$$
then from (\ref{rho-})
$$\rho_{\mathbb{E}}(f_{i}^{\delta_{i}})= (e^{\varepsilon_{i,1}}_{i,1} \circ \dots \circ e^{\varepsilon_{1,k_{i}}}_{i,k_{i}})^{\delta_{i}}.$$
The definition of $\rho_{\mathbb{P}}(p)$ now implies that
$$\rho_{\mathbb{P}}(p)= (e^{\varepsilon_{1,1}}_{1,1} \circ \dots \circ e^{\varepsilon_{1,k_{1}}}_{1,k_{1}})^{\delta_{1}} \circ \dots \circ (e^{\varepsilon_{n,1}}_{n,1} \circ \dots \circ e^{\varepsilon_{n,k_{n}}}_{n,k_{n}})^{\delta_{n}},$$
consequently
\begin{align*}
\tilde{\partial}_{2}([\lambda, \rho_{\mathbb{P}}(p), \lambda])&= \sum_{i=1}^{n} \delta_{i}\sum_{j=1}^{k_{i}}\varepsilon_{i,j}e_{i,j}\\
&= \sum_{i=1}^{n} \delta_{i} \psi (\rho_{\mathbb{E}}(f_{i}))\\
&=\sum_{i=1}^{n} \delta_{i}\hat{\rho_{1}}(f_{i}).
\end{align*}
\end{proof}

\begin{proposition} \label{p14} (Proposition 14, \cite{SES})
If $\mathbf{p}$ is an arbitrary homology trivializer for $H_{1}(\mathcal{D}(\cP))$, then the sequence (\ref{cxkp}) is exact.
\end{proposition}
\begin{proof}
Since $\mathbf{q}=\{\rho_{\mathbb{P}}(p): p \in \mathbf{p}^{\infty}\}$ is a homology trivializer for $\mathcal{D}(\cP)$, from Lemma 13 it is enough to prove that the sequence
\begin{equation} \label{cxkqq}
	\xymatrix{0 \ar[r] & B_{2}(\mathcal{D}(\cP),\mathbf{q}) \ar[r]^-{incl.} & K^{\mathbf{q}} \ar[r]^-{\tilde{\partial}_{2}} & B_{1}(\mathcal{D}(\cP))  \ar[r] & 0}
\end{equation}	
is exact. Let $\xi$ be a 2-cycle in $K^{\mathbf{q}}$. We can write $\xi$ as a sum
$$\xi=\xi'+ \sum_{i=1}^{m} \varepsilon_{i}(\iota e_{i}-\tau e_{i}).[\lambda,q_{i}, w_{i}] +\sum_{i=m+1}^{m+n} \varepsilon_{i} [w_{i}, q_{i}, \lambda].(\iota e_{i}-\tau e_{i})$$
where $\xi'$ is a 2-chain in $C_{2}(\mathcal{D}(\cP))$, and for $i=1,...,m+n$, $\varepsilon_{i}=\pm 1$, $w_{i} \in F$, $e_{i}$ is a positive edge in $\mathcal{D}(\cP)$ and $q_{i} \in \mathbf{q}$. 

Now for each $q \in \mathbf{q}$ choose some $\tilde{q} \in \mathbf{p}^{\infty}$ such that $\rho_{\mathbb{P}}(\tilde{q})=q$. For each edge $f^{\delta}$ in the defining path $\partial \tilde{q}$ of $\tilde{q}$ where $f$ is a positive edge and $\delta=\pm 1$, $(\rho_{\mathbb{E}}(f))^{\delta}\circ f^{-\delta}$ is a closed path in $\mathcal{D}(\cP^{\infty})$ and so from Lemma \ref{rho} it gives rise to the 1-cycle $\delta(\hat{\rho_{1}}(f)-f)$. Hence, since $\mathbf{p}^{\infty}$ is a homology trivializer, there is a 2-chain $c_{2}(q)$ in $C_{2}(\mathcal{D}(\cP^{\infty}), \mathbf{p}^{\infty})$ such that
$$\tilde{\partial}_{2}^{\infty} (c_{2}(q))= \sum_{f^{\delta} \in \partial \tilde{q}}\delta(\hat{\rho_{1}}(f)-f).$$
Note that $\hat{\rho_{2}}(c_{2}(q))$ is a 2-cycle since
\begin{align*}
\tilde{\partial}_{2}(\hat{\rho_{2}}(c_{2}(q)))&= \hat{\rho_{1}}(\tilde{\partial}_{2}^{\infty}(c_{2}(q))) \tag{Lemma \ref{chm}}\\
&=\sum_{f^{\delta} \in \partial \tilde{q}}\delta (\hat{\rho_{1}}(f)-\hat{\rho_{1}}(f)). \tag{by (\ref{retr})}\\
&=0.
\end{align*}
Consider now the 2-chain
\begin{align*}
\varsigma=&\xi'+ \sum_{i=1}^{m} \varepsilon_{i}(\iota e_{i}-\tau e_{i}).([\lambda,\tilde{q}_{i}, \lambda]+c_{2}(q_{i})).w_{i}\\
& +\sum_{i=m+1}^{m+n} \varepsilon_{i} w_{i}.([\lambda, \tilde{q}_{i}, \lambda]+c_{2}(q_{i})).(\iota e_{i}-\tau e_{i})
\end{align*}
in $K^{\mathbf{p^{\infty}}}$. Since for each $q \in \mathbf{q}$,
\begin{align*}
\tilde{\partial}_{2}^{\infty}([\lambda, \tilde{q}, \lambda]+c_{2}(q)) &= \sum_{f^{\delta} \in \partial \tilde{q}} \delta f+\sum_{f^{\delta} \in \partial \tilde{q}} \delta (\hat{\rho_{1}}(f)-f)\\
&=\sum_{f^{\delta} \in \partial \tilde{q}} \delta \hat{\rho_{1}}(f)\\
&=\tilde{\partial}_{2}[\lambda, q, \lambda] \tag{by Lemma \ref{pbf}},
\end{align*}
we have that
$$\tilde{\partial}_{2}^{\infty}(\varsigma)=\tilde{\partial}_{2}(\xi)=0,$$
so that $\varsigma$ in $K^{\mathbf{p^{\infty}}}$ is a 2-cycle. Moreover,
\begin{align}
\hat{\rho_{2}}(\varsigma)&=\xi + \sum_{i=1}^{m}\varepsilon_{i}(\iota e_{i}-\tau e_{i}).\hat{\rho_{2}}(c_{2}(q_{i})).w_{i}+ \sum_{i=m+1}^{m+n}\varepsilon_{i}w_{i}.\hat{\rho_{2}}(c_{2}(q_{i})).(\iota e_{i}-\tau e_{i}) \notag\\
&=\xi  \text{ (}\text{modulo } B_{2}(\mathcal{D}(\cP), \mathbf{q})) \label{ff}
\end{align}
because, as in the proof of Lemma 11 of \cite{SES}, if $z$ is any 2-cycle and $e$ a positive edge, then both $( \iota e - \tau e).z$ and $z.( \iota e -\tau e)$ are null-homologous. But from the special case of Proposition 14 we have that the sequence
\begin{equation*} 
	\xymatrix{0 \ar[r] & B_{2}(\mathcal{D}(\cP^{\infty}),\mathbf{p}^{\infty}) \ar[r]^-{incl.} & K^{\mathbf{p}^{\infty}} \ar[r]^-{\tilde{\partial}^{\infty}_{2}} & B_{1}(\mathcal{D}(\cP^{\infty}))  \ar[r] & 0}
\end{equation*}	
is exact, therefore there is some $\omega \in C_{3}(\mathcal{D}(\cP^{\infty}),\mathbf{p}^{\infty})$ such that $\varsigma=\tilde{\partial}_{3}^{\infty}(\omega)$. But then the 3-chain $\hat{\rho_{3}}(\omega)$ in $C_{3}(\mathcal{D}(\cP), \mathbf{q})$ has boundary
$$\tilde{\partial}_{3}(\hat{\rho_{3}}(\omega))=\hat{\rho_{2}}(\tilde{\partial}_{3}^{\infty}(\omega))=\hat{\rho_{2}}(\varsigma),$$
which from (\ref{ff}) is homologous to $\xi$, therefore $\xi$ is a 2-boundary and so the sequence (\ref{cxkp}) is exact.
\end{proof}

\subsection{The Pride complex associated with a group presentation} \label{pc}

In this section we will explain several results of Pride in \cite{LDHM2} which are used here to prove our main theorem. The reader also may refer to \cite{LDHM2-rev} for detailed proofs of the corresponding results. Associated to any group presentation $\mathcal{P}=(\mathbf{x},\mathbf{r})$, Pride considers two crossed modules. The first one is the free crossed module $(\Sigma, \hat{F}, \delta)$ associated to $\mathcal{P}=(\mathbf{x},\mathbf{r})$. To define the second, he constructs first a complex $\mathcal{D}(\mathfrak{s})^{\ast}$ arising from the monoid presentation
$$\mathfrak{s}=\langle\mathbf{x}, \mathbf{x}^{-1}: \mathbf{s} \rangle,$$
where
$$\mathbf{s}=\{(r,\lambda): r \in \mathbf{r}\} \cup \left\{ (x^{\varepsilon}x^{-\varepsilon},\lambda): x \in \mathbf{x}, \varepsilon=\pm 1\right\},$$
of the same group given by $\mathcal{P}=(\mathbf{x},\mathbf{r})$. In other words, the group is now realised as the quotient of the free monoid $F$ on $\mathbf{x} \cup \mathbf{x}^{-1}$ by the smallest congruence generated by the relations giving $\mathbf{s}$. We will give below some necessary details on this complex. We first mention that $\mathcal{D}(\mathfrak{s})^{\ast}$ is an extension of the usual Squier complex $\mathcal{D}(\mathfrak{s})$ arising from $\mathfrak{s}$. This complex is called in \cite{Gilbert} the Pride complex. We emphasize here that the definition of $\mathcal{D}(\mathfrak{s})$ in \cite{LDHM2} requires the attachment of 2-cells $[e^{\varepsilon}, f^{\delta}]$ where $e,f$ are positive edges and $\varepsilon, \delta=\pm 1$. But as it is observed in the latter paper \cite{SES} (see Remark 8 there or Remark \ref{abnd} in the present paper), since we are interested in the homotopy and homology of the complex, the attachment of 2-cells $[e^{\varepsilon}, f^{\delta}]$ is unnecessary in the presence of $[e,f]$ because the boundary of $[e^{\varepsilon}, f^{\delta}]$ is a cyclic permutation of $(\partial[e,f])^{\pm 1}$, hence it is null homotopic. So we assume in what follows that $\mathcal{D}(\mathfrak{s})$ is the one described in section \ref{sqc} apart for the notation of positive edges which are now written as in \cite{OK2002}. If $s \in \mathbf{s}$ is a rewriting rule, the corresponding positive edge $(\lambda, s,1,\lambda)$ will be denoted in what follows by $(\lambda, r,\lambda,\lambda)$ if $s=(r,\lambda)$ and $r \in \mathbf{r}$, and by $(\lambda, x^{\varepsilon}x^{-\varepsilon},\lambda, \lambda)$ if $s=(x^{\varepsilon}x^{-\varepsilon}, \lambda)$ where $x \in \mathbf{x}$. To complete the construction of $\mathcal{D}(\mathfrak{s})^{\ast}$ we need to add to $\mathcal{D}(\mathfrak{s})$ certain extra 2-cells along the closed paths 
$$t_{(x,\varepsilon)}=(\lambda,x^{\varepsilon}x^{-\varepsilon}, \lambda, x^{\varepsilon}) \circ (x^{\varepsilon}, x^{-\varepsilon}x^{\varepsilon}, \lambda, \lambda)^{-1} ,$$
where $x \in \mathbf{x}$ and $\varepsilon=\pm1$. We let 
$$\mathbf{t}=\{t_{(x,\varepsilon)}: x \in \mathbf{x}, \varepsilon=\pm 1\}.$$
The attaching of these two cells is done for every overlap of two trivial edges as depicted below
$$\xymatrix{x^{\varepsilon}x^{-\varepsilon}x^{\varepsilon} \ar@/^/[d]^{(x^{\varepsilon}, x^{-\varepsilon}x^{\varepsilon}, \lambda, \lambda)}\ar@/_/[d]_{(\lambda,x^{\varepsilon}x^{-\varepsilon}, \lambda, x^{\varepsilon})} & \\x^{\varepsilon}}$$
The attached 2-cell has boundary made of the following edges
$$\mathbb{A}=(\lambda,x^{\varepsilon}x^{-\varepsilon}, \lambda, x^{\varepsilon}) \text{ and } \mathbb{B}=(x^{\varepsilon}, x^{-\varepsilon}x^{\varepsilon}, \lambda, \lambda).$$
Together with such 2-cells, there are added to the complex all their ''translates'' $[u,t_{(x,\varepsilon)}, v]$ attached along the closed path
$$\xymatrix{ux^{\varepsilon}x^{-\varepsilon}x^{\varepsilon}v \ar@/^/[d]^{(ux^{\varepsilon}, x^{-\varepsilon}x^{\varepsilon}, \lambda, v)}\ar@/_/[d]_{(u,x^{\varepsilon}x^{-\varepsilon}, \lambda, x^{\varepsilon}v)} & \\ux^{\varepsilon}v}$$
In our paper the Pride complex $\mathcal{D}(\mathfrak{s})^{\ast}$ is denoted throughout by $\mathcal{D}(\mathfrak{s}) \sqcup F.\mathbf{t}.F$.

If $P$ is a path in $\mathcal{D}(\mathfrak{s}) \sqcup F.\mathbf{t}.F$ with $\iota(P)=W$ and $\tau(P)=Z$, there are defined in \cite{LDHM2}, $T_{W}$ and $T_{Z}$ to be arbitrary trivial paths from $W^{\ast}$ to $W$ and from $Z^{\ast}$ to $Z$ where $W^{\ast}$ and $Z^{\ast}$ are the unique reduced words freely equivalent to $W$ and $Z$ respectively. Then the path $T_{W} \circ P \circ T_{Z}^{-1}$ is denoted by $P^{\ast}$. The notation is not ambiguous since any two parallel trivial paths in $\mathcal{D}(\mathfrak{s}) \sqcup F.\mathbf{t}.F$ are homotopic. 

Pride has defined in \cite{LDHM2} an $\hat{F}$-crossed module $\Sigma^{\ast}$ out of $\mathcal{D}(\mathfrak{s}) \sqcup F.\mathbf{t}.F$ in the following way. The elements of $\Sigma^{\ast}$ are the homotopy classes $\langle P \rangle$ where $P$ is a path in the 1-skeleton of $\mathcal{D}(\mathfrak{s}) \sqcup F.\mathbf{t}.F$ such that $\tau(P)$ is the empty word and $\iota(P)$ is a freely reduced word from $\hat{F}$. He then defines a (non commutative) operation $+$ on $\Sigma^{\ast}$ by
$$\langle P_{1} \rangle + \langle P_{2} \rangle= \langle (P_{1}+P_{2})^{\ast} \rangle$$
and an action of the free group $\hat{F}$ with basis $\mathbf{x}$ on $\Sigma^{\ast}$ by
$$^{[W]} \langle P \rangle= \langle (WPW^{-1})^{\ast} \rangle.$$
Also he defines 
$$\delta^{\ast}: \Sigma^{\ast} \rightarrow \hat{F}$$
by
$$\delta^{\ast}(\langle P \rangle)= [\iota(P)].$$
It is proved in \cite{LDHM2} that the triple $(\Sigma^{\ast}, \hat{F}, \delta^{\ast})$ is a crossed module. Further, using the fact that $\Sigma$ is the free crossed module over $\cP=(\mathbf{x},\mathbf{r})$, it is proved that 
$$\eta: \Sigma \rightarrow \Sigma^{\ast}$$
defined by $r \mapsto \langle (\lambda, r,1,\lambda)\rangle $ is an isomorphism of crossed modules. The inverse $\psi: \Sigma^{\ast} \rightarrow \Sigma$ of $\eta$ is the map defined in the following way. It is first defined a map $\psi_{0}$ from the set of edges of $\mathcal{D}(\mathfrak{s}) \sqcup F.\mathbf{t}.F$ to $\Sigma$ as follows. Every trivial edge is mapped to $0$, and every edge $(u, r, \lambda, v)^{\varepsilon}$ is mapped to $([u], r)^{\varepsilon}$ where $[u]$ is the element of $\hat{F}$ represented by $u$. It is proved that this map extends to paths of $\mathcal{D}(\mathfrak{s}) \sqcup F.\mathbf{t}.F$ and it sends the boundaries of the defining 2-cells of $\mathcal{D}(\mathfrak{s}) \sqcup F.\mathbf{t}.F$ to 0. We thus have a morphism $\psi: \Sigma^{\ast} \rightarrow \Sigma$ given by
$$\langle P\rangle \mapsto \psi_{0}(P)$$
which is proved to be the inverse of $\eta$. By restriction it is obtained an isomorphism between $\text{Ker}\delta$ and $\text{Ker}\delta^{\ast}$. But $\text{Ker}\delta$ is itself isomorphic to $\pi_{2}(\mathcal{P})$, the second homotopy module of the standard complex associated with $\hat{\mathcal{P}}$, and $\text{Ker}\partial^{\ast}$ on the other hand, is isomorphic to $\pi_{1}(\mathcal{D}(\mathfrak{s}) \sqcup F.\mathbf{t}.F, \lambda)$, the first homotopy group of the connected component of $\mathcal{D}(\mathfrak{s}) \sqcup F.\mathbf{t}.F$ at $\lambda$. Recollecting, we have the following isomorphisms
\begin{equation} \label{ki}
	\xymatrix{\pi_{2}(\mathcal{P}) =\text{Ker}\delta \ar@<2pt>[r]^-(0.40){\eta} &  \text{Ker}\delta^{\ast} \ar@<2pt>[l]^-(0.60){\psi}   {=\pi_{1}(\mathcal{D}(\mathfrak{s}) \sqcup F.\mathbf{t}.F, \lambda)}.}
\end{equation}
The fundamental group $\pi_{1}(\mathcal{D}(\mathfrak{s}) \sqcup F.\mathbf{t}.F,\lambda)$ is abelian being isomorphic to $\pi_{2}(\mathcal{P})$ and therefore isomorphic to its abelianization $H_{1}(\mathcal{D}(\mathfrak{s}) \sqcup F.\mathbf{t}.F,\lambda)$. The role of the isomorphism between the two groups will be played by the well known Hurewicz homomorphism $h: \pi_{1}(\mathcal{D}(\mathfrak{s}) \sqcup F.\mathbf{t}.F, \lambda) \rightarrow  H_{1}(\mathcal{D}(\mathfrak{s}) \sqcup F.\mathbf{t}.F,\lambda)$ which sends the homotopy class of a loop to the homology class of the corresponding 1-cycle. In our proofs in the following sections, we will identify the homotopy class of any loop $\xi$ with $h(\xi)$ without further comment.

\subsection{A characterisation of the asphericity in terms of the Pride complex} \label{chasphpc}

Assume now we are given a presentation $\mathcal{P}=(\mathbf{x},\mathbf{r})$ of a group $G$. The new presentation $\mathcal{P}^{e}=(\mathbf{x},\mathbf{r} \cup \mathbf{r}^{-1})$ where $\mathbf{r}^{-1}=\{r^{-1}| r \in \mathbf{r}\}$, is still giving $G$. The free crossed module arising from $\mathcal{P}^{e}$ is in fact our crossed module $(\mathcal{G}(\Upsilon), \hat{F}, \delta^{e})$. 

We have on the other hand the monoid presentation of $G$
$$\mathfrak{s}^{e}=\langle\mathbf{x}, \mathbf{x}^{-1}: \mathbf{s}^{e} \rangle$$ 
where 
$$\mathbf{s}^{e}=\{(r^{\varepsilon},\lambda): r \in \mathbf{r}, \varepsilon=\pm 1\} \cup \left\{ (x^{\varepsilon}x^{-\varepsilon},\lambda): x \in \mathbf{x}, \varepsilon=\pm 1\right\}.$$
Related to the new presentation $\mathfrak{s}^{e}$ we have the corresponding Pride complex $\mathcal{D}(\mathfrak{s}^{e})^{\ast}=\mathcal{D}(\mathfrak{s}^{e}) \sqcup F.\mathbf{t}.F$. The isomorphisms (\ref{ki}) now take the form:
\begin{equation} \label{kie}
	\xymatrix{\pi^{e}=\pi_{2}(\mathcal{P}^{e})  \ar@<2pt>[r]^-(0.40){\eta} &   \ar@<2pt>[l]^-(0.60){\psi}   {\pi_{1}(\mathcal{D}(\mathfrak{s}^{e}) \sqcup F.\mathbf{t}.F, \lambda)}.}
\end{equation}
Theorem \ref{wdom} and isomorphisms in (\ref{kie}) imply that the asphericity of $\cP$ is equivalent to $\pi_{1}(\mathcal{D}(\mathfrak{s}^{e}) \sqcup F.\mathbf{t}.F,\lambda)$ being generated as an abelian group by the homotopy classes of all 1-cycles corresponding to $\eta(\mu(\mathfrak{U}))$. This section is devoted to proving that the asphericity of $\cP$ is equivalent to $H_{1}(\mathcal{D}(\mathfrak{s}^{e}) \sqcup F.\mathbf{p}.F,\lambda)=0$ where $\mathbf{p}=\mathbf{q} \cup \mathbf{t}$ and $\mathbf{q}$ is the set of 1-cycles corresponding to $\eta(\mu \sigma ((1,r), (1,r^{-1})))$ with $r \in \mathbf{r}$. In geometrical terms this means that adding to $\mathcal{D}(\mathfrak{s}^{e})$ extra 2-cells which are translates of 2-cells arising from loops in $\mathbf{p}$, gives a new 2-complex $\mathcal{D}(\mathfrak{s}^{e}) \sqcup F.\mathbf{p}.F$ with $H_{1}(\mathcal{D}(\mathfrak{s}^{e}) \sqcup F.\mathbf{p}.F,\lambda)=0$. But, as it is explained in Corollary \ref{IR}, this implies the exactness of the sequence (\ref{cxkp}) that is associated with the 3-complex $(\mathcal{D}(\mathfrak{s}^{e}), \mathbf{p} )$ which is obtained from $\mathcal{D}(\mathfrak{s}^{e}) \sqcup F.\mathbf{p}.F$ by adding 3-cells $[e,\sigma]$ and $[\sigma, e]$ for every 2-cell $\sigma$ in $\mathcal{D}(\mathfrak{s}^{e}) \sqcup F.\mathbf{p}.F$.

For every two paths of positive length $A=e^{\varepsilon_{1}}_{1} \circ \dots \circ e^{\varepsilon_{n}}_{n}$ and $B=f^{\delta_{1}}_{1} \circ \dots \circ f^{\delta_{m}}_{m}$ in $\mathcal{D}(\mathfrak{s}^{e}) \sqcup F.\mathbf{t}.F$ we have two parallel paths:
$$A. \iota B \circ \tau A. B \text{ and } \iota A. B \circ A. \tau B.$$
In what follows we use the notation $C \sim D$ to mean that two parallel paths $C$ and $D$ are homotopic to each other. 
\begin{lemma} \label{[A,B]}
	For every two paths $A=e^{\varepsilon_{1}}_{1} \circ \dots \circ e^{\varepsilon_{n}}_{n}$ and $B=f^{\delta_{1}}_{1} \circ \dots \circ f^{\delta_{m}}_{m}$ as above, $(A. \iota B \circ \tau A. B) \sim (\iota A. B \circ A. \tau B)$.
\end{lemma}
\begin{proof}
	The proof is done by induction on the maximum of $n$ and $m$. If $m=n=1$, then it follows that
	$$A. \iota B \circ \tau A. B= e^{\varepsilon_{1}}_{1}.\iota f^{\delta_{1}}_{1} \circ \tau e^{\varepsilon_{1}}_{1}. f^{\delta_{1}}_{1} \sim \iota e^{\varepsilon_{1}}_{1}. f^{\delta_{1}}_{1} \circ e^{\varepsilon_{1}}_{1}. \tau f^{\delta_{1}}_{1}= \iota A. B \circ A. \tau B.$$
Indeed, if $\varepsilon_{1}=\delta_{1}=1$, then this is an immediate consequence of the 2-cell $[e_{1},f_{1}]$. If $\varepsilon_{1}=\delta_{1}=-1$, then, since
\begin{equation} \label{hsc}
e_{1}.\iota f_{1} \circ \tau e_{1}. f_{1} \sim \iota e_{1}. f_{1} \circ e_{1}. \tau f_{1},
\end{equation}
it follows by taking inverses that
\begin{align*}
\iota e_{1}^{\varepsilon_{1}}.f_{1}^{\delta_{1}} \circ e_{1}^{\varepsilon_{1}}. \tau f_{1}^{\delta_{1}}&= \tau e_{1}. f_{1}^{-1} \circ e_{1}^{-1}.\iota f_{1}\\
& \sim e_{1}^{-1}.\tau f_{1} \circ \iota e_{1}. f_{1}^{-1}\\
&= e^{\varepsilon_{1}}_{1}.\iota f^{\delta_{1}}_{1} \circ \tau e^{\varepsilon_{1}}_{1}. f^{\delta_{1}}_{1}.
\end{align*}
In the case when $\varepsilon_{1}=1$ and $\delta_{1}=-1$, after composing on the left of (\ref{hsc}) by $\iota e_{1}.f_{1}^{-1}$ we obtain
$$\iota e_{1}.f_{1}^{-1} \circ e_{1}.\iota f_{1} \circ \tau e_{1}. f_{1} \sim \iota e_{1}.f_{1}^{-1} \circ\iota e_{1}. f_{1} \circ e_{1}. \tau f_{1}= e_{1}. \tau f_{1},$$
and then after composing the above on the right by $\tau e_{1}. f_{1}^{-1}$ we get
$$\iota e_{1}.f_{1}^{-1} \circ e_{1}.\iota f_{1} \sim e_{1}. \tau f_{1} \circ \tau e_{1}. f_{1}^{-1},$$
which is the same as
$$\iota e_{1}^{\varepsilon_{1}}.f_{1}^{\delta_{1}} \circ e_{1}^{\varepsilon_{1}}.\tau f_{1}^{\delta_{1}} \sim e_{1}^{\varepsilon_{1}}. \iota f_{1}^{\delta_{1}} \circ \tau e_{1}^{\varepsilon_{1}}. f_{1}^{\delta_{1}}.$$
The proof for the case when $\varepsilon_{1}=-1$ and $\delta_{1}=1$ is symmetric to the above and is omitted. 

For the inductive step, let (for instance) $B$ be the path of maximal length $m>1$. For the path $B'=f_{1}^{\delta_{1}} \circ \dots \circ f_{m-1}^{\delta_{m-1}}$ we know by induction that 
	$$A. \iota B' \circ \tau A. B' \sim \iota A. B' \circ A. \tau B'.$$
	Again, by induction for $A$ and $f_{m}^{\delta_{m}}$ we have that
	$$A. \iota f_{m}^{\delta_{m}} \circ \tau A. f_{m}^{\delta_{m}} \sim \iota A. f_{m}^{\delta_{m}} \circ A. \tau f_{m}^{\delta_{m}}.$$
	It follows that
	\begin{align*}
		A. \iota B \circ \tau A. B&= A. \iota B \circ \tau A. B' \circ \tau A. f_{m}^{\delta_{m}}\\
		&= A. \iota B' \circ \tau A. B' \circ \tau A. f_{m}^{\delta_{m}}\\
		& \sim \iota A. B' \circ A. \tau B' \circ \tau A. f_{m}\\
		&= \iota A. B' \circ A.\iota f_{m}^{\delta_{m}} \circ \tau A. f_{m}^{\delta_{m}}\\
		& \sim \iota A. B' \circ  \iota A. f_{m}^{\delta_{m}} \circ A. \tau f_{m}^{\delta_{m}} \\
		&=  \iota A. B \circ A. \tau B.
	\end{align*}
	There is a similar proof when $A$ is of maximal length.
\end{proof}
\begin{remark} \label{abnd}
The proof of the base step of the lemma shows that in the presence of 2-cells $[e,f]$ where both $e,f$ are positive edges, the attachment of 2-cells $[e^{\varepsilon},f^{\delta}]$ with $\varepsilon =-1$ or $\delta=-1$ is unnecessary as long as we are interested in the homotopy of the complex.
\end{remark}

For every $u \in \hat{F}$, regarded as an element of the free monoid $F$ on $\mathbf{x} \cup \mathbf{x}^{-1}$, and for every $r \in \mathbf{r}$, we see that
\begin{align*}
\eta(\mu \sigma (u,r))={^{u}}\eta(\mu \sigma(1,r))&={^{u}}\langle(\lambda,r,\lambda,\lambda) \rangle\\
&=\langle (u,r,\lambda,u^{-1})^{\ast}\rangle\\
&=   \langle T_{uru^{-1}}\circ (u,r,\lambda,u^{-1}) \circ T^{-1}_{uu^{-1}}\rangle .
\end{align*}
The path $T_{uru^{-1}}\circ (u,r,\lambda,u^{-1}) \circ T_{uu^{-1}}$ is a composition of $T_{uru^{-1}}$ which is a trivial path from the freely reduced word $(uru^{-1})^{\ast}$ to $uru^{-1}$ followed by the edge $(u,r,\lambda,u^{-1})$ and then by the inverse of the trivial path $T_{uu^{-1}}$ from $uu^{-1}$ to $\lambda$. 
Similarly to the above we have that
\begin{align*}
\eta(\mu \sigma((u,r^{-1})))={^{u}}\eta(\mu \sigma (1, r^{-1}))&={^{u}}\langle(\lambda,r^{-1},\lambda,\lambda) \rangle\\
&=\langle (u,r^{-1},\lambda,u^{-1})^{\ast}\rangle\\
&=\langle T_{ur^{-1}u^{-1}}  \circ (u,r^{-1},\lambda,u^{-1})\circ  T^{-1}_{uu^{-1}}\rangle.
\end{align*}
Then we have 
\begin{align*}
	\eta(\mu \sigma((u,r),(u,r^{-1})))&= \eta(\mu \sigma (u,r)) + \eta(\mu \sigma ((u,r^{-1})))\\
	&=\langle ((T_{uru^{-1}}\circ (u,r,\lambda,u^{-1}) \circ T^{-1}_{uu^{-1}})+(T_{ur^{-1}u^{-1}}  \circ (u,r^{-1},\lambda,u^{-1})\circ  T^{-1}_{uu^{-1}}))^{\ast} \rangle\\
	&=\langle T_{(uru^{-1})^{\ast}(ur^{-1}u^{-1})^{\ast}} \circ T_{uru^{-1}}\cdot (ur^{-1}u^{-1})^{\ast} \circ (u,r,\lambda,u^{-1})\cdot (ur^{-1}u^{-1})^{\ast} \\
	 &\circ   T^{-1}_{uu^{-1}} \cdot (ur^{-1}u^{-1})^{\ast} \circ T_{ur^{-1}u^{-1}} \circ (u,r^{-1}, \lambda, u^{-1}) \circ  T^{-1}_{uu^{-1}} \rangle.
\end{align*}
Now we define two closed paths in $\mathcal{D}(\mathfrak{s}^{e}) \sqcup F.\mathbf{t}.F$. First we let 
\begin{align*}
	P(r,u)= T_{(uru^{-1})(ur^{-1}u^{-1})} & \circ (u, r, \lambda, u^{-1}ur^{-1}u^{-1})  \\
	&\circ(T^{-1}_{uu^{-1}}\cdot ur^{-1}u^{-1} ) \circ (u, r^{-1},\lambda, u^{-1}) \circ T^{-1}_{uu^{-1}},
\end{align*}
and second
$$Q(r,u)= T_{urr^{-1}u^{-1}} \circ (u, r,\lambda, r^{-1}u^{-1}) \circ (u, r^{-1}, \lambda, u^{-1}) \circ T^{-1}_{uu^{-1}}.$$

\begin{proposition} \label{th}
The presentation $\cP$ is aspherical if and only if $\pi_{1}(\mathcal{D}(\mathfrak{s}^{e}) \sqcup F.\mathbf{t}.F,\lambda)$ is generated as an abelian group by the homotopy classes of loops $Q(r,u)$ with $r \in \mathbf{r}$ and $u \in F$.
\end{proposition}
\begin{proof}
As we mentioned earlier, the presentation $\cP$ is aspherical if and only if the set of all $\eta(\mu \sigma((u,r),(u,r^{-1})))$ generates $\pi_{1}(\mathcal{D}(\mathfrak{s}^{e}) \sqcup F.\mathbf{t}.F,\lambda)$. Given this, the proposition follows directly if we prove that for each $r \in \mathbf{r}$ and $u \in F$, $\eta(\mu \sigma((u,r),(u,r^{-1})))=\langle P(r,u) \rangle $ and that $P(r,u) \sim Q(r,u)$.

Let us prove first that $\eta(\mu \sigma((u,r),(u,r^{-1})))=\langle P(r,u) \rangle $. Consider the following four paths in $\mathcal{D}(\mathfrak{s}^{e}) \sqcup F.\mathbf{t}.F$	
\begin{align*}
		a&=(uru^{-1})^{\ast}\cdot T_{ur^{-1}u^{-1}},\\
		b&=T_{uru^{-1}} \cdot (ur^{-1}u^{-1}),\\
		d&=T_{uru^{-1}} \cdot (ur^{-1}u^{-1})^{\ast},\\
		c&=(uru^{-1})\cdot T_{ur^{-1}u^{-1}},
	\end{align*}
and observe from Lemma \ref{[A,B]} that 
\begin{equation} \label{abcd}
	a \circ b \sim d \circ c
\end{equation}
Second, consider
	\begin{align*}
		e&=(u,r,\lambda,u^{-1}(ur^{-1}u^{-1})^{\ast}),\\
		c&=(uru^{-1})\cdot T_{ur^{-1}u^{-1}},\\
		g&=(u,r,\lambda,u^{-1}(ur^{-1}u^{-1})),\\
		f&=(uu^{-1})\cdot T_{ur^{-1}u^{-1}},
	\end{align*}
where again from Lemma \ref{[A,B]} we have that $c \circ g \sim e \circ f$. This implies that
\begin{equation} \label{c-1}
	c^{-1} \sim g \circ f^{-1} \circ e^{-1}.
\end{equation}
And finally, let
	\begin{align*}
		y&= T^{-1}_{uu^{-1}}\cdot (ur^{-1}u^{-1})^{\ast} \\
		f&= (uu^{-1})\cdot T_{ur^{-1}u^{-1}}\\
		x&=T^{-1}_{uu^{-1}}\cdot (ur^{-1}u^{-1})\\
		z&=\lambda \cdot T_{ur^{-1}u^{-1}},
	\end{align*}
where as before $f \circ x \sim y \circ z$. This on the other hand implies that
\begin{equation} \label{f-1}
f^{-1} \sim x \circ z^{-1} \circ y^{-1}.
\end{equation}
Further we write $\ell= T_{(uru^{-1})^{\ast}(ur^{-1}u^{-1})^{\ast}}$, $\ell_{1}= T_{(uru^{-1})(ur^{-1}u^{-1})}$, $k=(u,r^{-1},\lambda,u^{-1})$ and $h=T^{-1}_{uu^{-1}}$. With the above abbreviations we have
	\begin{align*}
		\ell &\sim \ell_{1} \circ b^{-1} \circ  a^{-1}&& \text{($\parallel$ trivial parallel paths are $\sim$)}\\
		& \sim \ell_{1} \circ  c^{-1} \circ d^{-1} && \text{(from (\ref{abcd}))}\\
		& \sim \ell_{1} \circ (g \circ f^{-1} \circ e^{-1}) \circ d^{-1} && \text{(from (\ref{c-1}))}\\
		&\sim \ell_{1} \circ g \circ (x \circ z^{-1} \circ y^{-1}) \circ e^{-1} \circ d^{-1} && \text{(from  (\ref{f-1})).}
	\end{align*} 
	It follows that
	\begin{align*}
		\eta(\mu \sigma((u,r),(u,r^{-1})))&=\langle \ell \circ (d \circ e \circ y \circ z \circ  k \circ h) \rangle \\
		&=\langle (\ell_{1} \circ g \circ (x \circ z^{-1} \circ y^{-1}) \circ e^{-1} \circ d^{-1}) \circ (d \circ e \circ y \circ z \circ k \circ h)\rangle \\
		& = \langle  \ell_{1} \circ g \circ x \circ k \circ h \rangle \\
		&= \langle P(r,u) \rangle.
	\end{align*}
	Secondly, we prove that $P(r,u) \sim Q(r,u) $. Indeed, if we consider paths
	\begin{align*}
		A&=(u, r, \lambda, u^{-1}ur^{-1}u^{-1})\\
		B&= (ur) \cdot  T^{-1}_{u^{-1}u} \cdot (r^{-1}u^{-1})\\
		C&= u \cdot  T^{-1}_{u^{-1}u} \cdot  (r^{-1}u^{-1})\\
		D&= (u, r, \lambda, r^{-1}u^{-1}),
	\end{align*}
which from Lemma \ref{[A,B]} satisfy $A \circ C \sim B \circ D$,	then we have
	\begin{align*}
		P(r,u) & = T_{(uru^{-1})(ur^{-1}u^{-1})}  \circ (u, r, \lambda, u^{-1}ur^{-1}u^{-1})  \circ(T^{-1}_{uu^{-1}}\cdot ur^{-1}u^{-1} ) \circ (u, r^{-1},\lambda, u^{-1}) \circ T^{-1}_{uu^{-1}}\\
		&\sim T_{(uru^{-1})(ur^{-1}u^{-1})} \circ A \circ (u\cdot T^{-1}_{u^{-1}u}\cdot r^{-1}u^{-1}) \circ (u,r^{-1},\lambda,u^{-1})\circ T^{-1}_{uu^{-1}}\\
		&= T_{(uru^{-1})(ur^{-1}u^{-1})} \circ A \circ C \circ (u,r^{-1},\lambda,u^{-1})\circ T^{-1}_{uu^{-1}}\\
		&\sim T_{(uru^{-1})(ur^{-1}u^{-1})}  \circ B \circ D \circ (u,r^{-1},\lambda,u^{-1})\circ T^{-1}_{uu^{-1}}\\
		&\sim T_{urr^{-1}u^{-1}} \circ B^{-1} \circ B \circ D \circ (u,r^{-1},\lambda,u^{-1})\circ T^{-1}_{uu^{-1}}\\
		&= T_{urr^{-1}u^{-1}} \circ  D \circ (u,r^{-1},\lambda,u^{-1})\circ T^{-1}_{uu^{-1}}\\
		&= Q(r,u).
		\end{align*}
This concludes the proof.
\end{proof}
Passing to homology we have the following
\begin{proposition} \label{thc}
	The presentation $\cP$ is aspherical if and only if $H_{1}(\mathcal{D}(\mathfrak{s}^{e}) \sqcup F.\mathbf{t}.F,\lambda)$ is generated as an abelian group by the homology classes of 1-cycles corresponding to loops $Q(r,u)$ with $r \in \mathbf{r}$ and $u \in F$.
\end{proposition}

\begin{definition}
	In $\mathcal{D}(\mathfrak{s}^{e}) \sqcup F.\mathbf{t}.F$ we let $\mathbf{q}$ be the set of closed paths $q(r, \lambda)$ with $r \in \mathbf{r}$.
\end{definition}
If we attach to $\mathcal{D}(\mathfrak{s}^{e}) \sqcup F.\mathbf{t}.F$ 2-cells $[u, q(r, \lambda), v]$ along the closed paths $u. q(r, \lambda). v$ with $u,v \in F$, then we obtain a 2-complex $\mathcal{D}(\mathfrak{s}^{e}) \sqcup F.(\mathbf{q} \cup \mathbf{t}).F$. The asphericity of $\cP$ is encoded in the first homology group of the 2-skeleton of $\mathcal{D}(\mathfrak{s}^{e}) \sqcup F.(\mathbf{q} \cup \mathbf{t}).F$ as the following shows.
\begin{theorem} \label{asphh}
The presentation $\cP$ is aspherical if and only if $H_{1}(\mathcal{D}(\mathfrak{s}^{e}) \sqcup F.(\mathbf{q} \cup \mathbf{t}).F)=0$.
\end{theorem}
 
To prove the theorem we first note the following two lemmas.
\begin{lemma} \label{Jm}
	For every $\varsigma \in Z_{1}(\mathcal{D}(\mathfrak{s}^{e}) \sqcup F.\mathbf{t}.F)$ and every $u,v \in F$ such that $\bar{u}=\bar{v}$, $\varsigma \cdot u +B_{1}(\mathcal{D}(\mathfrak{s}^{e}) \sqcup F.\mathbf{t}.F)= \varsigma \cdot v +B_{1}(\mathcal{D}(\mathfrak{s}^{e}) \sqcup F.\mathbf{t}.F)$.
\end{lemma}
\begin{proof}
It is enough to prove that for every positive edge $f$ of $\mathcal{D}(\mathfrak{s}^{e})$, we have $\varsigma \cdot \iota f +B_{1}(\mathcal{D}(\mathfrak{s}^{e}) \sqcup F.\mathbf{t}.F)= \varsigma \cdot \tau f +B_{1}(\mathcal{D}(\mathfrak{s}^{e}) \sqcup F.\mathbf{t}.F)$. From Lemma 4.1 of \cite{LDHM1} it follows that $\varsigma\cdot (\iota f - \tau f) \in B_{1}(\mathcal{D}(\mathfrak{s}^{e}))$. But $B_{1}(\mathcal{D}(\mathfrak{s}^{e})) \subseteq B_{1}(\mathcal{D}(\mathfrak{s}^{e}) \sqcup F.\mathbf{t}.F)$ and we are done.
\end{proof}

If $u=x_{1}^{\varepsilon_{1}}\dots x_{n}^{\varepsilon_{n}} \in F$ is any word of length $|u|=n \in \mathbb{N}$, then a trivial path from $\lambda$ to $uu^{-1}$ is the following 
$$T_{uu^{-1}}=(\lambda,x_{1}^{\varepsilon_{1}}x_{1}^{-\varepsilon_{1}},\lambda,\lambda)^{-1} \circ \dots \circ  (x_{1}^{\varepsilon_{1}}\dots x_{|u|-1}^{\varepsilon_{|u|-1}}, x_{|u|}^{\varepsilon_{|u|}}x_{|u|}^{-\varepsilon_{|u|}},\lambda, x_{|u|-1}^{-\varepsilon_{|u|-1}}\dots x_{1}^{-\varepsilon_{1}})^{-1}.$$
We write for short
\begin{align*}
t_{u}^{(1)}&=(\lambda,x_{1}^{\varepsilon_{1}}x_{1}^{-\varepsilon_{1}},\lambda,\lambda)\\
...\\
t_{u}^{(|u|)}&=(x_{1}^{\varepsilon_{1}}\dots x_{|u|-1}^{\varepsilon_{|u|-1}}, x_{|u|}^{\varepsilon_{|u|}}x_{|u|}^{-\varepsilon_{|u|}},\lambda, x_{|u|-1}^{-\varepsilon_{|u|-1}}\dots x_{1}^{-\varepsilon_{1}}),
\end{align*}
and let 
$$\tau_{uu^{-1}}=t_{u}^{(1)}+\dots + t_{u}^{(|u|)}.$$ 
\begin{definition}
For every $r \in \mathbf{r}$ and $u \in F$, we let 
$$q(r,u)=(u,r,\lambda,r^{-1}u^{-1})+(u,r^{-1},\lambda,u^{-1})+\tau_{uu^{-1}}-\tau_{urr^{-1}u^{-1}},$$
be the 1-cycle that corresponds to the closed path $Q(r,u)$. When $u=\lambda$, we let
$$q(r,\lambda)=(\lambda,r,\lambda,r^{-1})+(\lambda,r^{-1},\lambda,\lambda)-\tau_{rr^{-1}}$$
be the 1-cycle corresponding to $q(r, \lambda)$.
\end{definition}
\begin{lemma} \label{qr}
For every $r \in \mathbf{r}$ and $u \in F$, $u. q(r,\lambda).u^{-1}+B_{1}(\mathcal{D}(\mathfrak{s}^{e}) \sqcup F.\mathbf{t}.F)=q(r,u)+B_{1}(\mathcal{D}(\mathfrak{s}^{e}) \sqcup F.\mathbf{t}.F)$.
\end{lemma}
\begin{proof}
First note that $T^{-1}_{uu^{-1}} \circ T_{urr^{-1}u^{-1}} \sim u. T_{rr^{-1}}.u^{-1}$ since any two trivial paths with the same end points are homotopic with each other. For the corresponding 1-chains we have that $$\tau_{uu^{-1}}-\tau_{urr^{-1}u^{-1}}+B_{1}(\mathcal{D}(\mathfrak{s}^{e}) \sqcup F.\mathbf{t}.F)=-u.\tau_{rr^{-1}}.u^{-1}+B_{1}(\mathcal{D}(\mathfrak{s}^{e}) \sqcup F.\mathbf{t}.F).$$
It follows now that
\begin{align*}
u. q(r,\lambda).u^{-1}+B_{1}(\mathcal{D}(\mathfrak{s}^{e}) \sqcup F.\mathbf{t}.F)&= (u,r,\lambda,r^{-1}u^{-1})+ (u,r^{-1},\lambda,u^{-1})- u. \tau_{rr^{-1}}.u^{-1}+\\
&+B_{1}(\mathcal{D}(\mathfrak{s}^{e}) \sqcup F.\mathbf{t}.F)\\
&=(u,r,\lambda,r^{-1}u^{-1})+ (u,r^{-1},\lambda,u^{-1})+\tau_{uu^{-1}}-\tau_{urr^{-1}u^{-1}}+\\
&+B_{1}(\mathcal{D}(\mathfrak{s}^{e}) \sqcup F.\mathbf{t}.F)\\
&=q(r,u)+B_{1}(\mathcal{D}(\mathfrak{s}^{e}) \sqcup F.\mathbf{t}.F),
\end{align*}
proving the claim.
\end{proof}

\begin{proof} (of Theorem \ref{asphh})
If $H_{1}(\mathcal{D}(\mathfrak{s}^{e})\sqcup F.(\mathbf{q} \cup \mathbf{t}).F)=0$, then $H_{1}(\mathcal{D}(\mathfrak{s}^{e})\sqcup F.(\mathbf{q} \cup \mathbf{t}).F,\lambda)=0$ which means that the homology classes of the loops $u.q(r, \lambda).v$ which sits inside $(\mathcal{D}(\mathfrak{s}^{e})\sqcup F.\mathbf{t}.F,\lambda)$ trivialize $H_{1}(\mathcal{D}(\mathfrak{s}^{e})\sqcup F.\mathbf{t}.F,\lambda)$. We claim that every such 1-cycle corresponding to $u.q(r, \lambda).v$ is in fact homologous to the 1-cycle corresponding to the loop $Q(r,u)$. Indeed, since $u. q(r, \lambda). v$ is a loop in $(\mathcal{D}(\mathfrak{s}^{e})\sqcup F.\mathbf{t}.F,\lambda)$, then $\bar{v}=\bar{u}^{-1}$. It follows from Lemma \ref{Jm} and Lemma \ref{qr} that
\begin{align*}
u. q(r,\lambda).v + B_{1}(\mathcal{D}(\mathfrak{s}^{e})\sqcup F.\mathbf{t}.F)&= u. q(r,\lambda).u^{-1} + B_{1}(\mathcal{D}(\mathfrak{s}^{e})\sqcup F.\mathbf{t}.F)\\
&=q(r,u)+B_{1}(\mathcal{D}(\mathfrak{s}^{e})\sqcup F.\mathbf{t}.F),
\end{align*}
which proves our claim. As a consequence of this we have that the homology classes of 1-cycles $q(r,u)$ trivialize $H_{1}(\mathcal{D}(\mathfrak{s}^{e})\sqcup F.\mathbf{t}.F,\lambda)$, and then from Proposition \ref{thc} we get the asphericity of $\cP$.
	
Conversely, if $\cP$ is aspherical, then from Proposition \ref{thc} and Lemma \ref{qr} $H_{1}(\mathcal{D}(\mathfrak{s}^{e})\sqcup F.\mathbf{t}.F,\lambda)$ is generated as an abelian group by the homology classes of 1-cycles $u. q(r, \lambda).u^{-1}$. But from \cite{LDHM2} it follows that the homology group $H_{1}(\mathcal{D}(\mathfrak{s}^{e})\sqcup F.\mathbf{t}.F, w)$ of the connected component $(\mathcal{D}(\mathfrak{s}^{e})\sqcup F.\mathbf{t}.F, w)$ is isomorphic to $H_{1}(\mathcal{D}(\mathfrak{s}^{e})\sqcup F.\mathbf{t}.F, \lambda)$, where the isomorphism $\phi_{w}:H_{1}(\mathcal{D}(\mathfrak{s}^{e})\sqcup F.\mathbf{t}.F, \lambda) \rightarrow H_{1}(\mathcal{D}(\mathfrak{s}^{e})\sqcup F.\mathbf{t}.F, w)$ maps each homology class of some 1-cycle $\varsigma$ to the homology class of $\varsigma \cdot w$. This shows that the set of the homology classes of 1-cycles $u.q(r, \lambda).u^{-1}w$ trivialize $H_{1}(\mathcal{D}(\mathfrak{s}^{e})\sqcup F.\mathbf{t}.F)$. We prove that this set equals the set of homology classes of 1-cycles $u.q(r, \lambda).v$ where $u,v \in F$. Indeed, for every $u,v \in F$ and every $q(r, \lambda)$, if we take $w=uv$, we get that $u.q(r, \lambda).u^{-1}uv+B_{1}(\mathcal{D}(\mathcal{D}(\mathfrak{s}^{e})\sqcup F.\mathbf{t}.F)$ is a generator of $H_{1}(\mathcal{D}(\mathfrak{s}^{e})\sqcup F.\mathbf{t}.F)$. But from Lemma \ref{Jm}, $u.q(r, \lambda).u^{-1}uv+B_{1}(\mathcal{D}(\mathfrak{s}^{e})\sqcup F.\mathbf{t}.F)=u.q(r, \lambda).v+B_{1}(\mathcal{D}(\mathfrak{s}^{e})\sqcup F.\mathbf{t}.F)$, hence $u.q(r, \lambda).v+B_{1}(\mathcal{D}(\mathfrak{s}^{e})\sqcup F.\mathbf{t}.F)$ is a generator of $H_{1}(\mathcal{D}(\mathfrak{s}^{e})\sqcup F.\mathbf{t}.F)$. For the converse, it is obvious that any generator $u.q(r, \lambda).u^{-1}w + B_{1}(\mathcal{D}(\mathfrak{s}^{e})\sqcup F.\mathbf{t}.F)$ is of the form $u.q(r, \lambda).v + B_{1}(\mathcal{D}(\mathfrak{s}^{e})\sqcup F.\mathbf{t}.F)$ with $v=u^{-1}w$.
\end{proof}
\begin{corollary}  \label{IR}
If $\cP=(\mathbf{x}, \mathbf{r})$ is an aspherical presentation of a group $G$, then the sequence $\mathcal{K}^{\mathbf{q} \cup \mathbf{t}}$ associated with the 3-complex $(\mathcal{D}(\mathfrak{s}^{e}), \mathbf{q} \cup \mathbf{t})$ arising from the monoid presentation $\mathfrak{s}^{e}=\langle\mathbf{x}, \mathbf{x}^{-1}: \mathbf{s}^{e} \rangle$ of $G$, is exact.
\end{corollary}
\begin{proof}
This is an immediate consequence of Theorem \ref{asphh} and of Proposition 14 of \cite{SES}.
\end{proof}

\subsection{A preliminary result} \label{pr}

Let $\mathcal{P}=(\mathbf{x},\mathbf{r})$ be an aspherical group presentation and $\mathcal{P}_{1}=(\mathbf{x},\mathbf{r}_{1})$ a subpresentation of the first where $\mathbf{r}_{1}=\mathbf{r}\setminus \{r_{0}\}$ and $r_{0}\in \mathbf{r}$ is a fixed relation. As before we let $\mathcal{P}^{e}=(\mathbf{x},\mathbf{r} \cup \mathbf{r}^{-1})$ and $\mathcal{P}_{1}^{e}=(\mathbf{x},\mathbf{r}_{1} \cup \mathbf{r}_{1}^{-1})$ be their extended presentations.

We denote by $\Upsilon_{1}$, $\mathfrak{U}_{1}$ monoids associated with $\mathcal{P}^{e}_{1}$, by $(\mathcal{G}(\Upsilon_{1}), \hat{F}, \delta_{1}^{e})$ the free crossed module of $\cP^{e}_{1}$ and let $\sigma_{1}:FM(\hat{F} \times (\mathbf{r}_{1} \cup \mathbf{r}_{1}^{-1})) \rightarrow \Upsilon_{1}$ be the corresponding canonical morphism and $\mu_{1}: \Upsilon_{1} \rightarrow \mathcal{G}(\Upsilon_{1})$ the morphism of the universal property of $\mathcal{G}(\Upsilon_{1})$. Further we let $\hat{\mathfrak{U}}_{1}$ be the subgroup of $\mathcal{G}(\Upsilon_{1})$ generated by the images of all $((u,r^{\varepsilon}), (u, r^{-\varepsilon}))$ where $u \in \hat{F}$ and $r \in \mathbf{r}_{1}$. The kernel of $\delta_{1}^{e}$ is denoted by $\pi^{e}_{1}$. Also we consider $\hat{\mathfrak{A}}_{1}$ the subgroup of $\hat{\mathfrak{U}}$ generated by all $\mu\sigma ((u,r^{\varepsilon}), (u, r^{-\varepsilon}))$ where where $u \in \hat{F}$ and $r \in \mathbf{r}_{1}$. We note that the freeness of 
$(\mathcal{G}(\Upsilon_{1}), \hat{F}, \delta_{1}^{e})$ on the map $v_{1}: \mathbf{r}_{1} \cup \mathbf{r}_{1}^{-1} \rightarrow \mathcal{G}(\Upsilon_{1})$ defined by $v_{1}(r^{\varepsilon})=\mu_{1}\sigma_{1}((1,r^{\varepsilon}))$ implies that relative to the map $v'_{1}: \mathbf{r}_{1} \cup \mathbf{r}_{1}^{-1} \rightarrow \mathcal{G}(\Upsilon)$ defined by $v'_{1}(r^{\varepsilon})=\mu\sigma((1,r^{\varepsilon}))$, there is a unique homomorphism $\hat{\varphi}: \mathcal{G}(\Upsilon_{1}) \rightarrow \mathcal{G}(\Upsilon)$ of crossed modules which renders the following diagram commutative
\begin{equation*}
	\xymatrix{ \mathcal{G}(\Upsilon_{1}) \ar[rr]^{\hat{\varphi}} \ar[rd]_{\delta^{e}_{1}} &&
		\mathcal{G}(\Upsilon) \ar[ld]^-{\delta^{e}} \\ & \hat{F}}
\end{equation*}
The following will be useful in the proof of our main theorem.
\begin{proposition} \label{quasi}
	If $\mathcal{P}$ is aspherical, then $\hat{\varphi}(\pi^{e}_{1})=\hat{\mathfrak{A}}_{1}$.
\end{proposition}
\begin{proof}
Let $\tilde{d} =\prod^{n}_{i=1} (\mu_{1}\sigma_{1}((\xi_{i}, r_{i}^{\gamma_{i}})))^{\ell_{i}}\in \pi_{1}^{e}$ with $\ell_{i}=\pm 1$. It follows from Theorem \ref{wdom} that
\begin{equation} \label{no}
\hat{\varphi}(\tilde{d})=\prod^{s}_{j=1}(\mu \sigma ((u_{j}, \rho^{\varepsilon_{j}}_{j}), (u_{j}, \rho^{-\varepsilon_{j}}_{j})))^{z_{j}} \cdot \prod^{t}_{k=1}(\mu \sigma ((v_{i}, r_{0}^{\eta_{i}}), (v_{i}, r_{0}^{-\eta_{i}})))^{\delta_{k}}
\end{equation}
where each $\rho_{j} \in \mathbf{r}_{1}$ and $z_{j}, \delta_{k} = \pm 1$. Since $\hat{\varphi}(\tilde{d})=\prod^{n}_{i=1} (\mu \sigma((\xi_{i}, r_{i}^{\gamma_{i}})))^{\ell_{i}}$, it follows from the above that in $\mathcal{G}(\Upsilon)$ we have 
	\begin{align*}
\prod^{n}_{i=1} (\mu \sigma((\xi_{i}, r_{i}^{\gamma_{i}})))^{\ell_{i}} \cdot \prod^{s}_{j=1}(\mu \sigma ((u_{j}, \rho^{\varepsilon_{j}}_{j}), (u_{j}, \rho^{-\varepsilon_{j}}_{j})))^{-z_{j}}=\prod^{t}_{k=1}(\mu \sigma ((v_{i}, r_{0}^{\eta_{i}}), (v_{i}, r_{0}^{-\eta_{i}})))^{\delta_{k}}.
	\end{align*}
We can apply $\hat{g}$ defined in Proposition \ref{free} on both sides and get
	\begin{align*}
		(\ell_{1}\xi_{1}^{\alpha}\cdot r_{1}^{\beta}+\cdot \cdot \cdot + \ell_{n}\xi_{n}^{\alpha}\cdot r_{n}^{\beta})-2(z_{1} u_{1}^{\alpha} \cdot \rho_{1}^{\beta} + \cdot \cdot \cdot +z_{s}u_{s}^{\alpha} \cdot \rho_{s}^{\beta})		=2(\delta_{1}v_{1}^{\alpha}+\cdot \cdot \cdot + \delta_{t}v_{t}^{\alpha})\cdot r_{0}^{\beta}.
	\end{align*}
	The freeness of $\mathcal{N}(\mathcal{P})$ on the set of elements $r^{\beta}$ implies in particular that
	\begin{equation*}
		\delta_{1}v_{1}^{\alpha}+\cdot \cdot \cdot + \delta_{t}v_{t}^{\alpha}=0,
	\end{equation*}
which in turn implies that $t$ is even and that the terms in the above expression fall into pairs $(\delta_{i} v^{\alpha}_{i}, \delta_{j}v^{\alpha}_{j})$ such that $v^{\alpha}_{i}=v^{\alpha}_{j}$ and $\delta_{j}=-\delta_{i}$. The easily verified fact that in $\mathcal{G}(\Upsilon)$, $\mu\sigma((u,r),(u,r^{-1}))=\mu\sigma((u,r^{-1}),(u, r))$ and the fact that if $v^{\alpha}=w^{\alpha}$, then for every $s \in \mathbf{r}$, $\mu\sigma((v,s^{\eta}), (v, s^{-\eta}))=\mu\sigma((w, s^{\eta}), ((w, s^{-\eta}))$, imply easily that for each such pair of terms as above we have
	\begin{equation*}
		\mu\sigma((v_{i}, r_{0}^{\eta_{i}}), ((v_{i},r_{0}^{-\eta_{i}}))=\mu\sigma((v_{j}, r_{0}^{\eta_{j}}), (v_{j}, r_{0}^{-\eta_{j}})).
	\end{equation*}
If we apply this and the fact that $\delta_{j}=-\delta_{i}$ to right had side of (\ref{no}) we derive that this product equals to 1, which shows that $\hat{\varphi}(\tilde{d}) \in \hat{\mathfrak{A}}_{1}$.  
\end{proof}

\subsection{The proof}

Throughout this section we assume that $\cP=(\mathbf{x}, \mathbf{r})$ is an aspherical presentation of the trivial group. Consider now a sub presentation $\cP_{1}=(\mathbf{x}, \mathbf{r}_{1})$ of $\cP$ where $\mathbf{r}_{1}= \mathbf{r} \setminus \{r_{0}\}$. We consider their respective extensions $\cP^{e}=(\mathbf{x}, \mathbf{r} \cup \mathbf{r}^{-1})$ and $\cP^{e}_{1}=(\mathbf{x}, \mathbf{r}_{1} \cup \mathbf{r}_{1}^{-1})$. For each of the above extended presentations, we have a monoid presentation of the same respective group, namely
$$\mathfrak{s}^{e}=\langle\mathbf{x}, \mathbf{x}^{-1}: \mathbf{s}^{e} \rangle$$ 
is a monoid presentation of the trivial group, where 
$$\mathbf{s}^{e}=\{(r^{\varepsilon},\lambda): r \in \mathbf{r}, \varepsilon=\pm 1\} \cup \left\{ (x^{\varepsilon}x^{-\varepsilon},\lambda): x \in \mathbf{x}, \varepsilon=\pm 1\right\},$$
and
$$\mathfrak{s}^{e}_{1}=\langle\mathbf{x}, \mathbf{x}^{-1}: \mathbf{s}^{e}_{1} \rangle$$ 
is a monoid presentation of the group given by $\cP^{e}_{1}$, where 
$$\mathbf{s}^{e}_{1}=\{(r_{1}^{\varepsilon},\lambda): r_{1} \in \mathbf{r}_{1}, \varepsilon=\pm 1\} \cup \left\{ (x^{\varepsilon}x^{-\varepsilon},\lambda): x \in \mathbf{x}, \varepsilon=\pm 1\right\}.$$
Related to $\mathfrak{s}^{e}$ we have defined two 2-complexes. The first one is the usual Squier complex $\mathcal{D}(\mathfrak{s}^{e})$, and the second one is its extension $\mathcal{D}(\mathfrak{s}^{e}) \sqcup F.\mathbf{t}.F$, and similarly we have two 2-complexes arising from $\mathfrak{s}^{e}_{1}$, $\mathcal{D}(\mathfrak{s}^{e}_{1})$ and its extension $\mathcal{D}(\mathfrak{s}^{e}_{1}) \sqcup F.\mathbf{t}.F$. Further, $\mathcal{D}(\mathfrak{s}^{e})\sqcup F.\mathbf{t}.F$ has been extended to a 3-complex $(\mathcal{D}(\mathfrak{s}^{e}), \mathbf{q} \cup \mathbf{t})$ by adding first to $\mathcal{D}(\mathfrak{s}^{e}) \sqcup F.\mathbf{t}.F$ 2-cells arising from $q(r, \lambda)$  and their translates to obtain in return the 2-complex $\mathcal{D}(\mathfrak{s}^{e}) \sqcup F.(\mathbf{t} \cup \mathbf{q}).F$, and than adding all the 3-cells $[e, \sigma]$ or $[\sigma, e]$ for every 2-cell $\sigma$ from $\mathcal{D}(\mathfrak{s}^{e}) \sqcup F.(\mathbf{t} \cup \mathbf{q}).F$ and every positive edge $e$. We write for short $(\mathcal{D}(\mathfrak{s}^{e}), \mathbf{q} \cup \mathbf{t})$ by $(\mathcal{D}(\mathfrak{s}^{e}), \mathbf{p})$ where $\mathbf{p}=\mathbf{q} \cup \mathbf{t}$. Likewise, $\mathcal{D}(\mathfrak{s}^{e}_{1})\sqcup F.\mathbf{t}.F$ extends to a 3-complex $(\mathcal{D}(\mathfrak{s}^{e}_{1}), \mathbf{p}_{1})$ where $\mathbf{p}_{1}=\mathbf{q}_{1} \cup \mathbf{t}$ and $\mathbf{q}_{1}$ is the set of 2-cells arising from $Q(r_{1},1)$ with $r_{1} \in \mathbf{r}_{1}$. But $(\mathcal{D}(\mathfrak{s}^{e}_{1}), \mathbf{p}_{1})$ is a subcomplex of $(\mathcal{D}(\mathfrak{s}^{e}), \mathbf{p})$, therefore we have the following exact sequence of abelian groups
$$\scalebox{0.9}{\xymatrix{ \dots \ar[r] & H_{2}(\mathcal{D}(\mathfrak{s}^{e}), \mathbf{p} ) \ar[r] & H_{2}((\mathcal{D}(\mathfrak{s}^{e}), \mathbf{p}), (\mathcal{D}(\mathfrak{s}^{e}_{1}), \mathbf{p}_{1})) \ar[r] & H_{1}(\mathcal{D}(\mathfrak{s}^{e}_{1}), \mathbf{p}_{1}) \ar[r] & H_{1}(\mathcal{D}(\mathfrak{s}^{e}), \mathbf{p}  ) \ar[r] & \dots
}}.$$
We know from Theorem \ref{asphh} that $H_{1}(\mathcal{D}(\mathfrak{s}^{e}), \mathbf{p}  )=0$, so if we prove that $H_{2}((\mathcal{D}(\mathfrak{s}^{e}), \mathbf{p}), (\mathcal{D}(\mathfrak{s}^{e}_{1}), \mathbf{p}_{1}))=0$, then the exactness of the above sequence will imply that $H_{1}(\mathcal{D}(\mathfrak{s}^{e}_{1}), \mathbf{p}_{1})=0$ and we are done. Before we proceed with the proof, we explain how the boundary maps for the corresponding quotient complex are defined. For this we consider the commutative diagram
\begin{equation} \label{def rel}
\scalebox{0.9}{\xymatrix{C_{3}(\mathcal{D}(\mathfrak{s}^{e}), \mathbf{p}) \ar[d]_{\mu_{3}} \ar[r]^{\tilde{\partial}_{3}} & C_{2}(\mathcal{D}(\mathfrak{s}^{e}), \mathbf{p}) \ar[d]_{\mu_{2}} \ar[r]^{\tilde{\partial}_{2}} & C_{1}(\mathcal{D}(\mathfrak{s}^{e}), \mathbf{p}) \ar[d]_{\mu_{1}}\\
	C_{3}(\mathcal{D}(\mathfrak{s}^{e}), \mathbf{p})/C_{3} (\mathcal{D}(\mathfrak{s}^{e}_{1}), \mathbf{p}_{1}) \ar[r]^{\hat{\partial}_{3}} & C_{2}(\mathcal{D}(\mathfrak{s}^{e}), \mathbf{p})/C_{2} (\mathcal{D}(\mathfrak{s}^{e}_{1}), \mathbf{p}_{1})  \ar[r]^{\hat{\partial}_{2}} & C_{1}(\mathcal{D}(\mathfrak{s}^{e}), \mathbf{p})/C_{1} (\mathcal{D}(\mathfrak{s}^{e}_{1}), \mathbf{p}_{1})}}
\end{equation}
where $\mu_{i}$ for $i=1,2,3$ are the canonical epimorphisms. Then, for $i=2,3$ and for every generating cell $\sigma \in C_{i}(\mathcal{D}(\mathfrak{s}^{e}), \mathbf{p})$ we have
$$\hat{\partial}_{i}(\mu_{i}\sigma)= \mu_{i-1}\tilde{\partial}_{i}(\sigma).$$
We write $Im (\hat{\partial}_{3})=B_{2}((\mathcal{D}(\mathfrak{s}^{e}), \mathbf{p}),(\mathcal{D}(\mathfrak{s}_{1}^{e}), \mathbf{p}_{1}))$, and similarly $Im (\hat{\partial}_{2})=B_{1}((\mathcal{D}(\mathfrak{s}^{e}), \mathbf{p}),(\mathcal{D}(\mathfrak{s}_{1}^{e}), \mathbf{p}_{1}))$. Also we let $Z_{2}((\mathcal{D}(\mathfrak{s}^{e}), \mathbf{p}),(\mathcal{D}(\mathfrak{s}_{1}^{e}), \mathbf{p}_{1}))=Ker(\hat{\partial}_{2})$ and then $$H_{2}((\mathcal{D}(\mathfrak{s}^{e}), \mathbf{p}),(\mathcal{D}(\mathfrak{s}_{1}^{e}), \mathbf{p}_{1}))=Z_{2}((\mathcal{D}(\mathfrak{s}^{e}), \mathbf{p}),(\mathcal{D}(\mathfrak{s}_{1}^{e}), \mathbf{p}_{1}))/B_{2}((\mathcal{D}(\mathfrak{s}^{e}), \mathbf{p}),(\mathcal{D}(\mathfrak{s}_{1}^{e}), \mathbf{p}_{1})).$$ 
We note that
$$C_{2}(\mathcal{D}(\mathfrak{s}^{e}), \mathbf{p})/C_{2} (\mathcal{D}(\mathfrak{s}_{1}^{e}), \mathbf{p}_{1}) \cong \mu_{2}(C_{2}(\mathcal{D}(\mathfrak{s}^{e}))) \oplus C_{2}^{\mathbf{q}_{0}},$$
where $\mu_{2}(C_{2}(\mathcal{D}(\mathfrak{s}^{e})))$ is the free abelian group generated by all 2-cells $[e,f]$ where at least one of the edges $e$ or $f$ arises from $r_{0}$, and $C_{2}^{\mathbf{q}_{0}}$ is the free abelian group on 2-cells $[u, q_{0},v]$ where $q_{0}=Q(r_{0},\lambda)$. The 2-cell $[\lambda, q_{0}, \lambda]$ corresponding to $q_{0}$ will be throughout denoted by $\widehat{q}_{0}$ and then $u.\widehat{q_{0}}.v=[u,q_{0},v]$. This representation of $C_{2}(\mathcal{D}(\mathfrak{s}^{e}), \mathbf{p})/C_{2} (\mathcal{D}(\mathfrak{s}_{1}^{e}), \mathbf{p}_{1})$ allows us to identify each $u.\widehat{q}_{0}.v$ with $\mu_{2}(u.\widehat{q}_{0}.v)$ and as a result we can write $\hat{\partial}_{2}(u.\widehat{q}_{0}.v)=\mu_{1}\tilde{\partial}_{2}(u.\widehat{q}_{0}.v)$. Also the above representation allows us to regard $Z_{2}((\mathcal{D}(\mathfrak{s}^{e}), \mathbf{p}),(\mathcal{D}(\mathfrak{s}_{1}^{e}), \mathbf{p}_{1}))$ as a subgroup of $\mu_{2}(C_{2}(\mathcal{D}(\mathfrak{s}^{e}))) \oplus C_{2}^{\mathbf{q}_{0}}$. Now we let
$$\varphi_{rel}: \mu_{2}(C_{2}(\mathcal{D}(\mathfrak{s}^{e}))) \oplus C_{2}^{\mathbf{q}_{0}} \rightarrow \mathbb{Z}G. \widehat{q}_{0}. \mathbb{Z}G$$
be the $(\mathbb{Z}F,\mathbb{Z}F)$-homomorphism defined by mapping $\mu_{2}(C_{2}(\mathcal{D}(\mathfrak{s}^{e})))$ to 0, and every 2-cell $[u, \widehat{q}_{0}, v]$ to $\bar{u}.\widehat{q}_{0}. \bar{v}$. Denote the kernel of $\varphi_{rel}$ by $K_{rel}^{\mathbf{q}_{0}}$. By the same argument as that of \cite{SES} we see that
$$K_{rel}^{\mathbf{q}_{0}}= \mu_{2}(C_{2}(\mathcal{D}(\mathfrak{s}^{e})))+ J. \widehat{q}_{0}.\mathbb{Z}F+\mathbb{Z}F.\widehat{q}_{0}.J.$$
Later we will make use of the fact that $K_{rel}^{\mathbf{q}_{0}}$ can be regarded as a sub group of $K^{\mathbf{p}}$.

Next we show that $B_{2}((\mathcal{D}(\mathfrak{s}^{e}), \mathbf{p}), (\mathcal{D}(\mathfrak{s}_{1}^{e}), \mathbf{p}_{1})) \subseteq K_{rel}^{\mathbf{q}_{0}}$ and that the restriction of $\hat{\partial}_{2}$ on $K_{rel}^{\mathbf{q}_{0}}$ sends $K_{rel}^{\mathbf{q}_{0}}$ onto the subgroup $B_{1}(\mathcal{D}(\mathfrak{s}^{e}), \mathcal{D}(\mathfrak{s}_{1}^{e}))$ of $ C_{1}(\mathcal{D}(\mathfrak{s}^{e}), \mathbf{p})/ C_{1}(\mathcal{D}(\mathfrak{s}_{1}^{e}), \mathbf{p}_{1})$ defined by
$$
B_{1}(\mathcal{D}(\mathfrak{s}^{e}), \mathcal{D}(\mathfrak{s}_{1}^{e}))=\{\mu_{1} (\tilde{\partial}_{2}(\beta))| \beta \in C_{2}(\mathcal{D}(\mathfrak{s}^{e}))\}.
$$
To see that $B_{2}((\mathcal{D}(\mathfrak{s}^{e}), \mathbf{p}), (\mathcal{D}(\mathfrak{s}_{1}^{e}), \mathbf{p}_{1})) \subseteq K_{rel}^{\mathbf{q}_{0}}$, we must prove that for every 3-cell $[f,\sigma]$ or $[\sigma, f]$, $$\hat{\partial}_{3}\left([f, \sigma]+C_{3}(\mathcal{D}(\mathfrak{s}_{1}^{e}), \mathbf{p}_{1})\right) \in K_{rel}^{\mathbf{q}_{0}}$$
and similarly,
$$\hat{\partial}_{3}\left([\sigma,f]+C_{3}(\mathcal{D}(\mathfrak{s}_{1}^{e}), \mathbf{p}_{1})\right) \in K_{rel}^{\mathbf{q}_{0}}.$$
We prove the second for convenience. Let $[\sigma, f] \notin C_{3}(\mathcal{D}(\mathfrak{s}_{1}^{e}), \mathbf{p}_{1})$. If $\sigma \in F.\widehat{q}_{0}.F$ or $f$ arises from $r_{0}$, then clearly
\begin{align*}
	\hat{\partial}_{3}\left([\sigma,f]+C_{3}(\mathcal{D}(\mathfrak{s}_{1}^{e}), \mathbf{p}_{1})\right)&=\left(\sigma. (\tau f- \iota f)+\sum_{i} \varepsilon_{i}[e_{i},f]\right)+C_{2}(\mathcal{D}(\mathfrak{s}_{1}^{e}), \mathbf{p}_{1})\in K_{rel}^{\mathbf{q}_{0}}.
\end{align*}
Otherwise, if $\sigma \notin F.\widehat{q}_{0}.F$ and $f$ arises from $\mathbf{r}_{1}$, then $\sigma=[g,h]$ where at least $g$ or $h$ arises from $r_{0}$. Again we see that $\hat{\partial}_{3}\left([\sigma,f]+C_{3}(\mathcal{D}(\mathfrak{s}_{1}^{e}), \mathbf{p}_{1})\right) \in K_{rel}^{\mathbf{q}_{0}}$.

Next we prove that the restriction of $\hat{\partial}_{2}$ on $K_{rel}^{\mathbf{q}_{0}}$ sends $K_{rel}^{\mathbf{q}_{0}}$ onto $B_{1}(\mathcal{D}(\mathfrak{s}^{e}), \mathcal{D}(\mathfrak{s}_{1}^{e}))$. Indeed, since for every $(\iota f  - \tau f). \sigma_{0} \in J. \widehat{q}_{0}.\mathbb{Z}F$ 
\begin{equation*}
	(\iota f  - \tau f). \sigma_{0}= -\tilde{\partial}_{3}[f, \sigma_{0}]-\sum_{i}\varepsilon_{i}[f, e_{i}],
\end{equation*}
then we can derive that
\begin{align*}
	\hat{\partial}_{2}\left((\iota f  - \tau f). \sigma_{0}\right)&=\mu_{1}\tilde{\partial}_{2}(-\tilde{\partial}_{3}[f, \sigma_{0}]-\sum_{i}\varepsilon_{i}[f, e_{i}])\\
	&=-\mu_{1}\tilde{\partial}_{2}\left(\sum_{i}\varepsilon_{i}[f, e_{i}]\right)\in B_{1}(\mathcal{D}(\mathfrak{s}^{e}), \mathcal{D}(\mathfrak{s}_{1}^{e})).
\end{align*}
In a symmetric way one can show that $\hat{\partial}_{2}\left(\sigma_{0}.(\iota f  - \tau f)\right) \in B_{1}(\mathcal{D}(\mathfrak{s}^{e}), \mathcal{D}(\mathfrak{s}_{1}^{e}))$. Finally, if $[e,f] \in \mu_{2}(C_{2}(\mathcal{D}(\mathfrak{s}^{e})))$, then
$$\hat{\partial}_{2}([e,f])=\mu_{1}\tilde{\partial_{2}}([e,f])\in B_{1}(\mathcal{D}(\mathfrak{s}^{e}), \mathcal{D}(\mathfrak{s}_{1}^{e})).$$
This also shows that the restriction of $\hat{\partial}_{2}$ on $K_{rel}^{\mathbf{q}_{0}}$ is onto.

Therefore we have the complex
\begin{equation} \label{c-rel}
\xymatrix{0 \ar[r] & B_{2}((\mathcal{D}(\mathfrak{s}^{e}), \mathbf{p}), (\mathcal{D}(\mathfrak{s}_{1}^{e}), \mathbf{p}_{1})) \ar[r]^-{incl.} & K_{rel}^{\mathbf{q}_{0}} \ar[r]^-{\hat{\partial}_{2}} & B_{1}(\mathcal{D}(\mathfrak{s}^{e}), \mathcal{D}(\mathfrak{s}_{1}^{e})) \ar[r] & 0.
}
\end{equation}
which is exact on the left and on the right. 
\begin{lemma} \label{c-rel-ext}
The complex (\ref{c-rel}) is exact.
\end{lemma}
\begin{proof}
For this we consider the commutative diagram
$$\xymatrix@C=35pt{0 \ar[r] & B_{2}(\mathcal{D}(\mathfrak{s}^{e}), \mathbf{p}) \ar[d]_{\mu_{2}} \ar[r]^{incl.} & K^{\mathbf{p}} \ar[d]_{\mu_{2}} \ar[r]^{\tilde{\partial}_{2}} & B_{1}(\mathcal{D}(\mathfrak{s}^{e})) \ar[d]_{\mu_{1}} \ar[r] & 0 \\
	0 \ar[r] & B_{2}((\mathcal{D}(\mathfrak{s}^{e}), \mathbf{p}), (\mathcal{D}(\mathfrak{s}_{1}^{e}), \mathbf{p}_{1})) \ar[r]^-{incl.} & K_{rel}^{\mathbf{q}_{0} \ar[r]^-{\hat{\partial}_{2}}} & B_{1}(\mathcal{D}(\mathfrak{s}^{e}), \mathcal{D}(\mathfrak{s}_{1}^{e})) \ar[r] & 0}$$
The top row is exact from Corollary \ref{IR}, and $\mu_{1},\mu_{2}$ are the restrictions of the  epimorphisms of (\ref{def rel}). Let $\xi =\sum_{i}z_{i}\sigma_{i}\in Ker \hat{\partial}_{2}$. We recall that $\xi$ can be regarded as an element of $K^{\mathbf{p}}$ with no terms arising from $\mathbf{p}_{1}$ and has no square 2-cells $[e,f]$ with both $e$ and $f$ in $\mathcal{D}(\mathfrak{s}^{e}_{1})$, therefore $\xi =\sum_{i}z_{i}\mu_{2}(\sigma_{i})$.  Further we have that
\begin{align*}
	0&=\hat{\partial}_{2}\left(\sum_{i}z_{i}\mu_{2}(\sigma_{i})\right)=\sum_{i}z_{i}\hat{\partial}_{2} \mu_{2}(\sigma_{i})\\
	&=\sum_{i}z_{i}\mu_{1}\tilde{\partial}_{2}(\sigma_{i})=\mu_{1}\tilde{\partial}_{2}\left(\sum_{i}z_{i}\sigma_{i}\right),
\end{align*}
which implies that 
\begin{equation} \label{in1}
\tilde{\partial}_{2}\left(\sum_{i}z_{i}\sigma_{i}\right) \in C_{1}(\mathcal{D}(\mathfrak{s}^{e}_{1})).
\end{equation}
If $\sigma_{i}$ is a square 2-cell which is represented in $\xi$, then from (\ref{part2s}) and from the convention in the end on section \ref{sqc} that each each rewriting rule $(r^{\varepsilon}, \lambda)$ or $(x^{\varepsilon}x^{-\varepsilon}, \lambda)$ is identified with the respective edge $(\lambda, r^{\varepsilon}, \lambda, \lambda)$ or $(\lambda, x^{\varepsilon}x^{-\varepsilon}, \lambda, \lambda)$ defined in section \ref{pc}, we see that
$$\tilde{\partial}_{2}(\sigma_{i}) \in J. \mathbf{s}^{e}. \mathbb{Z}F+ \mathbb{Z}F. \mathbf{s}^{e}. J$$
But 
$$\mathbf{s}^{e}=\mathbf{s}^{e}_{1} \cup \mathbf{s}_{0}$$
where 
$$\mathbf{s}_{0}=\{(r_{0}^{\varepsilon}, \lambda)| \varepsilon=\pm 1\},$$
hence 
$$\tilde{\partial}_{2}(\sigma_{i}) \in (J. \mathbf{s}_{0}. \mathbb{Z}F+ \mathbb{Z}F. \mathbf{s}_{0}. J)+(J. \mathbf{s}^{e}_1. \mathbb{Z}F+ \mathbb{Z}F. \mathbf{s}^{e}_{1}. J).$$
If now $\sigma_{i}$ represented in $\xi$ is either $j.\widehat{q}_{0}. v$ or $u. \widehat{q}_{0}. j$ with $u, v \in F$ and $j \in J$, then it follows from (\ref{part2p}) that 
$$\tilde{\partial}_{2}(\sigma_{i}) \in (J. \mathbf{s}_{0}. \mathbb{Z}F+ \mathbb{Z}F. \mathbf{s}_{0}. J)+(J. \mathbf{s}^{e}_1. \mathbb{Z}F+ \mathbb{Z}F. \mathbf{s}^{e}_{1}. J).$$
From (\ref{in1}) we have that
\begin{align*}
\tilde{\partial}_{2}(\xi) & \in ((J. \mathbf{s}_{0}. \mathbb{Z}F+ \mathbb{Z}F. \mathbf{s}_{0}. J)+(J. \mathbf{s}^{e}_1. \mathbb{Z}F+ \mathbb{Z}F. \mathbf{s}^{e}_{1}. J)) \cap C_{1}(\mathcal{D}(\mathfrak{s}_{1}))\\
&=((J. \mathbf{s}_{0}. \mathbb{Z}F+ \mathbb{Z}F. \mathbf{s}_{0}. J) \cap C_{1}(\mathcal{D}(\mathfrak{s}_{1})))+ (J. \mathbf{s}^{e}_1. \mathbb{Z}F+ \mathbb{Z}F. \mathbf{s}^{e}_{1}. J) \tag{Dedekind identity}\\
&= J. \mathbf{s}^{e}_1. \mathbb{Z}F+ \mathbb{Z}F. \mathbf{s}^{e}_{1}. J \tag{since $(J. \mathbf{s}_{0}. \mathbb{Z}F+ \mathbb{Z}F. \mathbf{s}_{0}. J) \cap C_{1}(\mathcal{D}(\mathfrak{s}_{1}))=0$.}
\end{align*}
Taking into account that $\tilde{\partial}_{2}\left(\sum_{i}z_{i}\sigma_{i}\right)$ is a 1-cycle in $Z_{1}(\mathcal{D}(\mathfrak{s}^{e}_{1}))$, it follows from the above that 
\begin{equation*} 
\tilde{\partial}_{2}\left(\sum_{i}z_{i}\sigma_{i}\right) \in Ker \tilde{\partial}_{1} \cap (J. \mathbf{s}^{e}_{1}. \mathbb{Z}F+ \mathbb{Z}F. \mathbf{s}^{e}_{1}. J).
\end{equation*}
Theorem 6.6 of \cite{OK2002} implies that there is a 2-chain $\sum_{j}k_{j} \beta_{j} \in C_{2}(\mathcal{D}(\mathfrak{s}^{e}_{1}))$ such that $\tilde{\partial}_{2}\left(\sum_{i}z_{i}\sigma_{i}\right)=\tilde{\partial}_{2}\left(\sum_{j}k_{j} \beta_{j}\right)$ and then we have the 2-cycle $\tilde{\xi}=\sum_{i}z_{i}\sigma_{i}-\sum_{j}k_{j} \beta_{j}$ in $K^{\mathbf{p}}$. It follows that $\tilde{\xi}$ is a 2-boundary since the top row of our diagram is exact, and has the property that
\begin{align*}
	\mu_{2}(\tilde{\xi})&=\mu_{2}\left(\sum_{i}z_{i}\sigma_{i}\right)-\mu_{2}\left(\sum_{j}k_{j} \beta_{j}\right)\\
	&=\mu_{2}\left(\sum_{i}z_{i}\sigma_{i}\right) && \text{(since each $\beta_{j} \in C_{2}(\mathcal{D}(\mathfrak{s}^{e}_{1}))$)}\\
	&=\sum_{i}z_{i}\mu_{2}(\sigma_{i})\\
	&=\xi,
\end{align*}
hence for some $w \in C_{3}(\mathcal{D}(\mathfrak{s}^{e}), \mathbf{p})$ we have that
$$\xi=\mu_{2}(\tilde{\xi})=\mu_{2}(\tilde{\partial}_{3}(w))=\hat{\partial}_{3}\mu_{3}(w).$$
This proves that $\xi$ is a relative 2-boundary and as a consequence the exactness of the bottom row.
\end{proof}

Further we note that $B_{1}(\mathcal{D}(\mathfrak{s}^{e}), \mathcal{D}(\mathfrak{s}_{1}^{e}))$ embeds in $Im (\hat{\partial}_{2})$. Indeed, any element $\mu_{1}\tilde{\partial}_{2}(\beta)$ of $B_{1}(\mathcal{D}(\mathfrak{s}^{e}), \mathcal{D}(\mathfrak{s}_{1}^{e}))$ where $\beta$ is a 2-chain from $C_{2}(\mathcal{D}(\mathfrak{s}^{e}))$ is in $Im (\hat{\partial}_{2})$ since $C_{2}(\mathcal{D}(\mathfrak{s}^{e})) \leq C_{2}(\mathcal{D}(\mathfrak{s}^{e}), \mathbf{p})$ and $C_{1}(\mathcal{D}(\mathfrak{s}_{1}^{e}), \mathbf{p}_{1})=C_{1}(\mathcal{D}(\mathfrak{s}^{e}_{1}))$.

Finally, consider the commutative diagram
$$\scalebox{0.9}{\xymatrix{0 \ar[r] & B_{2}((\mathcal{D}(\mathfrak{s}^{e}), \mathbf{p}), (\mathcal{D}(\mathfrak{s}_{1}^{e}), \mathbf{p}_{1})) \ar[d]_{\iota} \ar[r]^-{incl.} & K_{rel}^{\mathbf{q}_{0}} \ar[d]_{\iota} \ar[r]^-{\hat{\partial}_{2}} & B_{1}(\mathcal{D}(\mathfrak{s}^{e}), \mathcal{D}(\mathfrak{s}_{1}^{e})) \ar[d]_{\iota} \ar[r] & 0 \\
	0 \ar[r] & 	Z_{2}((\mathcal{D}(\mathfrak{s}^{e}), \mathbf{p}), (\mathcal{D}(\mathfrak{s}_{1}^{e}), \mathbf{p}_{1})) \ar[r]^-{\iota} & \mu_{2}(C_{2}(\mathcal{D}(\mathfrak{s}^{e}))) \oplus C_{2}^{\mathbf{q}_{0}} \ar[r]^-{\hat{\partial}_{2}} & Im (\hat{\partial}_{2})  \ar[r] & 0 }}$$
where the top row is exact from Lemma \ref{c-rel-ext}, and the bottom one is also exact where $Im(\hat{\partial}_{2}) \leq C_{1}(\mathcal{D}(\mathfrak{s}^{e}), \mathbf{p})/C_{1}(\mathcal{D}(\mathfrak{s}_{1}^{e}), \mathbf{p}_{1})$. From the Snake Lemma we get the exact sequence
\begin{equation} \label{mainses}
	\scalebox{0.9}{\xymatrix{0 \ar[r] & H_{2}((\mathcal{D}(\mathfrak{s}^{e}), \mathbf{p}), (\mathcal{D}(\mathfrak{s}_{1}^{e}), \mathbf{p}_{1})) \ar[r] & \mathbb{Z}G.\widehat{q}_{0}.  \mathbb{Z}G \ar[r]^-{\nu} & Im(\hat{\partial}_{2})/  B_{1}(\mathcal{D}(\mathfrak{s}^{e}), \mathcal{D}(\mathfrak{s}_{1}^{e})) \ar[r] & 0,}}
\end{equation}
where $\nu(\widehat{q}_{0})=\gamma(\hat{\partial}_{2}(\widehat{q}_{0})$ and $\gamma: Im (\hat{\partial}_{2}) \rightarrow Im (\hat{\partial}_{2})/B_{1}(\mathcal{D}(\mathfrak{s}^{e}), \mathcal{D}(\mathfrak{s}_{1}^{e}))$ is the canonical epimorphism. Since $G$ is the trivial group, we have that for every $u. \widehat{q}_{0}. v$,
\begin{equation} \label{fgen}
	\gamma(\hat{\partial}_{2}(u. \widehat{q}_{0}. v))=\gamma(\hat{\partial}_{2}(\widehat{q}_{0})).
\end{equation}
This follows easily if we prove that for every positive edge $e$ and $v \in F$ we have that
$$\gamma (\hat{\partial}_{2}(\iota e. \widehat{q}_{0} .v))=\gamma(\hat{\partial}_{2}(\tau e. \widehat{q}_{0}. v)),$$
and similarly, for every positive edge $f$ and $u \in F$,
$$\gamma(\hat{\partial}_{2}(u. \widehat{q}_{0}. \iota f))=\gamma(\hat{\partial}_{2}(u. \widehat{q}_{0}. \tau f)).$$
We prove the first claim for convenience. Since 
$$(\iota e- \tau e). \widehat{q}_{0}.v= -\tilde{\partial}_{3}[e,\widehat{q}_{0}.v]-\sum_{i}\varepsilon_{i}[e,e_{i}],$$
where $\partial [\lambda, q_{0}, v]= e_{1}^{\varepsilon_{1}}\circ \dots \circ e_{n}^{\varepsilon_{n}}$, then
$$(\mu_{1}\circ \tilde{\partial}_{2})((\iota e- \tau e). \widehat{q}_{0}.v)=-(\mu_{1}\circ \tilde{\partial}_{2})\left(\sum_{i}\varepsilon_{i}[e,e_{i}])\right).$$
But
$$(\mu_{1}\circ \tilde{\partial}_{2})\left(\sum_{i}\varepsilon_{i}[e,e_{i}]\right) \in B_{1}(\mathcal{D}(\mathfrak{s}^{e}), \mathcal{D}(\mathfrak{s}_{1}^{e})),$$
consequently
\begin{align*}
	\gamma(\hat{\partial}_{2}(\iota e. \widehat{q}_{0}. v)-\hat{\partial}_{2}(\tau e. \widehat{q}_{0}. v))&= 	\gamma (\hat{\partial}_{2}((\iota e-\tau e). \widehat{q}_{0}. v))\\
&=\gamma ((\mu_{1} \circ \tilde{\partial}_{2})((\iota e- \tau e). \widehat{q}_{0}.v)) \\
	&=-\gamma \left((\mu_{1}\circ \tilde{\partial}_{2})\left(\sum_{i}\varepsilon_{i}[e,e_{i}]\right)\right)\\
	&=0,
\end{align*}
which proves the first claim. 

An obvious consequence of (\ref{fgen}) is that $Im(\hat{\partial}_{2})/  B_{1}(\mathcal{D}(\mathfrak{s}^{e}), \mathcal{D}(\mathfrak{s}_{1}^{e}))$ is a cyclic group with generator $\gamma (\hat{\partial}_{2}(\widehat{q}_{0}))$. The key to proving our main theorem is that $Im(\hat{\partial}_{2})/  B_{1}(\mathcal{D}(\mathfrak{s}^{e}), \mathcal{D}(\mathfrak{s}_{1}^{e}))$ is infinite cyclic. Before that, we need to do some preparatory work.

If we let $G_{1}$ be the group given by $\cP_{1}$, then for every $g_{i} \in G_{1}$ we let $(\mathcal{D}(\mathfrak{s}_{1}^{e})\sqcup F. \mathbf{t}.F, u_{i})$ be the connected component of the 2-complex $\mathcal{D}(\mathfrak{s}_{1}^{e})\sqcup F. \mathbf{t}.F$ corresponding to $g_{i}$ where $u_{i}$ is a fixed vertex in this component and so $\bar{u}_{i}=g_{i}$. Let $H_{1}(\mathcal{D}(\mathfrak{s}_{1}^{e})\sqcup F. \mathbf{t}.F, u_{i})$ be the corresponding homology group. There is no loss of generality if we assume that each 1-cycle $\varsigma_{u_{i}} \in Z_{1}(\mathcal{D}(\mathfrak{s}_{1}^{e})\sqcup F. \mathbf{t}.F, u_{i})$ arises from a loop $\tilde{\varsigma}_{u_{i}}$ based at the vertex $u_{i}$. It follows from this that $\varsigma_{u_{i}}.u^{-1}_{i}\in Z_{1}(\mathcal{D}(\mathfrak{s}_{1}^{e})\sqcup F. \mathbf{t}.F, \lambda)$ corresponds to the loop $\tilde{\varsigma}_{g_{i}}.u^{-1}_{i}$ based at $u_{i}u^{-1}_{i}$. Letting $\tilde{\tau}_{i}$ be a trivial path from $\lambda$ to $u_{i}u^{-1}_{i}$ and $\tau_{i}$ the corresponding 1-chain, we see that $\varsigma_{u_{i}}.u^{-1}_{i}=\varsigma_{u_{i}}.u^{-1}_{i}+\tau_{i}-\tau_{i}$ is the 1-cycle corresponding to the loop $\tilde{\tau}_{i} \circ \tilde{\varsigma}_{u_{i}}.u^{-1}_{i} \circ \tilde{\tau}^{-1}_{i}$ based at $\lambda$.
It follows from this and from \cite{LDHM2} that $H_{1}(\mathcal{D}(\mathfrak{s}_{1}^{e})\sqcup F. \mathbf{t}.F, u_{i})$ is isomorphic to $H_{1}(\mathcal{D}(\mathfrak{s}_{1}^{e})\sqcup F. \mathbf{t}.F, \lambda)$ where the isomorphism 
$$\tilde{\theta}_{u_{i}}: H_{1}(\mathcal{D}(\mathfrak{s}_{1}^{e})\sqcup F. \mathbf{t}.F, u_{i}) \rightarrow H_{1}(\mathcal{D}(\mathfrak{s}_{1}^{e})\sqcup F. \mathbf{t}.F, \lambda)$$
is given by
\begin{equation} \label{tthu}
\varsigma_{u_{i}}+B_{1}(\mathcal{D}(\mathfrak{s}_{1}^{e})\sqcup F. \mathbf{t}.F, u_{i}) \mapsto \varsigma_{u_{i}}\cdot u_{i}^{-1}+B_{1}(\mathcal{D}(\mathfrak{s}_{1}^{e})\sqcup F. \mathbf{t}.F,\lambda).
\end{equation}
To see this we recall from \cite{LDHM2} that the mapping
$$\theta_{u_{i}}: \pi_{1}(\mathcal{D}(\mathfrak{s}_{1}^{e})\sqcup F. \mathbf{t}.F, u_{i}) \rightarrow \pi_{1}(\mathcal{D}(\mathfrak{s}_{1}^{e})\sqcup F. \mathbf{t}.F, \lambda) $$
defined by
$$\langle \tilde{\varsigma}_{u_{i}} \rangle \mapsto \langle  (\tilde{\varsigma}_{u_{i}}.u^{-1}_{i})^{\ast}\rangle$$
is an isomorphism, and since $ \pi_{1}(\mathcal{D}(\mathfrak{s}_{1}^{e})\sqcup F. \mathbf{t}.F, \lambda)$ is abelian as explained in the end of section \ref{pc}, then its isomorphic $\pi_{1}(\mathcal{D}(\mathfrak{s}_{1}^{e})\sqcup F. \mathbf{t}.F, u_{i})$ is also abelian which implies that the mapping
$$h_{i}: \pi_{1}(\mathcal{D}(\mathfrak{s}_{1}^{e})\sqcup F. \mathbf{t}.F, u_{i}) \rightarrow H_{1}(\mathcal{D}(\mathfrak{s}_{1}^{e})\sqcup F. \mathbf{t}.F, u_{i})$$
which sends the homotopy class of the loop $\tilde{\varsigma}_{u_{i}}$ to the homology class of the corresponding 1-cycle $\varsigma_{u_{i}}$, is an isomorphism as well. Writing by $\tilde{\theta}_{u_{i}}$ the composite of the isomorphisms
$$
\xymatrix{H_{1}(\mathcal{D}(\mathfrak{s}_{1}^{e})\sqcup F. \mathbf{t}.F, u_{i}) \ar[r]^{h^{-1}_{i}} & \pi_{1}(\mathcal{D}(\mathfrak{s}_{1}^{e})\sqcup F. \mathbf{t}.F, u_{i}) \ar[r]^{\theta_{u_{i}}} & \pi_{1}(\mathcal{D}(\mathfrak{s}_{1}^{e})\sqcup F. \mathbf{t}.F, \lambda) \ar[r]^{h_{\lambda}} & H_{1}(\mathcal{D}(\mathfrak{s}_{1}^{e})\sqcup F. \mathbf{t}.F, \lambda)
}
$$
we see that
\begin{align*}
\tilde{\theta}_{u_{i}}(\varsigma_{u_{i}})&= h_{\lambda}(\theta_{u_{i}} (h^{-1}_{i}( \varsigma_{u_{i}}+ B_{1}(\mathcal{D}(\mathfrak{s}_{1}^{e})\sqcup F. \mathbf{t}.F,u_{i})))\\
&= h_{\lambda}(\theta_{u_{i}}(\langle \tilde{\varsigma}_{u_{i}}\rangle ))\\
&= h_{\lambda}(\langle  (\tilde{\varsigma}_{u_{i}}.u^{-1}_{i})^{\ast}\rangle)\\
&= h_{\lambda}(\langle \tilde{\tau}_{i} \circ \tilde{\varsigma}_{u_{i}}.u^{-1}_{i} \circ \tilde{\tau}^{-1}_{i}\rangle) \tag{definition of $P^{\ast}$ in \cite{LDHM2}}\\
&=\tau_{i} + \varsigma_{u_{i}}.u^{-1}_{i}-\tau_{i} + B_{1}(\mathcal{D}(\mathfrak{s}_{1}^{e})\sqcup F. \mathbf{t}.F,\lambda)\\
&= \varsigma_{u_{i}}.u^{-1}_{i} + B_{1}(\mathcal{D}(\mathfrak{s}_{1}^{e})\sqcup F. \mathbf{t}.F,\lambda),
\end{align*}
therefore we have (\ref{tthu}).

The homology group $H_{1}(\mathcal{D}(\mathfrak{s}_{1}^{e})\sqcup F. \mathbf{t}.F)$ decomposes as a direct sum
$$H_{1}(\mathcal{D}(\mathfrak{s}_{1}^{e})\sqcup F. \mathbf{t}.F)=\oplus_{g_{i} \in G_{1}}H_{1}(\mathcal{D}(\mathfrak{s}_{1}^{e})\sqcup F. \mathbf{t}.F, u_{i}),$$
and any 1-cycle $\varsigma$ decomposes uniquely as
$$\varsigma= \varsigma_{u_{1}}+\dots + \varsigma_{u_{n}}$$
where $\varsigma_{u_{i}} \in Z_{1}(\mathcal{D}(\mathfrak{s}_{1}^{e})\sqcup F. \mathbf{t}.F, u_{i})$ is a 1-cycle arising from a loop $\tilde{\varsigma}_{u_{i}}$ based at $u_{i}$ and then $\varsigma+B_{1}(\mathcal{D}(\mathfrak{s}_{1}^{e})\sqcup F. \mathbf{t}.F)$ writes uniquely as
$$\varsigma+B_{1}(\mathcal{D}(\mathfrak{s}_{1}^{e})\sqcup F. \mathbf{t}.F)=(\varsigma_{u_{1}}+B_{1}(\mathcal{D}(\mathfrak{s}_{1}^{e})\sqcup F. \mathbf{t}.F, u_{1}))+ \dots + (\varsigma_{u_{n}}+B_{1}(\mathcal{D}(\mathfrak{s}_{1}^{e})\sqcup F. \mathbf{t}.F, u_{n})).$$

We let $\psi_{1}: H_{1}(\mathcal{D}(\mathfrak{s}_{1}^{e})\sqcup F. \mathbf{t}.F,1) \rightarrow \pi^{e}_{1}$ and $\eta: \pi^{e} \rightarrow H_{1}(\mathcal{D}(\mathfrak{s}^{e})\sqcup F. \mathbf{t}.F,1)$ be the isomorphism of \cite{LDHM2}. With these notations the following holds true.
\begin{lemma} \label{inv}
For every $\varsigma\in Z_{1}(\mathcal{D}(\mathfrak{s}_{1}^{e})\sqcup F. \mathbf{t}.F, \lambda)$, $\hat{\varphi}\psi_{1}(\varsigma+ B_{1}(\mathcal{D}(\mathfrak{s}_{1}^{e})\sqcup F. \mathbf{t}.F, \lambda))=\psi(\varsigma+ B_{1}(\mathcal{D}(\mathfrak{s}^{e})\sqcup F. \mathbf{t}.F))$.
\end{lemma}
\begin{proof}
Since the maps $\psi$ and $\psi_{1}$ are defined on homotopy classes of paths ending at $\lambda$, we have to make the following remark first. We may regard the 1-cycle $\varsigma$ as a finite sum of 1-cycles arising from loops in $(\mathcal{D}(\mathfrak{s}_{1}^{e})\sqcup F. \mathbf{t}.F, \lambda)$ which implies in turn that $\varsigma$ is a 1-cycle corresponding to a loop in $(\mathcal{D}(\mathfrak{s}_{1}^{e})\sqcup F. \mathbf{t}.F, \lambda)$ based at $\lambda$. Assume now that
	$$\varsigma=\sum_{i \in I}z_{i}(u_{i},s_{i},\lambda,v_{i}),$$ and let 
	$$J=\{i \in I: s_{i}=r^{\varepsilon_{i}}_{i} \text{ where } r^{\varepsilon_{i}}_{i} \in \mathbf{r}_{1}^{\pm 1}\}.$$ 
If $J=\emptyset$, then the equality of the lemma is straightforward. If $J\neq \emptyset$, then 
	\begin{equation} \label{psi10}
		\psi_{1,0}(\varsigma)= \prod_{j \in J}(\mu_{1}\sigma_{1}((u_{j},s_{j})))^{z_{j}} \text{ and } \psi_{0}(\varsigma)= \prod_{j \in J}(\mu\sigma((u_{j},s_{j})))^{z_{j}}.
	\end{equation}
Further we have that
	\begin{align*}
		\hat{\varphi}\psi_{1}(\varsigma+ B_{1}(\mathcal{D}(\mathfrak{s}_{1}^{e})\sqcup F. \mathbf{t}.F, \lambda))&= \hat{\varphi}(\psi_{1,0}(\varsigma))  \tag{by the definition of $\psi_{1}$}\\
		&= \hat{\varphi} \left(\prod_{j \in J}(\mu_{1}\sigma_{1}(u_{j},s_{j})))^{z_{j}}\right)  \tag{by (\ref{psi10})}\\
		&=\prod_{j \in J}(\mu\sigma((u_{j}, s_{j})))^{z_{j}}   \tag{by the definition of $\hat{\varphi}$}\\
		&=\psi_{0}(\varsigma)  \tag{by (\ref{psi10})}\\
		&=\psi(\varsigma+ B_{1}(\mathcal{D}(\mathfrak{s}^{e})\sqcup F. \mathbf{t}.F))  \tag{by the definition of $\psi$},
	\end{align*}
	proving the lemma.
\end{proof}

With the decomposition $H_{1}(\mathcal{D}(\mathfrak{s}_{1}^{e})\sqcup F. \mathbf{t}.F)=\oplus_{g_{i} \in G_{1}}H_{1}(\mathcal{D}(\mathfrak{s}_{1}^{e})\sqcup F. \mathbf{t}.F, u_{i})$, consider the following sequence of homomorphisms

$$\scalebox{0.8}{\xymatrix{\oplus_{g_{i} \in G_{1}}H_{1}(\mathcal{D}(\mathfrak{s}_{1}^{e})\sqcup F. \mathbf{t}.F, u_{i}) \ar[r]^{\oplus \tilde{\theta}_{u_{i}}} & \oplus_{g_{i} \in G_{1}}H_{1}(\mathcal{D}(\mathfrak{s}_{1}^{e})\sqcup F. \mathbf{t}.F, \lambda) \ar[r]^-{\oplus \psi_{1}} & \oplus_{g_{i} \in G_{1}} \pi^{e}_{1} \ar[r]^{\oplus \hat{\varphi}}&  \oplus_{g_{i} \in G_{1}}\pi^{e} \ar[r]^-{\gamma} & H_{1}(\mathcal{D}(\mathfrak{s}^{e}) \sqcup F. \mathbf{t}.F)},}$$
where 
$$\gamma \left(\sum_{g_{i}} d_{g_{i}}\right)= \sum_{g_{i}}\eta(d_{g_{i}}),$$ 
and write for short $\chi= \gamma (\oplus \hat{\varphi}) (\oplus \psi_{1}) (\oplus \tilde{\theta}_{u_{i}})$. 
\begin{lemma} \label{st}
For any element $\varsigma+B_{1}(\mathcal{D}(\mathfrak{s}_{1}^{e})\sqcup F. \mathbf{t}.F) \in H_{1}(\mathcal{D}(\mathfrak{s}_{1}^{e})\sqcup F. \mathbf{t}.F)$, we have $\chi(\varsigma+B_{1}(\mathcal{D}(\mathfrak{s}_{1}^{e})\sqcup F. \mathbf{t}.F))=\varsigma +B_{1}(\mathcal{D}(\mathfrak{s}^{e})\sqcup F. \mathbf{t}.F)$.
\end{lemma}
\begin{proof}
Assume as before that $\varsigma+B_{1}(\mathcal{D}(\mathfrak{s}_{1}^{e})\sqcup F. \mathbf{t}.F)$ is expressed as $$\varsigma+B_{1}(\mathcal{D}(\mathfrak{s}_{1}^{e})\sqcup F. \mathbf{t}.F)=\sum_{g_{i}}(\varsigma_{u_{i}}+B_{1}(\mathcal{D}(\mathfrak{s}_{1}^{e})\sqcup F. \mathbf{t}.F, u_{i})).$$
We these notations we have
\begin{align*}
\allowdisplaybreaks
\chi(\varsigma+B_{1}(\mathcal{D}(\mathfrak{s}_{1}^{e})\sqcup F. \mathbf{t}.F))&=(\gamma (\oplus \hat{\varphi}) (\oplus \psi_{1}) (\oplus \tilde{\theta}_{u_{i}}))\left(\sum_{g_{i}}(\varsigma_{u_{i}}+B_{1}(\mathcal{D}(\mathfrak{s}_{1}^{e})\sqcup F. \mathbf{t}.F, u_{i}))\right)\\
&=(\gamma (\oplus \hat{\varphi}) (\oplus \psi_{1}) )\left(\sum_{g_{i}}(\varsigma_{u_{i}}\cdot u_{i}^{-1}+B_{1}(\mathcal{D}(\mathfrak{s}_{1}^{e})\sqcup F. \mathbf{t}.F, \lambda))\right)\\
&=(\gamma (\oplus \hat{\varphi}))\left(\sum_{g_{i}}\psi_{1}(\varsigma_{u_{i}}\cdot u_{i}^{-1}+B_{1}(\mathcal{D}(\mathfrak{s}_{1}^{e})\sqcup F. \mathbf{t}.F, \lambda))\right)\\
&=\gamma \left(\sum_{g_{i}}\hat{\varphi}\psi_{1}(\varsigma_{u_{i}}\cdot u_{i}^{-1}+B_{1}(\mathcal{D}(\mathfrak{s}_{1}^{e})\sqcup F. \mathbf{t}.F, \lambda))\right)\\
&=\gamma \left(\sum_{g_{i}}\psi(\varsigma_{u_{i}}\cdot u_{i}^{-1}+B_{1}(\mathcal{D}(\mathfrak{s}^{e})\sqcup F. \mathbf{t}.F))\right)   \tag{by Lemma \ref{inv}} \\
&=\sum_{g_{i}}\eta \psi(\varsigma_{u_{i}}\cdot u_{i}^{-1}+B_{1}(\mathcal{D}(\mathfrak{s}^{e})\sqcup F. \mathbf{t}.F))\\
&=\sum_{g_{i}} \varsigma_{u_{i}}\cdot u_{i}^{-1}+B_{1}(\mathcal{D}(\mathfrak{s}^{e})\sqcup F. \mathbf{t}.F) \\
&=\sum_{g_{i}} \varsigma_{u_{i}} +B_{1}(\mathcal{D}(\mathfrak{s}^{e})\sqcup F. \mathbf{t}.F)  \tag{by Lemma \ref{Jm}}\\
&=\varsigma +B_{1}(\mathcal{D}(\mathfrak{s}^{e})\sqcup F. \mathbf{t}.F).
\end{align*}
\end{proof}

\begin{theorem}
	The subpresentation $\cP_{1}$ is aspherical.
\end{theorem}
\begin{proof}
	We prove first that $Im(\hat{\partial}_{2})/  B_{1}(\mathcal{D}(\mathfrak{s}^{e}), \mathcal{D}(\mathfrak{s}_{1}^{e}))$ is infinite cyclic. If we assume the contrary, then there is $z \in \mathbb{Z}^{\ast}$ such that $z\hat{\partial}_{2}(\widehat{q}_{0}) \in B_{1}(\mathcal{D}(\mathfrak{s}^{e}), \mathcal{D}(\mathfrak{s}_{1}^{e}))$, which implies in turn that there is a 2-chain $\xi \in C_{2}(\mathcal{D}(\mathfrak{s}^{e}))$  such that 
	$$z \tilde{\partial}_{2}(\widehat{q}_{0})+C_{1}(\mathcal{D}(\mathfrak{s}^{e}_{1}))=  \tilde{\partial}_{2} (\xi)+C_{1}(\mathcal{D}(\mathfrak{s}^{e}_{1})).$$
	It follows that $\varsigma=z \tilde{\partial}_{2}(\widehat{q}_{0})-\tilde{\partial}_{2} (\xi)$ is a 1-cycle in $C_{1}(\mathcal{D}(\mathfrak{s}^{e}_{1}))$ and therefore $\varsigma \in Z_{1}(\mathcal{D}(\mathfrak{s}_{1}^{e})\sqcup F. \mathbf{t}.F)$. Now we see that
	\begin{align*}
		\chi (\varsigma+ B_{1}(\mathcal{D}(\mathfrak{s}_{1}^{e})\sqcup F. \mathbf{t}.F))&= \varsigma+B_{1}(\mathcal{D}(\mathfrak{s}^{e})\sqcup F. \mathbf{t}.F) \tag{by Lemma \ref{st}}\\
		&=(z \tilde{\partial}_{2}(\widehat{q}_{0})-\tilde{\partial}_{2} (\xi))+B_{1}(\mathcal{D}(\mathfrak{s}^{e})\sqcup F. \mathbf{t}.F)\\
		&=z \tilde{\partial}_{2}( \widehat{q}_{0})+ B_{1}(\mathcal{D}(\mathfrak{s}^{e})\sqcup F. \mathbf{t}.F)\\
		&=z q(r_{0},\lambda)+B_{1}(\mathcal{D}(\mathfrak{s}^{e})\sqcup F. \mathbf{t}.F).
	\end{align*}
	But from Proposition \ref{quasi} it follows that
	$$((\oplus \hat{\varphi}) (\oplus \psi_{1}) (\oplus \theta_{u_{i}}))(\varsigma+ B_{1}(\mathcal{D}(\mathfrak{s}_{1}^{e})\sqcup F. \mathbf{t}.F))=\sum_{g_{i}}v_{g_{i}},$$
	where $v_{g_{i}}\in \hat{\mathfrak{A}}_{1}$, say $v_{g_{i}}=(\mu \sigma ((u_{i}, r_{i}), (u_{i},r^{-1}_{i})))^{z_{i}}$ where $r_{i} \in \mathbf{r}_{1}$, $u_{i} \in F$ and $z_{i} \in \mathbb{Z}$. Now we have
	\begin{align*}
		\chi (\varsigma+ B_{1}(\mathcal{D}(\mathfrak{s}_{1}^{e})\sqcup F. \mathbf{t}.F))&=\gamma \left(\sum_{g_{i}}v_{g_{i}}\right)\\
		&= \gamma \left(\sum_{g_{i}}(\mu \sigma ((u_{i}, r_{i}), (u_{i},r^{-1}_{i})))^{z_{i}}\right)\\
		&= \sum_{g_{i}} \eta ((\mu \sigma ((u_{i}, r_{i}), (u_{i},r^{-1}_{i})))^{z_{i}})\\
		&= \sum_{g_{i}} z_{i}(q(r_{i},u_{i}) + B_{1}(\mathcal{D}(\mathfrak{s}^{e})\sqcup F. \mathbf{t}.F) ) && \text{(Proposition \ref{th})}.
	\end{align*}
	Recollecting, we have in $H_{1}(\mathcal{D}(\mathfrak{s}^{e})\sqcup F. \mathbf{t}.F)$ the equality
	$$z q(r_{0},\lambda)+ B_{1}(\mathcal{D}(\mathfrak{s}^{e})\sqcup F. \mathbf{t}.F)=\sum_{g_{i}} z_{i} q(r_{i},u_{i}) + B_{1}(\mathcal{D}(\mathfrak{s}^{e})\sqcup F. \mathbf{t}.F) .$$
	In $\pi^{e}=\hat{\mathfrak{U}}$ this equality translates to
	$$(\mu \sigma ((1,r_{0}), (1,r_{0}^{-1})))^{z}= \prod_{g_{i}} (\mu \sigma ((u_{i}, r_{i}), (u_{i},r^{-1}_{i})))^{z_{i}}$$
	which from Proposition \ref{free} is impossible since each $r_{i} \neq r_{0}$. So it remains that $Im(\hat{\partial}_{2})/  B_{1}(\mathcal{D}(\mathfrak{s}^{e}), \mathcal{D}(\mathfrak{s}_{1}^{e}))$ is infinite cyclic, and as a result it is isomorphic to $\mathbb{Z}G.\widehat{q}_{0}.  \mathbb{Z}G$ where the isomorphism sends $\widehat{q}_{0}$ to 
	$\hat{\partial}_{2}(\widehat{q}_{0})+B_{1}(\mathcal{D}(\mathfrak{s}^{e}), \mathcal{D}(\mathfrak{s}_{1}^{e}))$ which is the free generator of $Im(\hat{\partial}_{2})/  B_{1}(\mathcal{D}(\mathfrak{s}^{e}), \mathcal{D}(\mathfrak{s}_{1}^{e}))$. But this map is the map $\nu$ of (\ref{mainses}), therefore $H_{2}((\mathcal{D}(\mathfrak{s}^{e}), \mathbf{p}), (\mathcal{D}(\mathfrak{s}_{1}^{e}), \mathbf{p}_{1}))=0$ as desired. 
\end{proof}

\end{document}